\documentclass[11pt,a4paper]{article}
\usepackage[dvips]{graphicx}
\usepackage{amssymb,amsmath,amsthm}
\usepackage{enumerate}
\usepackage{ascmac}
\usepackage[dvips]{graphicx}
\usepackage{float}
\usepackage{wrapfig}
\usepackage{fancybox}
\usepackage[usenames]{color}

\newtheorem{df}{Definition}[section]
\newtheorem{thm}{Theorem}[section]

\newtheorem{lem}{Lemma}[section]
\newtheorem{pro}{Proposition}[section]
\newtheorem{rem}{Remark}[section]

\newcommand{\dis}{\displaystyle}
\newcommand{\R}{{\Bbb R}}
\newcommand{\N}{{\Bbb N}}
\newcommand{\pa}{\partial}

\newcommand{\loc}{\text{loc}}
\newcommand{\h}{\hspace{0.5mm}}
\newcommand{\hh}{\hspace{5mm}}
\newcommand{\hhh}{\hspace{7.5mm}}

\def\cad#1{\csname #1\endcsname }

\title{Non self-similar blow-up solutions to the heat equation
with nonlinear boundary conditions}
\author{Junichi Harada}
\date{}

\begin{document}
\maketitle

%%%%%%%%%%%%%%%%%%%%%%%%%%%%%%%%%%%%%%%%%%%%%%%%%%%%%%%%%%%
\begin{abstract}
This paper is concerned with finite blow-up solutions of
the heat equation with nonlinear boundary conditions.
It is known that
a rate of blow-up solutions is the same as the self-similar rate
for a Sobolev subcritical case.
A goal of this paper is to construct a blow-up solution
whose blow-up rate is different from the self-similar rate
for a Sobolev supercritical case.
\end{abstract}
%%%%%%%%%%%%%%%%%%%%%%%%%%%%%%%%%%%%%%%%%%%%%%%%%%%%%%%%%%%


%%%%%%%%%%%%%%%%%%%%%%%%%%%%%%%%%%%%%%%%%%%%%%%%%%%%%%%%%%%
\section{Introduction}
%%%%%%%%%%%%%%%%%%%%%%%%%%%%%%%%%%%%%%%%%%%%%%%%%%%%%%%%%%%

We study positive solutions of the heat equation with nonlinear boundary conditions:
\begin{equation}\label{u-eq}
\begin{cases}
\dis
 \pa_tu = \Delta u,
& (x,t)\in\R_+^n\times(0,T),
\\ \dis
 \pa_{\nu}u=u^q,
& (x,t)\in\pa\R_+^n\times(0,T),
\\ \dis
 u(x,0) = u_0(x),
& x\in\R_+^n,
\end{cases}
\end{equation}
where $\R_+^n=\{x\in\R^n;x_n>0\}$, $\pa_{\nu}=-\pa/\pa x_n$, $q>1$ and
\[
 u_0\in C(\overline{\R_+^n})\cap L^\infty(\R_+^n),
\hh u_0(x)\geq0.
\]
It is well known that \eqref{u-eq} admits a unique local classical solution
$u(x,t)\in BC(\overline{\R_+^n}\times[0,\tau))\cap C^{2,1}(\overline{\R_+^n}\times(0,\tau))$
for small $\tau>0$,
where $BC(\Omega)=C(\Omega)\cap L^\infty(\Omega)$.
However
by the presence of nonlinearity $u^q$ on the boundary,
a solution $u(x,t)$ may blow up in a finite time $T>0$, namely
\[
 \limsup_{t\to T}\|u(t)\|_{L^{\infty}(\R_+^n)} = \infty.
\]
In fact,
a solution of \eqref{u-eq} actually blows up in a finite time
under some conditions on the initial data
(e.g. \cite{Deng-F-L}, \cite{Fila-Q}, \cite{Ishige-K}).
In this paper,
we are concerned with the asymptotic behavior of blow-up solutions of \eqref{u-eq}.
Let $q_S=n/(n-2)$ if $n\geq3$ and $q_S=\infty$ if $n=1,2$.
For the case $1<q<q_S$,
it is known that
a finite time blow-up solution $u(x,t)$ of \eqref{u-eq} satisfies
\begin{equation}\label{TypeI-eq}
 \sup_{t\in(0,T)}(T-t)^{1/2(q-1)}\|u(t)\|_{L^\infty(\R_+^n)} < \infty,
\end{equation}
where $T>0$ is the blow-up time of $u(x,t)$ (\cite{Chlebik-F}, \cite{Quittner-S}).
More precisely,
let $x_0\in\pa\R_+^n$ be the blow-up point of $u(x,t)$,
then the asymptotic behavior of $u(x,t)$ is described by
the backward self-similar blow-up solution (\cite{Chlebik-F2}):
\begin{equation}\label{selfsimilar-eq}
 \lim_{t\to T}
 \sup_{|z|<R(T-t)^{1/2}}\left| (T-t)^{1/2(q-1)}u(x_0+z,t)-\chi\left( z_n/\sqrt{T-t} \right) \right|
\end{equation}
for any $R>0$,
where $\chi(\xi)$ is a unique positive solution of
\[
\begin{cases}
\dis
 \chi'' - \frac{\xi}{2}\chi' - \frac{\chi}{2(q-1)} = 0
& \text{for } \xi>0,
\\ \dis
 \chi' = -\chi^q
& \text{on } \xi=0.
\end{cases}
\]
Following their works,
more precise asymptotic behavior of blow-up solutions were studied in \cite{HaradaBlow-up,HaradaBlow-up2}.
In general,
the estimate \eqref{TypeI-eq} is known to be important as the first step
to study the asymptotic behavior of blow-up solutions.
Once \eqref{TypeI-eq} is derived,
one may obtain more precise asymptotic behavior such as \eqref{selfsimilar-eq}.
However
it is not yet  known whether \eqref{TypeI-eq} always holds for the case $q\geq q_S$.
An aim of this paper is to show the existence of finite time blow-up solutions of \eqref{u-eq}
which does not satisfy \eqref{TypeI-eq} for some range of $q\geq q_S$.
This kind of non self-similar blow-up phenomenon was already studied in various semilinear parabolic equations.
Particularly,
this paper is motivated by \cite{Herrero-V,Herrero-V2}.
In that paper,
they studied finite time blow-up solutions of
\begin{equation}\label{Fujita-eq}
 u_t = \Delta u + u^p,
\hh (x,t)\in\R^n\times(0,T).
\end{equation}
There are vast papers devoting finite time blow-up solutions of \eqref{Fujita-eq}
and their asymptotic behavior.
Let
\[
 p_S =
 \begin{cases}
 \infty & \text{if } n=1,2,\\ \dis
 \frac{n+2}{n-2} & \text{if } n\geq3,
 \end{cases}
\hspace{10mm}
 p_{JL} =
 \begin{cases}
 \dis
 \infty & \text{if } n\leq10,\\ \dis
 \frac{(n-2)^2-4n+8\sqrt{n-1}}{(n-2)(n-10)} & \text{if } n\geq11.
 \end{cases}
\]
As for the blow-up rate,
it was shown in \cite{Giga-M-S,Giga-K} that
if $1<p<p_S$, every finite blow-up solution of \eqref{Fujita-eq} satisfies
\begin{equation}\label{FujitaTypeI-eq}
 \sup_{t\in(0,T)}(T-t)^{1/(p-1)}\|u(t)\|_{L^\infty(\R^n)}
<
 \infty.
\end{equation}
This estimate is corresponding to \eqref{TypeI-eq},
which is called type I blow-up.
However
\eqref{FujitaTypeI-eq} does not hold in general for $p\geq p_S$.
In fact,
Herrero and Vel\'azuquez (\cite{Herrero-V,Herrero-V2}) constructed finite time blow-up
solutions satisfying
\begin{equation*}
 \sup_{t\in(0,T)}(T-t)^{1/(p-1)}\|u(t)\|_{L^\infty(\R^n)} = \infty
\end{equation*}
for $q>q_{JL}$ (see also \cite{Mizoguchi}).
This blow-up is called type II.
They also gave the exact blow-up rate for type II blow-up solutions constructed in that paper.
Their method relies on the matched asymptotic expansion technique.
However
this technique includes a formal argument,
it is justified by Brouwer's fixed point type theorem with tough pointwise a priori estimates.
This technique is known to be a strong tool to study the non self-similar phenomena
in semilinear parabolic equations.

In this paper,
following their arguments,
we will construct non self-similar blow-up solutions of \eqref{u-eq}
which does not satisfy \eqref{TypeI-eq}.

%%%%%%%%%%%%%%%%%%%%%%%%%%%%%%%%%%%%%%%%%%%%%%%%%%%%%%%%%%%
\begin{thm}\label{1-thm}
Let $q$ be JL-supercritical $(${\rm see Definition \ref{HJL-df}}$)$.
Then there exists a positive $x_n$-axial symmetric initial data $u_0(x)\in BC(\overline{\R_+^n})$
such that
a solution $u(x,t)$ of {\rm\eqref{u-eq}} with the initial data $u_0(x)$ blows up in a finite time $T>0$
and satisfies
\begin{equation}\label{TypeII-eq}
 \sup_{t\in(0,T)}(T-t)^{1/2(q-1)}\|u(t)\|_{L^\infty(\R_+^n)} = \infty.
\end{equation}
\end{thm}
%%%%%%%%%%%%%%%%%%%%%%%%%%%%%%%%%%%%%%%%%%%%%%%%%%%%%%%%%%%

%%%%%%%%%%%%%%%%%%%%%%%%%%%%%%%%%%%%%%%%%%%%%%%%%%%%%%%%%%%
\begin{rem}
As far as the author knows,
this paper seems to be the first one
which treats non self-similar blow-up solutions in a non radial setting.
However
we will see that
our argument is reduced to a radial case in the matching process.
\end{rem}
%%%%%%%%%%%%%%%%%%%%%%%%%%%%%%%%%%%%%%%%%%%%%%%%%%%%%%%%%%%

Our idea of the proof is almost same as that of \cite{Herrero-V2}.
To study finite time blow-up solutions,
we first introduce the self-similar variables as usual.
\[
 \varphi(y,s) =
 (T-t)^{1/2(q-1)}u\left( (T-t)^{1/2}x,t \right),
\hspace{7.5mm} T-t=e^{-s}.
\]
Then
we will construct a solution which converges to the singular stationary solution $U_\infty(y)$
in the self-similar variables.
Since $U_\infty(0)=\infty$,
this solution gives the desired non self-similar blow-up solution
satisfying $\|\varphi(s)\|_\infty\to\infty$ as $s\to\infty$,
which is equivalent to \eqref{TypeII-eq}.
To do that,
we linearize the rescaled equation around the singular stationary solution $U_\infty(y)$
and
construct a solution which behaves as
\begin{equation}\label{expdecay-eq}
 \varphi(y,s) \sim U_\infty(y) + ce^{-\lambda_\ell s}\phi_\ell(y),
\end{equation}
where $\phi_\ell(y)$ and $\lambda_\ell$ are the $\ell$-th eigenfunction
and the $\ell$-th eigenvalue of the linearized operator.
However
since $U_\infty(0)=0$ and $|\phi_\ell(0)|=\infty$,
this does not give the desired blow-up solution.
To overcome this difficulty,
following the idea in \cite{Herrero-V2},
we assume a solution behaves a different way
in an inner region $0<|y|<R(s)$ and an outer region $|y|>R(s)$ ($R(s)\to0$ as $s\to0$).
In fact,
we
will see that a solutions behaves as \eqref{expdecay-eq} in the outer region,
however 
it is described in a different way in the inner region.
In this argument,
the matching process at $|y|=R(s)$ plays a crucial role.
Finally
we justify this formal argument by Brouwer's type fixed point theorem. 

In a non radial setting,
to obtain a large time decay estimate,
we can not apply the method in \cite{Herrero-V2}.
Here
we improve their argument
by combining the $L^2$-decay of solutions and the $L^\infty$-$L^2$ estimate for
the linearized equation.
Furthermore
we provide the fundamental solution of the heat equation with a singular boundary condition
and establish its upper and lower bound.
This equation is strongly related the heat equation with an inverse-square potential.

The rest of this paper is organized as follows.
In Section \ref{Notations-sec},
we collect notations for convenience.
In Section \ref{Preliminaries-sec},
we recall a singular stationary solution and regular stationary solutions. 
Furthermore
we provide a definition of a JL-critical exponent.
Section \ref{Linearizedproblems-sec}
is devoted to the study of linearized problem in the self-similar variables.
Here
the fundamental solution of the heat equation with a singular boundary condition
and
its Duhamel's principle are discussed.
In Section \ref{Functionalsetting-sec},
we introduce a suitable functional space and give fundamental properties of solutions
for a later argument.
Finally
Section \ref{Shorttime-sec}\hspace{0.5mm}-\hspace{0.5mm}Section \ref{Exterior-sec} provides 
a priori estimates to apply a fixed point theorem.
This part is a key and the most heavy part in this paper.
In Appendix,
we collect fundamental properties of eigenfunctions of the linearized problem.

%%%%%%%%%%%%%%%%%%%%%%%%%%%%%%%%%%%%%%%%%%%%%%%%%%%%%%%%%%%
%%%%%%%%%%%%%%%%%%%%%%%%%%%%%%%%%%%%%%%%%%%%%%%%%%%%%%%%%%%
\section{Notations}\label{Notations-sec}
%%%%%%%%%%%%%%%%%%%%%%%%%%%%%%%%%%%%%%%%%%%%%%%%%%%%%%%%%%%
%%%%%%%%%%%%%%%%%%%%%%%%%%%%%%%%%%%%%%%%%%%%%%%%%%%%%%%%%%%

%%%%%%%%%%%%%%%%%%%%%%%%%%%%%%%%%%%%%%%%%%%%%%%%%%%%%%%%%%%
\begin{df}\label{Hx_naxial-df}
A function $u(x)$ defined on $\R_+^n$ is called a $x_n$-axial symmetric function,
if $u(x)$ is expressed by $u(x)=U(|x'|,x_n)$ for some function
$U$ defined on $\R_+\times\R_+$.
\end{df}
%%%%%%%%%%%%%%%%%%%%%%%%%%%%%%%%%%%%%%%%%%%%%%%%%%%%%%%%%%%

Throughout this paper,
solutions are always assumed to be $x_n$-axial symmetric functions.
For $x_n$-axial symmetric functions,
we use the polar coordinate:
\[
 r=\sqrt{|x'|^2+x_n^2},
\hspace{5mm}
 \theta=\arctan\left( \frac{|x'|}{x_n} \right).
\]
Let
$S_+^{n-1}=\{\omega=(\omega',\cos\theta);
 \omega'\in\R^{n-1},\ \theta\in[0,\pi/2),\ |\omega'|^2+(\cos\theta)^2=1\}$
be a half unit sphere.
A function $\xi(\omega)$ on $S_+^{n-1}$ is called a $x_n$-axial symmetric function
if $\xi(\omega)$ depends only on $\theta$.
We denote by $\Delta_S$
the Laplace-Beltrami operator on $S^{n-1}$.
Then
the Laplace-Beltrami operator $\Delta_S$
is expressed in a local coordinate by
\[
 \Delta_S\xi
=
 \bigl(
 \pa_{\theta\theta}
 +
 (n-2)(\cot\theta)\pa_{\theta}
 \bigr)\xi
\]
for any $x_n$-axial symmetric function $\xi\in C^2(S_+^{n-1})$.
Furthermore
the positive (negative) part of a function $u$
is denoted by
$u_+=\max\{u,0\}$ ($u_-=\max\{-u,0\}$).
From this definition,
it is clear that $u=u_+-u_-$.
Throughout this paper,
we use
\[
 m = 1/(q-1).
\]

%%%%%%%%%%%%%%%%%%%%%%%%%%%%%%%%%%%%%%%%%%%%%%%%%%%%%%%%%%%
%%%%%%%%%%%%%%%%%%%%%%%%%%%%%%%%%%%%%%%%%%%%%%%%%%%%%%%%%%%
\section{Preliminaries}\label{Preliminaries-sec}
%%%%%%%%%%%%%%%%%%%%%%%%%%%%%%%%%%%%%%%%%%%%%%%%%%%%%%%%%%%
%%%%%%%%%%%%%%%%%%%%%%%%%%%%%%%%%%%%%%%%%%%%%%%%%%%%%%%%%%%

%%%%%%%%%%%%%%%%%%%%%%%%%%%%%%%%%%%%%%%%%%%%%%%%%%%%%%%%%%%
\subsection{Singular stationary solutions}
%%%%%%%%%%%%%%%%%%%%%%%%%%%%%%%%%%%%%%%%%%%%%%%%%%%%%%%%%%%

First we introduce a singular solution of the following elliptic problem:
\begin{equation}\label{AA2-eq}
 \Delta U = 0 \hspace{3mm}\text{in } \R_+^n,
\hspace{10mm}
 \pa_\nu U = U^q \hspace{3mm}\text{on } \pa\R_+^n.
\end{equation}
We look for a singular solution which has a special form:
\[
 U_{\infty}(x)=V(\theta)r^{-1/(q-1)}.
\]
Then $V(\theta)$ is a solution of
\begin{equation}\label{HeatJL1-eq}
\begin{cases}
 \Delta_SV
=
 m(n-2-m) V
& \text{in } (0,\pi/2),
\\
 \pa_{\theta}V=V^q
& \text{on } \{\pi/2\}.
\end{cases}
\end{equation}

%%%%%%%%%%%%%%%%%%%%%%%%%%%%%%%%%%%%%%%%%%%%%%%%%%%%%%%%%%%
\begin{lem}[Lemma 9 \cite{Quittner-R}]\label{HexistenceV-lem}
For $q>(n-1)/(n-2)$,
there exists a unique positive solution of \eqref{HeatJL1-eq}.
\end{lem}
%%%%%%%%%%%%%%%%%%%%%%%%%%%%%%%%%%%%%%%%%%%%%%%%%%%%%%%%%%%

Throughout this paper,
we denote by $V(\theta)$ the unique solution of \eqref{HeatJL1-eq} and
by $U_{\infty}(x)=V(\theta)r^{-1/(q-1)}$ a singular solution of \eqref{AA2-eq}.
Furthermore
for simplicity of notations,
we put
\[
 {\cal K} = qV|_{\theta=\pi/2}^{q-1}.
\]

%%%%%%%%%%%%%%%%%%%%%%%%%%%%%%%%%%%%%%%%%%%%%%%%%%%%%%%%%%%
\subsection{JL-critical exponent}
%%%%%%%%%%%%%%%%%%%%%%%%%%%%%%%%%%%%%%%%%%%%%%%%%%%%%%%%%%%

To define a JL-critical exponent,
we first introduce the trace Hardy inequality.

%%%%%%%%%%%%%%%%%%%%%%%%%%%%%%%%%%%%%%%%%%%%%%%%%%%%%%%%%%%
\begin{lem}[Theorem 1.4 \cite{Davila-D-M}]\label{HHardy-lem}
Let $n\geq3$.
We define
\[
 c_H
=
 \inf_{u\in H^1(\R_+^n)}
 \frac{\dis\int_{\R_+^n}|\nabla u|^2dx}
 {\dis\int_{\pa\R_+^n}|x'|^{-1}|u|^2dx'}.
\]
Then $c_H$ is given by
$c_H=2\Gamma\left( n/4 \right)^2\Gamma\left( (n-2)/4 \right)^{-2}$.
\end{lem}
%%%%%%%%%%%%%%%%%%%%%%%%%%%%%%%%%%%%%%%%%%%%%%%%%%%%%%%%%%%

From Lemma \ref{HHardy-lem},
we can define
\[
\label{b(q)-eq}
 \mu(q)
=
 \inf_{u\in H^1(\R_+^n)}
 \frac{\dis\int_{\R_+^n}|\nabla u|^2dx-{\cal K}\int_{\pa\R_+^n}r^{-1}u^2dx'}
 {\dis \int_{\pa\R_+^n}|x'|^{-1}u^2dx'}.
\]
This expression is obtained by
linearizing the equation around the singular stationary solution $U_\infty(x)$.
% where we recall that ${\cal K}r^{-1}=qU_\infty(x)^{q-1}|_{\pa\R_+^n}$ by \eqref{K-eq}.

%%%%%%%%%%%%%%%%%%%%%%%%%%%%%%%%%%%%%%%%%%%%%%%%%%%%%%%%%%%
\begin{df}\label{HJL-df}
A exponent $q$ is called
JL-supercritcal if $\mu(q)>0$, JL-critical if $\mu(q)=0$ and JL-subcritical if $\mu(q)<0$.
\end{df}
%%%%%%%%%%%%%%%%%%%%%%%%%%%%%%%%%%%%%%%%%%%%%%%%%%%%%%%%%%%

%%%%%%%%%%%%%%%%%%%%%%%%%%%%%%%%%%%%%%%%%%%%%%%%%%%%%%%%%%%
\begin{rem}\label{JL-rem}
By the explicit expression of $U_{\infty}(x)=V(\theta)r^{-1/(q-1)}$
and Lemma {\rm\ref{HHardy-lem}},
we see that
$\mu(q)>0$ is equivalent to ${\cal K}<c_H$.
Hence
an exponent $q$ is
JL-supercritcal if ${\cal K}<c_H$, JL-critical if ${\cal K}=c_H$ and JL-subcritical if ${\cal K}>c_H$.
\end{rem}
%%%%%%%%%%%%%%%%%%%%%%%%%%%%%%%%%%%%%%%%%%%%%%%%%%%%%%%%%%%

Unfortunately
we do not know the explicit expression of a JL-critical exponent.
However
we find that
$q$ is JL-subcritical if $q$ is close to $n/(n-2)$
and
$q$ is JL-supercritical if $q$ and $n$ are large enough.

%%%%%%%%%%%%%%%%%%%%%%%%%%%%%%%%%%%%%%%%%%%%%%%%%%%%%%%%%%%
\begin{lem}[Lemma 4.1\h-\h{Lemma 4.2} in \cite{Harada}]
For $n\geq3$ there exists $q_0>n/(n-2)$ such that
$\mu(q)<0$ if $n/(n-2)<q<q_0$.
Moreover
there exists $n_0\in\N$ such that for $n\geq n_0$
there exists $q_1>q_0$ such that $\mu(q)>0$ if $q>q_1$.
\end{lem}
%%%%%%%%%%%%%%%%%%%%%%%%%%%%%%%%%%%%%%%%%%%%%%%%%%%%%%%%%%%

%%%%%%%%%%%%%%%%%%%%%%%%%%%%%%%%%%%%%%%%%%%%%%%%%%%%%%%%%%%
\subsection{Regular stationary solutions}
%%%%%%%%%%%%%%%%%%%%%%%%%%%%%%%%%%%%%%%%%%%%%%%%%%%%%%%%%%%

In this subsection,
we collect the qualitative property of
positive $x_n$-axial symmetric solutions of \eqref{AA2-eq} obtained in \cite{Harada}.
Let $\kappa_i$, $e_i(\theta)$ be the $i$-th eigenvalue,
the $i$-th eigenfunction with $\|e_i\|_{L^2(S_+^{n-1})}=1$ of
\begin{equation}\label{Eigen-eq}
\begin{cases}
\dis
 -\Delta_S e = \kappa e
& \text{in } (0,\pi/2),
\\ \dis
 \pa_{\theta}e = {\cal K}e
& \text{on } \{\pi/2\}.
\end{cases}
\end{equation}
Then
from Lemma 6.8\h-\h{Lemma} 6.10 in \cite{Harada},
the first and the second eigenvalues are estimated as follows.

%%%%%%%%%%%%%%%%%%%%%%%%%%%%%%%%%%%%%%%%%%%%%%%%%%%%%%%%%%%
\begin{lem}\label{Heigenvalue-lem}
The first eigenvalue $\kappa_1$ is estimated as follows\hspace{0.5mm}$:$
\[
\begin{array}{cl}
\dis
 \kappa_1>-(n-2)^2/4
& \mathrm{if}\
 q\ \mathrm{is\ JL}\text{-}\mathrm{supercritical}, 
\\[1mm] \dis
 \kappa_1=-(n-2)^2/4
& \mathrm{if}\
 q\ \mathrm{is\ JL}\text{-}\mathrm{critical},
\\[1mm] \dis
 \kappa_1<-(n-2)^2/4
& \mathrm{if}\
 q\ \mathrm{is\ JL}\text{-}\mathrm{subcritical}.
\end{array}
\]
Moreover
the second eigenvalue $\kappa_2$ is always positive.
\end{lem}
%%%%%%%%%%%%%%%%%%%%%%%%%%%%%%%%%%%%%%%%%%%%%%%%%%%%%%%%%%%

From Lemma \ref{Heigenvalue-lem},
it is easily seen that
\[
 \mu^2-(n-2-2m)\mu-\{m(n-2-m)+\kappa_1\}=0
\]
has two real roots if and only if $q$ is JL-supercritical.
We denote the small positive root by $\mu_1$,
which is written by
\begin{equation}\label{mu1-eq}
 \mu_1 = \frac{(n-2-2m)-\sqrt{(n-2-2m)^2+4\{m(n-2-m)+\kappa_1\}}}{2}.
\end{equation}

%%%%%%%%%%%%%%%%%%%%%%%%%%%%%%%%%%%%%%%%%%%%%%%%%%%%%%%%%%%
\begin{thm}[Theorem 1.1 \cite{Harada}]
\label{JLsuper-thm}
Let $q$ be JL-supercritical or JL-critical.
Then there exists a family of positive $x_n$-axial symmetric solutions
$\{U_{\alpha}(x)\}_{\alpha>0}$
$(U_{\alpha}(0)=\alpha)$ of \eqref{AA2-eq}
satisfying the following properties.
\\[1mm]
$({\rm i})$
 $U_{\alpha}(x)=\alpha U_1(\alpha^{q-1}x)\leq U_{\infty}(x)$,
\hspace{10mm}
$({\rm ii})$
 $U_{\alpha_1}(x)<U_{\alpha_2}(x)$
 if $\alpha_1<\alpha_2$,
\\[1mm]
$({\rm iii})$
 $\lim_{\alpha\to\infty}U_{\alpha}(x)=U_{\infty}(x)$ {\rm for}
 $x\in\overline{\R_+^n}\setminus\{0\}$,
\\[1mm]
$({\rm iv})$
there exists $\epsilon>0$ such that for any $\alpha>0$
there exist $k_\alpha=\alpha^{-\mu_1/m}k_1>0$ and $k_\alpha'\in\R$
such that
the following asymptotic expansion holds
for large $r>0$
\[
 U_{\alpha}(x)
=
 U_{\infty}(x)
+
 \begin{cases}
 (-k_\alpha+O(r^{-\epsilon}))e_1(\theta)r^{-(m+\mu_1)}
 & \text{{\rm if JL-supercritical}},
 \\
 (-k_\alpha\log r+k_\alpha'+o(r^{-\epsilon}))e_1(\theta)r^{-(n-2)/2}
 & \text{{\rm if JL-critical}},
 \end{cases}
\]
where
the polar coordinate $r=|x|$, $\tan\theta=|x'|/x_n$ is used
and
the asymptotic expansion holds uniformly for $\theta\in(0,\pi/2)$.
\end{thm}
%%%%%%%%%%%%%%%%%%%%%%%%%%%%%%%%%%%%%%%%%%%%%%%%%%%%%%%%%%%

%%%%%%%%%%%%%%%%%%%%%%%%%%%%%%%%%%%%%%%%%%%%%%%%%%%%%%%%%%%
%%%%%%%%%%%%%%%%%%%%%%%%%%%%%%%%%%%%%%%%%%%%%%%%%%%%%%%%%%%
\section{Linearized problems around the singular solution}
\label{Linearizedproblems-sec}
%%%%%%%%%%%%%%%%%%%%%%%%%%%%%%%%%%%%%%%%%%%%%%%%%%%%%%%%%%%
%%%%%%%%%%%%%%%%%%%%%%%%%%%%%%%%%%%%%%%%%%%%%%%%%%%%%%%%%%%

To study blow-up solutions,
we introduce self-similar variables and a rescaled solution:
\[
 \varphi(y,s) = (T-t)^{1/2(q-1)}u((T-t)^{1/2}y,t),
\hh T-t=e^{-s}.
\]
Put $s_T=-\log T$.
Then $\varphi(y,s)$ satisfies
\begin{equation}\label{varphi-eq}
\begin{cases}
\dis
 \varphi_s = \Delta \varphi -\frac{y}{2}\cdot\nabla\varphi-\frac{m}{2}\varphi,
& (y,s)\in\R_+^n\times(s_T,\infty),\\
 \pa_\nu\varphi = \varphi^q, & (y,s)\in\pa\R_+^n\times(s_T,\infty).
\end{cases}
\end{equation}
Here we put
\[
 \Phi(y,s) = \varphi(y,s) - U_\infty(y).
\]
Then
since $qU_\infty(x)^{q-1}|_{\pa\R_+^n}={\cal K}r^{-1}$,
it is easily seen that $\Phi(y,s)$ solves
\begin{equation}\label{Phi-eq}
\begin{cases}
\dis
 \Phi_s = \Delta\Phi - \frac{y}{2}\cdot\nabla\Phi-\frac{m}{2}\Phi,
& (y,s)\in\R_+^n\times(s_T,\infty),\\[3mm]
\dis
 \pa_\nu\Phi = {\cal K}r^{-1}\Phi + f(\Phi), & (y,s)\in\pa\R_+^n\times(s_T,\infty), 
\end{cases}
\end{equation}
where
\[
 f(\Phi) = (\Phi+U_\infty)^q - U_\infty^q - {\cal K}\Phi.
\]
Now
we define weighted Lebesgue spaces and Sobolev spaces:
\[
\begin{array}{c}
\dis
 L_\rho^p =
 \left\{
 v\in L_{\loc}^p(\R_+^n);\int_{\R_+^n}|v(y)|^p\rho(y)dy<\infty
 \right\},
\\[6mm] \dis
 H_{\rho}^k =
 \left\{
 v\in L_{\rho}^2; D^{\alpha}v\in L_{\rho}^2
 \text{ for any } \alpha=(\alpha_1,\cdots,\alpha_n)
 \text{ satisfying } |\alpha|\leq k
 \right\},
\\[2mm] \dis
L_{\rho}^p(\pa\R_+^n) =
 \left\{
 v\in L_{\loc}^p(\pa\R_+^n);\int_{\pa\R_+^n}|v(y')|^p\rho(y')dy'<\infty
 \right\},
\end{array}
\]
where a wight function $\rho(y)$ is given by
\[
 \rho(y) = e^{-|y|^2/4}.
\]
The norms are given by
\[
\begin{array}{c}
\dis
 \|v\|_{L_{\rho}^p}^p = \int_{\R_+^n} |v(y)|^p\rho(y)dy,
\hspace{7.5mm}
 \|v\|_{H_{\rho}^k}^2 =
 \sum_{|\alpha|\leq k}\|D^{\alpha}v\|_{L_{\rho}^2}^2,
\\ \dis
 \|v\|_{L_{\rho}^p(\pa\R_+^n)}^p = \int_{\pa\R_+^n} |v(y')|^p\rho(y')dy'
\end{array}
\]
and the inner product on $L_{\rho}^2$ is naturally defined by
\[
 (v_1,v_2)_{\rho} = \int_{\R_+^n}v_1(y)v_2(y)\rho(y)dy.
\]
For simplicity,
we set
\[
 \|\cdot\|_\rho=\|\cdot\|_{L_{\rho}^2}.
\]
Let $H_\rho^*$ be the dual space of $H_\rho^1$.
To study the asymptotic behavior of $\Phi(y,s)$,
we define a linear operator $A$: $D(A)\to H_\rho^*$ with $D(A)=H_\rho^1$ by
\[
 {}_{H_\rho^*}\langle A\Phi,\eta \rangle_{H_\rho^1}
=
 -(\nabla \Phi,\nabla\eta)_\rho-\frac{m}{2}(\Phi,\eta)_\rho+
 {\cal K}\int_{\pa\R_+^n}r^{-1}\Phi\eta\rho\h dy'.
\]

%%%%%%%%%%%%%%%%%%%%%%%%%%%%%%%%%%%%%%%%%%%%%%%%%%%%%%%%%%%
\begin{lem}\label{OperatorA-eq}
Let $q$ be JL-supercritical.
Then an operator $(-A+\mu)${\h}$:$ $D(A)\to H_\rho^*$ has compact inverse on $L_\rho^2$
for large $\mu>0$.
Moreover its inverse is self-adjoint on $L_\rho^2$.
\end{lem}
%%%%%%%%%%%%%%%%%%%%%%%%%%%%%%%%%%%%%%%%%%%%%%%%%%%%%%%%%%%

%%%%%%%%%%%%%%%%%%%%%%%%%%%%%%%%%%%%%%%%%%%%%%%%%%%%%%%%%%%
\begin{proof}
First
we claim that there exist $\alpha_0,\mu>0$ such that
\begin{equation}\label{Monotone-eq}
 ((-A+\mu)\Phi,\Phi)_\rho
\geq
 \alpha_0\left( \|\nabla\Phi\|_\rho^2+\|\Phi\|_\rho^2 \right)
\hspace{5mm}\text{for } \Phi\in D(A).
\end{equation}
Let $\theta$ be a cut off function such that $\theta(|y|)=1$ for $|y|<1$ and $\theta(|y|)=0$ for $|y|>2$.
Then
by the trace Hardy inequality,
it holds that
\[
\begin{array}{l}
\dis
 c_H\int_{\pa\R_+^n}r^{-1}\theta^2\Phi^2\rho\h dy'
\leq
 \int_{\R_+^n}|\nabla(\theta \Phi\rho^{1/2})|^2dy
\\[5mm] \dis \hspace{10mm}
=
 \int_{\R_+^n}
 \left(
 \theta^2|\nabla\Phi|^2+
 \frac{1}{2}\nabla\Phi^2\cdot\nabla\theta^2-\frac{y}{4}\cdot\nabla(\theta\Phi)^2+
 \left( \frac{|y|^2}{16}\theta^2+|\nabla\theta|^2 \right)\Phi^2
 \right)\rho\h dy
\\[5mm] \dis \hspace{10mm}
=
 \int_{\R_+^n}
 \left(
 \theta^2|\nabla\Phi|^2-
 \frac{1}{2}\Phi^2\Delta\theta^2+\frac{n}{4}\theta^2\Phi^2+
 \left( \frac{|y|^2}{16}\theta^2+|\nabla\theta|^2 \right)\Phi^2
 \right)\rho\h dy
\\[5mm] \dis \hspace{10mm}
\leq
 \int_{\R_+^n}|\nabla\Phi|^2\rho\h dy
 +
 c_0\int_{\R_+^n}\Phi^2 \rho\h dy.
\end{array}
\]
Therefore
we obtain
\[
\dis
 (-A\Phi,\Phi)_\rho
\geq
 \left( 1-\frac{{\cal K}}{c_H} \right)\|\nabla\Phi\|_\rho^2
-
 c_0\|\Phi\|_\rho
-
 {\cal K}\int_{\pa\R_+^n}r^{-1}(1-\theta^2)\Phi^2\rho\h dy'.
\]
Since ${\cal K}<c_H$,
by using $\int_{\pa\R_+^n}\Phi^2\rho dy'\leq
 \epsilon\int_{\R_+^n}|\nabla\Phi|^2\rho dy+c_1\epsilon^{-1}\int_{\R_+^n}\Phi^2\rho dy$
(see Lemma 3.1 in \cite{HaradaBlow-up2}),
we can assure \eqref{Monotone-eq}.
Therefore 
the operator $(-A+\mu)$ has inverse from $L_\rho^2$ to $H_\rho^1$.
Furthermore
its inverse is clearly self-adjoint on $L_\rho^2$.
Finally
we prove that $(-A+\mu)^{-1}$ is compact on $L_\rho^2$.
Let $f\in L_\rho^2$ and $\Phi=(-A+\mu)^{-1}f$.
Then it holds from \eqref{Monotone-eq} that
\[
 \|\nabla\Phi\|_\rho^2+\|\Phi\|_\rho^2 \leq \frac{4}{\alpha_0^2}\|f\|_\rho^2.
\]
We recall that the embedding $H_\rho^1\to L_\rho^2$ is compact
(see Lemma A.2 in \cite{HaradaBlow-up2}).
Therefore the proof is completed.
\end{proof}
%%%%%%%%%%%%%%%%%%%%%%%%%%%%%%%%%%%%%%%%%%%%%%%%%%%%%%%%%%%

From Lemma \ref{OperatorA-eq},
we find that $L_\rho^2$ is spanned by eigenfunctions of
\begin{eqnarray}\label{EigenFull-eq}
\begin{cases}
\dis
 -\left( \Delta - \frac{y}{2}\cdot\nabla - \frac{m}{2} \right)\phi = \lambda \phi 
& \text{in } \R_+^n,
\\[2mm] \dis
 \pa_{\nu}\phi = {\cal K}r^{-1}\phi
& \text{on } \pa\R_+^n.
\end{cases}
\end{eqnarray}
Since solutions are assumed to be $y_n$-axial symmetric,
every eigenfunction of \eqref{EigenFull-eq} is given by the following separation of variables:
\[
 \phi(y) = e(\theta)a(r).
\]
Let $e_i(\theta)$ and $\kappa_i$ be the $i$-th eigenfunction with $\|e_i\|_{L^2(S_+^{n-1})}=1$ and
the $i$-th eigenvalue of \eqref{Eigen-eq}.
Moreover
let $a_{ij}(r)$ and $\lambda_{ij}$ be the $j$-th eigenfunction with
$\int_0^\infty a_{ij}(r)^2\rho r^{n-1}dr=1$ and the $j$-th eigenvalue of
\[
 -\left(
 a'' + \frac{n-1}{r}a' - \frac{\kappa_i}{r^2}a' - \frac{r}{2}a' - \frac{m}{2}a
 \right)
= \lambda a,
\hh r>0.
\]
Then
all eigenfunctions are expressed by
\[
 \phi_{ij}(y)=e_i(\theta)a_{ij}(r)
\]
and its eigenvalue is given by $\lambda_{ij}$.
The detail is stated in Appendix.

%%%%%%%%%%%%%%%%%%%%%%%%%%%%%%%%%%%%%%%%%%%%%%%%%%%%%%%%%%%
\subsection{Heat kernel and representation formula}
\label{Heatkernel-sec}
%%%%%%%%%%%%%%%%%%%%%%%%%%%%%%%%%%%%%%%%%%%%%%%%%%%%%%%%%%%

In this subsection,
we provide fundamental estimates of the following linear parabolic equations related to \eqref{Phi-eq}
and give a representation formula (the Duhamel principle).
\[
\begin{cases}
 \dis
 \Phi_s = \Delta\Phi-\frac{y}{2}\cdot\nabla\Phi-\frac{m}{2}\Phi,
 & (y,s)\in\R_+^n\times(0,\infty),\\[2mm]
 \pa_\nu\Phi = {\cal K}r^{-1}\Phi,
 & (y,s)\in\pa\R_+^n\times(0,\infty), \\[1mm]
 \Phi(y,0) = \Phi_0(y),
 & y\in\pa\R_+^n.
\end{cases}
\]
Let $\mu_1$ be given in \eqref{mu1-eq} and $e_1(\theta)$ be the first eigenfunction of \eqref{Eigen-eq}.
We set
\[
 \sigma(y) = r^{-\gamma}e_1(\theta),
\hspace{10mm}
 \gamma=m+\mu_1.
\]
Then
it is easily verified that
\[
 m < \gamma < \frac{n-2}{2}.
\]
Furthermore
we find that $\sigma(y)$ solves
\begin{equation}\label{sigma-eq}
\begin{cases}
\dis
 -\Delta \sigma = 0 & \text{in } \R_+^n,
\\ \dis
 \pa_\nu\sigma = {\cal K}r^{-1}\sigma & \text{on } \pa\R_+^n.
\end{cases}
\end{equation}
Now
we introduce a new function $b(y,s)$ by 
\[
 b(y,s) = \Phi(y,s)/\sigma(y).
\]
Then $b(y,s)$ satisfies
\begin{equation}\label{Psi-eq}
\begin{cases}
\dis
 b_s = \Delta b - \frac{y}{2}\cdot\nabla b +
 \frac{2\nabla\sigma}{\sigma}\cdot\nabla b +
 \left( \frac{\gamma-m}{2} \right)b,
 & (y,s)\in\R_+^n\times(0,\infty),
\\[2mm] \dis
 \pa_\nu b = 0,
 & (y,s)\in\pa\R_+^n\times(0,\infty),
\\[2mm] \dis
 b(y,0) = b_0(y) := \Phi_0(y)/\sigma(y),
 & y\in\pa\R_+^n.
\end{cases}
\end{equation}
To show the existence of the heat kernel of \eqref{Psi-eq},
we go back to the original variables $(x,t)$.
\[
 z(x,t) = (1-t)^{(\gamma-m)/2}b((1-t)^{-1/2}x,-\log(1-t)).
\]
Now we put
\[
 {\cal B}(x) = \sigma(x)^2.
\]
Then since $\sigma(y)$ satisfies \eqref{sigma-eq},
it is easily seen that $z(x,t)$ solves
\begin{equation}\label{z-eq}
\begin{cases}
 \dis
 z_t = \frac{1}{{\cal B}}\h\text{div}\left( {\cal B}\nabla z \right),
 & (x,s)\in\R_+^n\times(0,\infty),
 \\ \dis
 \pa_\nu z = 0,
 & (x,s)\in\pa\R_+^n\times(0,\infty),
 \\ \dis
 z(x,0) = z_0(x) := b_0(x),
 & x\in\pa\R_+^n.
\end{cases}
\end{equation}
This equation is also written in another form.
\[
 z_t= 
 \frac{1}{{\cal B}}\h\text{div}\left( {\cal B}\nabla z \right) =
 \Delta z + \left( \frac{2\nabla\sigma}{\sigma} \right)\cdot\nabla z.
\]
Here
we introduce other weighted Lebesgue spaces related to \eqref{z-eq}.
\[
 L_{\cal B}^p =
 \left\{
 z\in L_{\text{loc}}^p(\R_+^n);\int_{\R_+^n}|z(x)|^p{\cal B}(x)dx<\infty
 \right\},
\hspace{7.5mm}
 H_{\cal B}^1 = \{z\in L_{\cal B}^2;\nabla z\in L_{\cal B}^2\}.
\]
Inner products on $L_{\cal B}^2$ and $H_{\cal B}^1$ are defined by
\[
\begin{array}{c}
\dis
 (z_1,z_2)_{\cal B} := \int_{\R_+^n}z_1(x)z_2(x){\cal B}(x)dx,
\\[6mm] \dis
 (z_1,z_2)_{H_{\cal B}^1} :=
 (\nabla z_1,\nabla z_2)_{\cal B}+(z_1,z_2)_{\cal B}.
\end{array}
\]
For simplicity,
we write $\|\cdot\|_{\cal B}=\|\cdot\|_{L_{\cal B}^2}$.
Furthermore
we denote by $H_{\cal B}^*$ be the dual space of $H_{\cal B}^1$.
Let $J$: $H_{\cal B}^1\to H_{\cal B}^*$ be a mapping defined by
${}_{H_{\cal B}^*}\langle Jz,\zeta \rangle_{H_{\cal B}^1}=(z,\zeta)_{H_{\cal B}^1}$ for all $\zeta\in H_{\cal B}^1$.
Then
an inner product of $H_{\cal B}^*$ is defined by
\[
 (z_1,z_2)_{H_{\cal B}^*} := (J^{-1}z_1,J^{-1}z_2)_{H_{\cal B}^1},
\hh \forall z_1,z_2\in H_{\cal B}^*.
\]
In the usual manner,
$z\in L_{\cal B}^2$ can be considered as an element of $H_{\cal B}^*$.
\[
 {}_{H_{\cal B}^*}\langle z,\zeta \rangle_{H_{\cal B}^1} := (z,\zeta)_{\cal B},
\hh \forall\zeta\in H_{\cal B}^1.
\]
We define a linear operator $A_0$: $H_{\cal B}^1\to H_{\cal B}^*$ by
\[
 {}_{H_{\cal B}^*}\langle A_0z,\zeta \rangle_{H_{\cal B}^1} := -(\nabla z,\nabla\zeta)_{\cal B},
\hh \forall \zeta\in H_{\cal B}^1.
\]
Then
a weak solution of \eqref{z-eq} is defined by
\begin{equation}\label{xiweak-eq}
\begin{cases}
\dis
 {}_{H_{\cal B}^*}\langle z_t(t),\zeta \rangle_{H_{\cal B}^1}
=
 {}_{H_{\cal B}^*}\langle A_0z(t),\zeta \rangle_{H_{\cal B}^1}
=
 -(\nabla z(t),\nabla\zeta)_{\cal B},
\hh \forall \zeta\in H_{\cal B}^1,
\\ \dis
 z(0) = z_0.
\end{cases}
\end{equation}
Since
${}_{H_{\cal B}^*}\langle -A_0z,z \rangle_{H_{\cal B}^1}=-\|\nabla z\|_{\cal B}^2$
for all $z\in H_{\cal B}^1$,
it is easily verified that
the operator $A_0$ with $D(A_0)=H_{\cal B}^1$ is self-adjoint on $H_{\cal B}^*$. 
Therefore
we obtain the following results.

%%%%%%%%%%%%%%%%%%%%%%%%%%%%%%%%%%%%%%%%%%%%%%%%%%%%%%%%%%%
\begin{lem}\label{Existence-lem}
Let $z_0\in H_{\cal B}^*$.
Then
there exists a unique weak solution of \eqref{xiweak-eq} satisfying
\[
 z\in C([0,\infty);H_{\cal B}^*)\cap C^1((0,\infty);H_{\cal B}^*)\cap C((0,\infty);H_{\cal B}^1).
\]
In particular,
$g(t)=\|z(t)\|_{\cal B}^2$ is absolutely continuous on $(0,\infty)$ and
$g'(t)=-2\|\nabla z(t)\|_{\cal B}^2$ for a.e. $t\in(0,\infty)$.
Moreover
if $z_0\in H_{\cal B}^1$,
then the solution $z$ is in $C([0,\infty);H_{\cal B})$.
\end{lem}
%%%%%%%%%%%%%%%%%%%%%%%%%%%%%%%%%%%%%%%%%%%%%%%%%%%%%%%%%%%

Next
we provide estimates of time derivatives of solutions constructed in Lemma \ref{Existence-lem}.

%%%%%%%%%%%%%%%%%%%%%%%%%%%%%%%%%%%%%%%%%%%%%%%%%%%%%%%%%%%
\begin{lem}\label{Existence2-lem}
Let $z_0\in L_{\cal B}^2$.
Then
there exists a unique weak solution of \eqref{xiweak-eq} satisfying
\[
\begin{array}{c}
\dis
 z\in C([0,\infty);L_{\cal B}^2)\cap C^1((0,\infty);L_{\cal B}^2)\cap C((0,\infty);H_{\cal B}^1),
\\[2mm]
 t\|z_t(t)\|_{\cal B} + \sqrt{t}\|\nabla z(t)\|_{\cal B} \leq c\|z_0\|_{\cal B}.
\end{array}
\]
Furthermore
if $z_0\in L_{\cal B}^1\cap L_{\cal B}^2$,
then it holds that
\[
 z\in C([0,\infty);L_{\cal B}^1),
\hh\hh
 \|z(t)\|_{L_{\cal B}^1}=\|z_0\|_{L_{\cal B}^1},
\hspace{3mm}\forall t>0.
\]
\end{lem}
%%%%%%%%%%%%%%%%%%%%%%%%%%%%%%%%%%%%%%%%%%%%%%%%%%%%%%%%%%%

\begin{proof}
To obtain solutions satisfying the desired regularity,
we consider approximation problems.
Let $\theta(|x|)$ be a smooth cut off function such that
$\theta(r)=0$ for $r\in(0,1/4)$ and $\theta(r)=1$ for $r>3/4$,
and set $\theta_\epsilon(|x|)=\theta(|x|/\epsilon)$.
Now we define ${\cal B}_\epsilon(x)$ and $\sigma_{\epsilon}(x)$ by
\[
 \sigma_\epsilon(x)
=
 \theta_\epsilon\sigma(x) + (1-\theta_\epsilon)\epsilon^{-\gamma},
\hspace{7.5mm}
 {\cal B}_\epsilon(x)=\sigma_\epsilon(x)^2.
\]
Then
there exist $c_1,c_2,c_3,c_4>0$ such that for $0<|x|<\epsilon$
\[
 c_1\epsilon^{-\gamma} < \sigma_\epsilon(x) < c_2\epsilon^{-\gamma},
\hhh
 |\nabla\sigma_\epsilon(x)|< c_3\epsilon^{-\gamma-1},
\hhh
 |D^2\sigma_\epsilon(x)|< c_4\epsilon^{-\gamma-2}.
\]
Therefore there exists $c>0$ such that for $0<|x|<\epsilon$
\begin{equation}\label{Dsigma-eq}
 \frac{|\nabla\sigma_\epsilon(x)|}{\sigma_\epsilon(x)} < c|x|^{-1},
\hhh
 \frac{|D^2\sigma_\epsilon(x)|}{\sigma_\epsilon(x)} < c|x|^{-2}.
\end{equation}
Now we consider the following approximation equation.
\begin{equation}\label{approximation-eq}
 \begin{cases}
 \dis
 z_t = \frac{1}{{\cal B}_\epsilon}\text{div}\left( {\cal B}_\epsilon\nabla z \right),
 & (x,t)\in\R_+^n\times(0,\infty),
 \\ \dis
 \pa_\nu z = 0
 & (x,t)\in\pa\R_+^n\times(0,\infty),
 \\ \dis
 z(x,0) = z_0(x),
 & x\in\pa\R_+^n.
 \end{cases}
\end{equation}
First
we assume that
\begin{equation}\label{assume-eq}
 z_0\in C_c^\infty(\overline{\R_+^n}),\hhh
 \pa_\nu z_0 =0 \hspace{3mm}\text{on } \pa\R_+^n.
\end{equation}
Then since a right-hand side of \eqref{approximation-eq} is written by
\[
 \frac{1}{{\cal B}_\epsilon}\text{div}\left( {\cal B}_\epsilon\nabla z \right)
=
 \Delta z +\frac{2\nabla\sigma_\epsilon}{\sigma_\epsilon}\cdot\nabla z
\]
and $\nabla\sigma_\epsilon/\sigma_\epsilon$ is smooth,
if $z_0$ satisfies \eqref{assume-eq},
there exists a unique solution $z_\epsilon(x,t)$ such that
$z_\epsilon(x,t)\in C^{2,1}(\overline{\R_+^n}\times[0,\infty))\cap
 C^\infty(\overline{\R_+^n}\times(0,\infty))$
and
\[
 z_\epsilon(t)\in C^1([0,\infty);L_{{\cal B}_\epsilon}^2),
\hhh
 z_\epsilon(t)\in C^k((0,\infty);H_{{\cal B}_\epsilon}^1)
\hspace{3mm} (k\geq1),
\]
where $L_{{\cal B}_\epsilon}^2$ and $H_{{\cal B}_\epsilon}^1$ are defined in the same manner as $L_{\cal B}^2$ and $H_{\cal B}^1$.
By a maximum principle,
it is easily seen that
\begin{equation}\label{zepsiloninfty-eq}
 \|z_\epsilon(t)\|_{\infty} \leq \|z_0\|_\infty,
\hspace{3mm} t>0.
\end{equation}
Put $\|\cdot\|_{{\cal B}_\epsilon}=\|\cdot\|_{L_{{\cal B}_\epsilon}^2}$.
Then
multiplying \eqref{approximation-eq} by
$z_\epsilon\sigma_\epsilon^2$, $\pa_tz_\epsilon\sigma_\epsilon^2$, $t\pa_tz_\epsilon\sigma_\epsilon^2$
and integrating over $\R_+^n\times(0,t)$ respectively,
we get
\begin{equation}\label{appriori1-eq}
\begin{array}{c}
\dis
 \frac{1}{2}\|z_\epsilon(t)\|_{{\cal B}_\epsilon}^2
+
 \int_0^t\|\nabla z_\epsilon(\tau)\|_{{\cal B}_\epsilon}^2d\tau
=
 \frac{1}{2}\|z_0\|_{{\cal B}_\epsilon}^2,
\\[4mm] \dis
 \int_0^t\|\pa_tz_\epsilon(\tau)\|_{{\cal B}_\epsilon}^2d\tau
+
 \frac{1}{2}\|\nabla z_\epsilon(t)\|_{{\cal B}_\epsilon}^2
=
 \frac{1}{2}\|\nabla z_0\|_{{\cal B}_\epsilon}^2d\tau,
\\[4mm] \dis
 \int_0^t\tau\|\pa_tz_\epsilon(\tau)\|_{{\cal B}_\epsilon}^2d\tau
+
 \frac{t}{2}\|\nabla z_\epsilon(t)\|_{{\cal B}_\epsilon}^2
=
 \frac{1}{2}\int_0^t\|\nabla z_\epsilon(\tau)\|_{{\cal B}_\epsilon}^2d\tau.
\end{array}
\end{equation}
To obtain a priori estimates for higher derivatives,  
we differentiate \eqref{approximation-eq} with respect to $t$,
then $Z_\epsilon(x,t):=\pa_tz_\epsilon(x,t)$ satisfies
\begin{equation}\label{Zepsilon-eq}
 \begin{cases}
 \dis
 Z_t = \frac{1}{{\cal B}_\epsilon}\text{div}\left( {\cal B}_\epsilon\nabla Z \right),
 & (x,t)\in\R_+^n\times(0,\infty),
 \\ \dis
 \pa_\nu Z = 0,
 & (x,t)\in\pa\R_+^n\times(0,\infty).
 \end{cases}
\end{equation}
Since $Z_\epsilon(x,t)\in C(\overline{\R_+^n}\times[0,\infty))$,
by a maximum principle,
we get
\begin{equation}\label{infty2-eq}
 \|\pa_tz_\epsilon(t)\|_\infty
=
 \|Z_\epsilon(t)\|_\infty
\leq
 \|Z_\epsilon(0)\|_\infty
\leq
 \left\|
 \Delta z_0+\frac{2\nabla\sigma_\epsilon}{\sigma_\epsilon}\cdot\nabla z_0
 \right\|_\infty.
\end{equation}
Multiplying \eqref{Zepsilon-eq} by $t^2Z_\epsilon\sigma_\epsilon^2$,
$t^3\pa_tZ_\epsilon\sigma_\epsilon^2$ respectively
and integrating over $\R_+^n\times(0,t)$,
we get
\begin{equation}\label{appriori2-eq}
\begin{array}{c}
\dis
 \frac{t^2}{2}\|Z_\epsilon(t)\|_{{\cal B}_\epsilon}^2
+
 \int_0^t\tau^2\|\nabla Z_\epsilon(\tau)\|_{{\cal B}_\epsilon}^2d\tau
=
 \int_0^t\tau\|Z_\epsilon(\tau)\|_{{\cal B}_\epsilon}^2d\tau,
\\[4mm] \dis
 \int_0^t\tau^3\|\pa_tZ_\epsilon(\tau)\|_{{\cal B}_\epsilon}^2d\tau
+
 \frac{t^3}{2}\|\nabla Z_\epsilon(t)\|_{{\cal B}_\epsilon}^2
=
 \frac{3}{2}\int_0^t\tau^2\|\nabla Z_\epsilon(\tau)\|_{{\cal B}_\epsilon}^2d\tau.
\end{array}
\end{equation}
Therefore
since $\sigma_\epsilon(x)\leq c\sigma(x)$ for some $c>0$,
we obtain from \eqref{appriori1-eq} and \eqref{appriori2-eq}
\[
\begin{array}{c}
\dis
 t^2\|Z_\epsilon(t)\|_{{\cal B}_\epsilon}^2
\leq
 \|z_0\|_{{\cal B}_\epsilon}^2
\leq
 c\|z_0\|_{\cal B}^2,
\\[2mm] \dis
 \int_0^t
 \left(
 \tau^3\|\pa_tZ_\epsilon(\tau)\|_{\sigma_\epsilon^2}^2
 +
 \tau^2\|\nabla Z_\epsilon(\tau)\|_{\sigma_\epsilon^2}^2
 +
 \tau\|Z_\epsilon(\tau)\|_{\sigma_\epsilon^2}^2
 \right)
\leq
 c\|z_0\|_{\cal B}^2.
\end{array}
\]
As a consequence,
since $\sigma_\epsilon(x)=\sigma(x)$ for $|x|>\epsilon$,
it holds that for $t>\delta$
\[
 \int_\delta^t\int_{\R_+^n\setminus B_\epsilon}
 \left(
 \delta^3|\pa_tZ_\epsilon(x,t)|^2+\delta^2|\nabla Z_\epsilon(x,t)|^2+\delta|Z_\epsilon(x,t)|^2
 \right)
 {\cal B}(x)dx
\leq
 c\|z_0\|_{\cal B}^2.
\]
Therefore
by a compact embedding,
there exist a subsequence $\{Z_{\epsilon_k}(x,t)\}_{k\in\N}\subset\{Z_{\epsilon}(x,t)\}_{\epsilon>0}$
and a limiting function $Z(x,t)\in L^2(\R_+^n\setminus B_r\times(\delta,\delta^{-1}))$ for any $r,\delta\in(0,1)$
such that for any $\nu>0$ and $r_1,\delta_1\in(0,1)$,
there exists $k_1(\nu,r_1,\delta_1)>0$ such that for $k>k_1$
\[
 \int_{\delta_1}^t\int_{\R_+^n\setminus B_{r_1}}|Z_{\epsilon_k}(x,t)-Z(x,t)|^2{\cal B}(x)dx \leq \nu.
\]
As a consequence,
by using \eqref{infty2-eq},
we get for $k,k'>k_1$
\[
\begin{array}{lll}
\dis
 \int_{\delta_1}^t\int_{\R_+^n}|Z_{\epsilon_k}(x,t)-Z_{\epsilon_{k'}}(x,t)|^2{\cal B}(x)dx
\hspace{-2mm}&\leq&\hspace{-2mm} \dis
 \nu + 2\sup_{\epsilon\in(0,1)}\sup_{\tau\in(0,t)}\|Z_\epsilon(\tau)\|_\infty^2
 \int_{\delta_1}^t\int_{B_{r_1}}{\cal B}(x)dx
\\[6mm] \hspace{-2mm}&\leq&\hspace{-2mm} \dis
 \nu + 2t\left\|
 \Delta z_0+\frac{2\nabla\sigma_\epsilon}{\sigma_\epsilon}\cdot\nabla z_0
 \right\|_\infty
 \int_{B_{r_1}}{\cal B}(x)dx.
\end{array}
\]
This implies
\begin{equation}\label{UniformZ-eq}
 Z_{\epsilon_k} \to Z \hh\text{in } L^2(\R_+^n\times(\delta,\delta^{-1}))
\end{equation}
for any $\delta\in(0,1)$.
Let $\theta_\mu(x)$ be a smooth cut off function satisfying
\[
 \theta_\mu(x) =
 \begin{cases}
 0 & \text{if } |x|<\mu,\\
 1 & \text{if } |x|>2\mu,
 \end{cases}
\hhh
 0\leq\theta_\mu(x)\leq1,
\hhh
 |\nabla\theta_\mu(x)|<c/\mu.
\]
Since $\sigma_\epsilon(x)=\sigma(x)$ for $|x|>\epsilon$,
multiplying \eqref{Zepsilon-eq} by
$(t-\delta)(Z_{\epsilon}(x,t)-Z_{\epsilon'}(x,t))\theta_r(x)^2\sigma(x)^2$
and integrating over $\R_+^n\times(\delta,t)$,
then we get for $t>\delta>0$ and $\epsilon',\epsilon<r$
\[
\begin{array}{l}
\dis
 (t-\delta)\int_{\R_+^n}
 \left( Z_{\epsilon}(x,t)-Z_{\epsilon'}(x,t) \right)^2\theta_r(x)^2{\cal B}(x)dx
\\[4mm] \dis \hspace{15mm}
\leq
 \int_\delta^td\tau\int_{\R_+^n}
 |(Z_{\epsilon}(x,\tau)-Z_{\epsilon'}(x,\tau))|^2{\cal B}(x)dx
\\[4mm] \dis \hspace{25mm}
+
 cr^{-2}\int_\delta^t(\tau-\delta)d\tau\int_{\R_+^n}
 |(Z_{\epsilon}(x,\tau)-Z_{\epsilon'}(x,\tau))|^2{\cal B}(x)dx.
\end{array}
\]
Therefore
by virtue of \eqref{UniformZ-eq},
we find
\[
 \lim_{k\to\infty}\int_{\R_+^n\setminus B_r}|Z_{\epsilon_k}(x,t)-Z(x,t)|^2{\cal B}(x)dx=0
\hh \text{uniformly for } t\in(\delta,\delta^{-1})
\]
with any fixed $r>0$ and $\delta\in(0,1)$.
As a consequence,
from \eqref{infty2-eq},
we obtain
\[
 \lim_{k\to\infty}\int_{\R_+^n}|Z_{\epsilon_k}(x,t)-Z(x,t)|^2{\cal B}(x)dx=0
\hh \text{uniformly for } t\in(\delta,\delta^{-1})
\]
with any fixed $\delta\in(0,1)$.
Next we provide estimates of $\nabla z_\epsilon(x,t)$.
Since $\sigma_\epsilon(x)=\sigma(x)$ for $|x|>\epsilon$,
we see that
\[
\begin{array}{l}
\dis
 \int_{\R_+^n}|\nabla z_{\epsilon}(x,t)|^2{\cal B}(x)dx
=
 \int_{\R_+^n} z_{\epsilon}(x,t)\nabla\bigl( {\cal B}(x)\nabla z_\epsilon(x,t) \bigr)dx
\\[4mm] \dis \hspace{15mm}
=
 \int_{\R_+^n\setminus B_\epsilon} z_{\epsilon}(x,t)\nabla\bigl( {\cal B}_\epsilon(x)\nabla z_\epsilon(x,t) \bigr)dx
+
 \int_{B_\epsilon}z_{\epsilon}(x,t)\nabla\bigl( {\cal B}(x)\nabla z_\epsilon(x,t) \bigr)dx
\\[6mm] \dis \hspace{15mm}
=
 \int_{\R_+^n\setminus B_\epsilon} z_{\epsilon}(x,t)\pa_tz_\epsilon(x,t){\cal B}_\epsilon(x)dx
+
 \int_{B_\epsilon}z_{\epsilon}(x,t)\nabla\bigl( {\cal B}(x)\nabla z_\epsilon(x,t) \bigr)dx
\\[6mm] \dis \hspace{15mm}
\leq
 \|z_\epsilon(t)\|_{\cal B}\|\pa_tz_\epsilon(t)\|_{\cal B} + I_\epsilon.
\end{array}
\]
Since $z_\epsilon(x,t)$ is a solution of \eqref{approximation-eq},
by virtue of \eqref{Dsigma-eq},
$I_\epsilon$ is estimated as follows.
\[
\begin{array}{lll}
\dis
 I_\epsilon
\hspace{-2mm}&=&\hspace{-2mm} \dis
 \int_{B_\epsilon}z_{\epsilon}(x,t)
 \left(
 \Delta z_\epsilon(x,t)+\frac{2\nabla\sigma}{\sigma}\cdot\nabla z_\epsilon(x,t)
 \right)
 {\cal B}(x)dx
\\[6mm] \hspace{-2mm}&=&\hspace{-2mm} \dis
 \int_{B_\epsilon}z_{\epsilon}(x,t)
 \left(
 \pa_tz_\epsilon(x,t)
 -
 \frac{2\nabla\sigma_\epsilon}{\sigma_\epsilon}\cdot\nabla z_\epsilon(x,t)
 +
 \frac{2\nabla\sigma}{\sigma}\cdot\nabla z_\epsilon(x,t)
 \right)
 {\cal B}(x)dx
\\[6mm] \hspace{-2mm}&=&\hspace{-2mm} \dis
 (z_\epsilon(t),\pa_tz_\epsilon(t))_{\cal B}
-
 \int_{\pa B_\epsilon}
 z_\epsilon(x,t)^2
 \left(
 \frac{\pa_\nu\sigma_\epsilon}{\sigma_\epsilon}-\frac{\pa_\nu\sigma}{\sigma}
 \right)
 {\cal B}(x)dx
\\[4mm] && \dis \hspace{37.5mm}
+
 \int_{B_\epsilon}
 z_\epsilon(x,t)^2
 \nabla\cdot
 \left(
 \left( \frac{\nabla\sigma_\epsilon}{\sigma_\epsilon}-\frac{\nabla\sigma}{\sigma} \right)
 {\cal B}(x)
 \right)dx
\\[6mm] \hspace{-2mm}&\leq&\hspace{-2mm} \dis
 \|z_\epsilon(t)\|_{\cal B}\|\pa_tz_\epsilon(t)\|_{\cal B}
+
 \|z_\epsilon(t)\|_\infty^2
 \left(
 \int_{\pa B_\epsilon}|x|^{-2\gamma-1}dx
 +
 \int_{B_\epsilon}|x|^{-2\gamma-2}dx
 \right).
\end{array}
\]
Therefore
there exists $c>0$ such that
\begin{equation}\label{nablazestimate-eq}
 \sup_{t\in(0,\infty)}\|\nabla z_\epsilon(t)\|_{\cal B} \leq c.
\end{equation}
Furthermore
we see that
\[
\begin{array}{l}
\dis
 \int_{\R_+^n}|\nabla z_{\epsilon}(x,t)-\nabla z_{\epsilon'}(x,t))|^2{\cal B}(x)dx
\\[4mm] \dis \hspace{15mm}
=
 \int_{\R_+^n}
 (z_{\epsilon}(x,t)-z_{\epsilon'}(x,t))
 \nabla\bigl( {\cal B}(x)(\nabla z_\epsilon(x,t)-\nabla z_{\epsilon'}(x,t)) \bigr)dx
\\[6mm] \dis \hspace{15mm}
=
 \int_{\R_+^n\setminus B_\epsilon}
 (z_{\epsilon}(x,t)-z_{\epsilon'}(x,t))(\pa_tz_{\epsilon}(x,t)-\pa_tz_{\epsilon'}(x,t))
 {\cal B}(x)dx
\\[4mm] \dis \hspace{40mm}
+
 \int_{B_\epsilon}
 (z_{\epsilon}(x,t)-z_{\epsilon'}(x,t))
 \nabla\bigl( {\cal B}(x)(\nabla z_{\epsilon}(x,t)-\nabla z_{\epsilon'}(x,t)) \bigr)dx
\\[6mm] \dis \hspace{15mm}
\leq
 \|z_{\epsilon}(t)-z_{\epsilon'}(t))\|_{\cal B}
 \|\pa_tz_{\epsilon}(t)-\pa_tz_{\epsilon'}(t))\|_{\cal B}
+
 J_\epsilon.
\end{array}
\]
Then from \eqref{Dsigma-eq},
$J_\epsilon$ is estimated as follows.
\[
\begin{array}{lll}
\dis
 J_\epsilon
\hspace{-2mm}&=&\hspace{-2mm} \dis
 \int_{B_\epsilon}
 (z_{\epsilon}(x,t)-z_{\epsilon'}(x,t))
 \nabla\bigl( {\cal B}(x)(\nabla z_{\epsilon}(x,t)-\nabla z_{\epsilon'}(x,t)) \bigr)dx
\\[6mm] \hspace{-2mm}&=&\hspace{-2mm} \dis
 \int_{B_\epsilon}
 (z_{\epsilon}(x,t)-z_{\epsilon'}(x,t))
 (\Delta z_{\epsilon}(x,t)-\Delta z_{\epsilon'}(x,t))
 {\cal B}(x)dx
\\[2mm] && \dis \hspace{10mm}
+
 \int_{B_\epsilon}
 (z_{\epsilon}(x,t)-z_{\epsilon'}(x,t))
 \frac{2\nabla\sigma}{\sigma}\cdot(\nabla z_{\epsilon}(x,t)-\nabla z_{\epsilon'}(x,t))
 {\cal B}(x)dx
\\[6mm] \hspace{-2mm}&=&\hspace{-2mm} \dis
 \int_{B_\epsilon}
 (z_{\epsilon}(x,t)-z_{\epsilon'}(x,t))
 (\pa_tz_{\epsilon}(x,t)-\pa_tz_{\epsilon'}(x,t))
 {\cal B}(x)dx
\\[2mm] && \dis \hspace{10mm}
+  \int_{B_\epsilon}
 (z_{\epsilon}(x,t)-z_{\epsilon'}(x,t))
 \left(
 \frac{2\nabla\sigma_\epsilon}{\sigma_\epsilon}\cdot\nabla z_{\epsilon}(x,t)
 -
 \frac{2\nabla\sigma_{\epsilon'}}{\sigma_{\epsilon'}}\cdot\nabla z_{\epsilon'}(x,t)
 \right)
 {\cal B}(x)dx
\\[4mm] && \dis \hspace{22.5mm}
+
 \int_{B_\epsilon}
 (z_{\epsilon}(x,t)-z_{\epsilon'}(x,t))
 \frac{2\nabla\sigma}{\sigma}\cdot(\nabla z_{\epsilon}(x,t)-\nabla z_{\epsilon'}(x,t))
 {\cal B}(x)dx
\\[6mm] \hspace{-2mm}&\leq&\hspace{-2mm} \dis
 \|z_{\epsilon}(t)-z_{\epsilon'}(t)\|_{\cal B}
 \|\pa_tz_{\epsilon}(t)-\pa_tz_{\epsilon'}(t)\|_{\cal B}
\\[2mm] && \dis \hspace{10mm}
+
 c\|z_{\epsilon}(t)-z_{\epsilon'}(t)\|_{\infty}
 \Bigl(
 \|\nabla z_\epsilon(t)\|_{\cal B}+\|\nabla z_{\epsilon'}(t)\|_{\cal B}
 \Bigr)
 \left( \int_{B_\epsilon}|x|^{-2}{\cal B}(x)dx \right)^{1/2}.
\end{array}
\]
As a consequence,
by virtue of \eqref{nablazestimate-eq},
we obtain
\[
 \lim_{\epsilon\to\infty}J_\epsilon=0
\hh\text{uniformly for } t\in[0,t_0)
\]
with any fixed $t_0>0$.
This implies
\[
 \lim_{\epsilon,\epsilon'\to\infty}
 \int_{\R_+^n}|\nabla z_\epsilon(x,t)-\nabla z_{\epsilon'}(x,t)|^2\sigma(x)^2dx=0
\hh\text{uniformly for } t\in[0,t_0)
\]
with any fixed $t_0>0$.
Finally we verify a conservation in $L_{\cal B}^1$.
Let $s_\mu(\zeta)$ be a nondecreasing smooth function such that
$s_\mu(\zeta)=0$ if $\zeta\leq0$, $s_\mu(\zeta)=1$ if $\zeta>2\mu$ and $|s_\mu'(\zeta)|<c/\mu$.
Set $S_\mu(\zeta)=\int_0^\zeta s_\mu(a)da$.
Then
since $z_\epsilon\in C^1([0,\infty);L_{{\cal B}_\epsilon}^1)\cap
 C([0,\infty);H_{{\cal B}_\epsilon}^1)$,
we get for $0<\epsilon'<\epsilon$
\[
\begin{array}{lll}
\dis
 \pa_t\int_{\R_+^n}S_\mu(z_\epsilon-z_{\epsilon'})\theta_\epsilon(x)^2{\cal B}(x)dx
\hspace{-2mm}&=&\hspace{-2mm} \dis
 -\int_{\R_+^n}(\nabla z_\epsilon-\nabla z_{\epsilon'})\cdot\nabla
 \left( s_\mu(z_\epsilon-z_{\epsilon'})\theta_\epsilon(x)^2 \right)
 {\cal B}(x)dx
\\[4mm] \hspace{-2mm}&\leq&\hspace{-2mm} \dis
 2\int_{\R_+^n}|\nabla z_\epsilon-\nabla z_{\epsilon'}||\nabla\theta_\epsilon|{\cal B}(x)dx
\\[4mm] \hspace{-2mm}&\leq&\hspace{-2mm} \dis
 c\|\nabla z_\epsilon-\nabla z_{\epsilon'}\|_{\cal B}
 \left( \epsilon^{-2}\int_{B_\epsilon}{\cal B}(x)dx \right)^{1/2}.
\end{array}
\]
Therefore
integrating both sides and taking $\mu\to0$,
we obtain
\[
 \|(z_\epsilon(t)-z_{\epsilon'}(t))_+\|_{L_{\cal B}^1(\R_+^n\setminus B_\epsilon)}
\leq
 c\left( \epsilon^{-2}\int_{B_\epsilon}{\cal B}(x)dx \right)^{1/2}
 \int_0^t\|\nabla z_\epsilon(\tau)-\nabla z_{\epsilon'}(\tau)\|_{\cal B}d\tau.
\]
Since
estimates of $(z_\epsilon-z_{\epsilon'})_-$ is derived by the same way,
we obtain
\begin{equation}\label{L^1-eq}
 \|z_\epsilon(t)-z_{\epsilon'}(t)\|_{L_{\cal B}^1(\R_+^n\setminus B_\epsilon)}
\leq
 c\left( \epsilon^{-2}\int_{B_\epsilon}{\cal B}(x)dx \right)^{1/2}
 \int_0^t\|\nabla z_\epsilon(\tau)-\nabla z_{\epsilon'}(\tau)\|_{\cal B}d\tau.
\end{equation}
Since
$\int_{R_+^n}z_\epsilon(x,t){\cal B}_\epsilon(x)dx=\int_{R_+^n}z_0(x){\cal B}_\epsilon(x)dx$,
by using \eqref{zepsiloninfty-eq}, \eqref{nablazestimate-eq}-\eqref{L^1-eq}
and $2\gamma<(n-2)$,
we find that
$z_{\epsilon_k}(t)$ converges to $z(t)$ in $C([0,t_0];L_{\cal B}^1)$
and it satisfies $\int_{R_+^n}z(x,t){\cal B}(x)dx=\int_{R_+^n}z_0(x){\cal B}(x)dx$.
Therefore
from above a priori estimates,
this limiting function $z(x,t)$ is assumed to be
\[
 z(x,t)\in
 C([0,\infty);L_{\cal B}^1\cap L_{\cal B}^2)\cap C^1((0,\infty);L_{\cal B}^2)\cap
 C([0,\infty);H_{\cal B}^1).
\]
Furthermore
it solves
\[
\begin{cases}
\dis
 (z_t(t),\zeta)_{\cal B} = (\nabla z(t), \nabla\zeta)_{\cal B},\hspace{3mm}
 \forall\zeta\in C_c^\infty(\overline{\R_+^n}),\hh t>0,
\\[2mm] \dis
 z(0)=z_0.
\end{cases}
\]
Since $z_{\epsilon_k}\to z$ in
$C^1([0,t_0];L_{\cal B}^2)\cap C([0,t_0];H_{\cal B}^1)$
for any fixed $t_0$,
by using \eqref{appriori1-eq} and \eqref{appriori2-eq},
we see that
\[
 \|z(t)\|_{\cal B} + \sqrt{t}\|\nabla z(t)\|_{\cal B} + t\|\pa_tz(t)\|_{\cal B}
\leq
 c\|z_0\|_{\cal B}.
\]
Since $X=\{\zeta\in C_c^\infty(\R_+^n);\pa_\nu\zeta=0$ on $\pa\R_+^n\}$
is dense in $L_{\cal B}^2\cap L_{\cal B}^1$,
by a density argument,
we obtain the conclusion.
\end{proof}
%%%%%%%%%%%%%%%%%%%%%%%%%%%%%%%%%%%%%%%%%%%%%%%%%%%%%%%%%%%

We define a semigroup $e^{A_0t}$: $L_{\cal B}^2\to L_{\cal B}^2$ by $e^{A_0t}z_0=z(t)$,
where $z(t)$ is a unique solution of \eqref{xiweak-eq} given in Lemma \ref{Existence2-lem}.
Now
we construct the heat kernel of \eqref{xiweak-eq} (see e.g. Theorem 7.7 and Theorem 7.13 in \cite{Grigoryan})
First
we provide fundamental properties of $e^{A_0t}$.

%%%%%%%%%%%%%%%%%%%%%%%%%%%%%%%%%%%%%%%%%%%%%%%%%%%%%%%%%%%
\begin{lem}\label{L^pL^q-lem}
Let $e^{A_0t}$\h{\rm:} $L_{\cal B}^2\to L_{\cal B}^2$ be a semigroup defined above.
Then there exists $\nu>0$ such that
\\[2mm]
\begin{tabular}{cl}
{\rm(i)} & $\|e^{A_0t}z_0\|_{\cal B}\leq \|z_0\|_{\cal B}$ {\rm for} $t>0$,
\\[2mm]
{\rm(ii)} & $\|e^{A_0t}z_0\|_{\infty}\leq ct^{-\nu/2}\|z_0\|_{L_{\cal B}^1}$ {\rm for} $t>0$,
\\[2mm]
{\rm(iii)} & $\|e^{A_0t}z_0\|_{\infty}\leq ct^{-\nu/4}\|z_0\|_{\cal B}$ {\rm for} $t>0$.
\end{tabular}
\\[2mm]
Furthermore
let $z(x,t)=(e^{A_0t}z_0)(x)$ and $z_0\in L_{\cal B}^2$.
Then 
it holds that
\begin{equation}\label{pointwisexi-eq}
 z(x,t)\in C^{2,1}(\overline{\R_+^n}\setminus\{0\}\times(0,\infty)).
\end{equation}
\end{lem}
%%%%%%%%%%%%%%%%%%%%%%%%%%%%%%%%%%%%%%%%%%%%%%%%%%%%%%%%%%%

\begin{proof}
By Lemma \ref{Existence2-lem},
we note that $\pa_t\|u(t)\|_{\cal B}^2=-2\|\nabla u(t)\|_{\cal B}^2$.
This implies (i).
Next we will show (ii).
We recall the Caffarelli-Kohn-Nirenberg inequality:
\begin{equation}\label{CaffarelliKN-eq}
 \bigl\| |x|^mu \bigr\|_{L^r(\R_+^n)}
\leq
 c\bigl\| |x|^\alpha|\nabla u| \bigr\|_{L^p(\R_+^n)}^a
 \bigl\| |x|^\beta u \bigr\|_{L^q(\R_+^n)}^{1-a},
\end{equation}
where
$p,q\geq1$, $r>0$, $0\leq a\leq 1$, $m=a\sigma+(1-a)\beta$ and
\[
\begin{array}{c}
\dis
 \frac{1}{p}+\frac{\alpha}{n}>0,
\hh
 \frac{1}{q}+\frac{\beta}{n}>0,
\hh
 \frac{1}{r}+\frac{m}{n}>0,
\\[6mm] \dis
 \frac{1}{r}+\frac{m}{n} =
 a\left( \frac{1}{p}+\frac{\alpha-1}{n} \right) + (1-a)\left( \frac{1}{q}+\frac{\beta}{n} \right)
\end{array}
\]
\[
 \begin{array}{ll}
 \alpha-\sigma\geq0 & \text{if } a>0,
 \\[1mm] \dis
 \alpha-\sigma\leq1 &\dis \text{if } a>0,\ \frac{1}{r}+\frac{m}{n}=\frac{1}{p}+\frac{\alpha-1}{n}.
 \end{array}
\]
Here
we recall that ${\cal B}(x)\sim|x|^{-2\gamma}$.
Then
by using the Caffarelli-Kohn-Nirenberg inequality with
$m=\alpha=-\gamma$, $\beta=-2\gamma$, $r=p=2$, $q=1$,
we get
\[
 \|z\|_{\cal B} \leq c\|\nabla z\|_{\cal B}^a\|z\|_{L_{\cal B}^1}^{1-a}
\]
for some $a\in(0,1)$.
Therefore
from Theorem 4.1.1 in \cite{Coste},
we obtain
\begin{equation*}
 \|e^{A_0t}z_0\|_{\cal B} \leq ct^{-\nu/2}\|z_0\|_{L_{\cal B}^1},
\hh t>0,
\end{equation*}
where $\nu/2=a/(1-a)$,
which shows (ii).
As a consequence,
it holds that
\[
\begin{array}{lll}
\dis
 \|e^{A_0t}z_0\|_{\cal B}
\hspace{-2mm}&=&\hspace{-2mm} \dis
 \|e^{A_0t/2}e^{A_0t/2}z_0\|_{L_{\cal B}^2}
\leq
 ct^{-\nu/4}\|e^{A_0t/2}z_0\|_{L_{\cal B}^1}
\\[2mm] \hspace{-2mm}&\leq&\hspace{-2mm} \dis
 ct^{-\nu/4}\|z_0\|_{L_{\cal B}^1}.
\end{array}
\]
Since $e^{A_0t}$ is a bounded self-adjoint operator on $L_{\cal B}^2$, 
by the duality, we obtain (iii).
Since $\sigma(x)$ is positive and smooth on $\overline{\R_+^n}\setminus\{0\}$,
by virtue of (ii),
a local regularity theory for parabolic equations implies \eqref{pointwisexi-eq}.
\end{proof}
%%%%%%%%%%%%%%%%%%%%%%%%%%%%%%%%%%%%%%%%%%%%%%%%%%%%%%%%%%%

Following p.{\h}p.\h198\h-\h200 in \cite{Grigoryan},
we construct the heat kernel of \eqref{z-eq}.
By virtue of \eqref{pointwisexi-eq},
for any fixed $t>0$ and $x\in\R_+$,
the mapping $z_0\mapsto(e^{A_0t}z_0)(x)$: $L_{\cal B}^2\to\R$
becomes a bounded linear operator.
Therefore
by the Riesz representation theorem,
there exists $\hat{p}_{t,x}(\zeta)\in L_{\cal B}^2$ such that
\[
 (\hat{p}_{t,x},f)_{\cal B} = (e^{A_0t}f)(x),
\hh \forall f\in L_{\cal B}^2.
\]
Then
it is verified that
\[
 \hat{p}_{t,x}(\zeta)>0,\hh \hat{p}_{t,x}\in L_{\cal B}^1,\hh \|\hat{p}_{t,x}\|_{L_{\cal B}^1}=1,
\hh x\in\R_+^n,\ t>0.
\]
Moreover
by Lemma \ref{L^pL^q-lem} (ii),
the duality argument shows
\[
 \|\hat{p}_{t,x}\|_{\infty} \leq ct^{-\nu/2},
\hh x\in\R_+^n,\ t>0,
\]
where $c>0$ is a constant independent of $x$ and $t$.
We put
\[
 p(x,\xi,t) = (\hat{p}_{t/2,x},\hat{p}_{t/2,\xi})_{\cal B},
\hh x,\xi\in\R_+^n,\ t>0.
\]
Then
from Theorem 7.13 in \cite{Grigoryan},
$e^{A_0t}z_0$ is expressed by
\[
 (e^{A_0t}z_0)(x) =  \int_{\R_+^n}p(x,\xi,t)z_0(\xi){\cal B}(\xi)d\xi,
\hh x\in\R_+^n,\ t>0
\]
for $z_0\in L_{\cal B}^2$.
Then Lemma \ref{Existence2-lem} implies
\[
 e^{A_0t}z_0\in
 C([0,\infty);L_{\cal B}^2)\cap C^1((0,\infty);L_{\cal B}^2)\cap
 C((0,\infty);H_{\cal B}^1),\hh
 z_0\in L_{\cal B}^2. 
\]

%%%%%%%%%%%%%%%%%%%%%%%%%%%%%%%%%%%%%%%%%%%%%%%%%%%%%%%%%%%
\begin{lem}\label{Boundp-lem}
Let $p(x,\xi,t)$ be defined above.
Then
$p(\cdot,\xi,\cdot)\in C^{2,1}(\R_+^n\times(0,\infty))$
satisfies \eqref{z-eq} for any fixed $\xi\in\R_+^n$.
Furthermore
for $r>0$ and $t_0>0$ there exists $c(r,t_0)>0$ such that
\[
 |p_t(x,\xi,t)|+|\nabla_xp(x,\xi,t)|+|D_x^2 p(x,\xi,t)| \leq c(r,t_0),
\hh |x|>r,\ \xi\in\R_+^n,\ t>t_0.
\]
\end{lem}
%%%%%%%%%%%%%%%%%%%%%%%%%%%%%%%%%%%%%%%%%%%%%%%%%%%%%%%%%%%

%%%%%%%%%%%%%%%%%%%%%%%%%%%%%%%%%%%%%%%%%%%%%%%%%%%%%%%%%%%
\begin{proof}
By definition of $p(x,\xi,t)$,
we observe that $p(x,\xi,t)=(e^{A_0t/2}[\hat{p}_{t/2,\xi}(\cdot)])(x)$.
Hence
$p(\cdot,\xi,\cdot)$ satisfies \eqref{z-eq} for any fixed $\xi\in\R_+^n$.
Furthermore
from Lemma \ref{L^pL^q-lem} and $\|\hat{p}_{t/2.\xi}\|_{L_{\cal B}^1}=1$,
there exists $\nu_p>0$ and $c_p>0$ for $p\geq2$
\[
 \|p(x,\xi,t)\|_{L_{\cal B}^p,dx}
\leq
 c_pt^{-\nu_p}\|\hat{p}_{t/2,\xi}\|_{L_{\cal B}^1}
\leq
 c_pt^{-\nu_p}.
\]
Since coefficients $\nabla\sigma/\sigma$ of \eqref{z-eq} is bounded far from the origin,
by a parabolic regularity theory,
for $r>0$ and $t_0>0$ there exists $c(r,t_0)>0$ such that
\[
 |p_t(x,\xi,t)|+|\nabla_xp(x,\xi,t)|+|D_x^2 p(x,\xi,t)| \leq c(r,t_0),
\hh |x|>r,\ \xi\in\R_+^n,\ t>t_0,
\]
which completes the proof.
\end{proof}
%%%%%%%%%%%%%%%%%%%%%%%%%%%%%%%%%%%%%%%%%%%%%%%%%%%%%%%%%%%

To apply Theorem 2.7 in \cite{Grigoryan-SC},
we prepare several lemmas.
Let
\[
 V(x_0,R)=\int_{B(x_0,R)}{\cal B}(x)dx.
\]

%%%%%%%%%%%%%%%%%%%%%%%%%%%%%%%%%%%%%%%%%%%%%%%%%%%%%%%%%%%
\begin{lem}[Doubling property]\label{Doubling-lem}
There exists $c>0$ such that
\begin{equation}\label{Doubling-eq}
 V(x_0,2R) \leq cV(x_0,R)
\hspace{5mm}\mathrm{for}\ x_0\in\R_+^n,\ R>0.
\end{equation}
Furthermore
there exist $c_1,c_2>0$ such that
\[
 c_1R^n(|x_0|+R)^{-2\gamma} < V(x_0,R) < c_2R^n(|x_0|+R)^{-2\gamma}.
\]
\end{lem}
%%%%%%%%%%%%%%%%%%%%%%%%%%%%%%%%%%%%%%%%%%%%%%%%%%%%%%%%%%%

%%%%%%%%%%%%%%%%%%%%%%%%%%%%%%%%%%%%%%%%%%%%%%%%%%%%%%%%%%%
\begin{proof}
Repeating calculations as in the proof of Proposition 4.1 in \cite{Moschini-T},
we obtain the conclusion.
\end{proof}
%%%%%%%%%%%%%%%%%%%%%%%%%%%%%%%%%%%%%%%%%%%%%%%%%%%%%%%%%%%

Here
we recall a fundamental result concerning the volume growth (see Lemma 5.2.7 in \cite{Coste}).

%%%%%%%%%%%%%%%%%%%%%%%%%%%%%%%%%%%%%%%%%%%%%%%%%%%%%%%%%%%
\begin{lem}\label{Doubling2-lem}
Assume \eqref{Doubling-eq}.
Then there exists $c>0$ such that
\[
 V(x_1,r) \leq e^{c|x_1-x_2|/r}V(x_2,r),
\hh x_1,x_2\in\R_+^n,\ r>0.
\]
\end{lem}
%%%%%%%%%%%%%%%%%%%%%%%%%%%%%%%%%%%%%%%%%%%%%%%%%%%%%%%%%%%

%%%%%%%%%%%%%%%%%%%%%%%%%%%%%%%%%%%%%%%%%%%%%%%%%%%%%%%%%%%
\begin{lem}[Weighted Poincar\'e inequality]
There exists $c>0$ such that
\[
 \inf_{\xi\in\R}\int_{B(x_0,R)}|f(x)-\xi|^2{\cal B}(x)dx
\leq
 cR^2\int_{B(x_0,R)}|\nabla f(x)|^2{\cal B}(x)dx,
\hh \forall f\in C^1(\overline{B(x_0,R)}).
\]
\end{lem}
%%%%%%%%%%%%%%%%%%%%%%%%%%%%%%%%%%%%%%%%%%%%%%%%%%%%%%%%%%%

%%%%%%%%%%%%%%%%%%%%%%%%%%%%%%%%%%%%%%%%%%%%%%%%%%%%%%%%%%%
\begin{proof}
The proof follows from Proposition 4.3 in \cite{Moschini-T} with a slight modification.
\end{proof}
%%%%%%%%%%%%%%%%%%%%%%%%%%%%%%%%%%%%%%%%%%%%%%%%%%%%%%%%%%%

%%%%%%%%%%%%%%%%%%%%%%%%%%%%%%%%%%%%%%%%%%%%%%%%%%%%%%%%%%%
\begin{rem}
It is known that
\[
 \inf_{\xi\in\R}\int_{B(x_0,R)}|f(x)-\xi|^2{\cal B}(x)dx
=
 \int_{B(x_0,R)}|f(x)-f_{\mathrm{av}}(x)|^2{\cal B}(x)dx,
\hh \forall f\in C^1(\overline{B(x_0,R)}),
\]
where $f_{\mathrm{av}}(x)$ is the average of $f(x)$ on $B(x_0,R)$.
\end{rem}
%%%%%%%%%%%%%%%%%%%%%%%%%%%%%%%%%%%%%%%%%%%%%%%%%%%%%%%%%%%

Therefore
Theorem 2.7 in \cite{Grigoryan-SC} implies the following the heat kernel estimates.
\begin{equation}\label{HeatKernel-eq}
 \frac{c_1}{V(x,\sqrt{t})}\exp\left( -\frac{c_3|x-\xi|^2}{t} \right) < p(x,\xi,t) <
 \frac{c_2}{V(x,\sqrt{t})}\exp\left( -\frac{c_4|x-\xi|^2}{t} \right).
\end{equation}
Applying Lemma \ref{Doubling-lem}\h-{\h}Lemma \ref{Doubling2-lem} to \eqref{HeatKernel-eq},
we obtain the desired lower and upper estimates of $p(x,\xi,t)$.

%%%%%%%%%%%%%%%%%%%%%%%%%%%%%%%%%%%%%%%%%%%%%%%%%%%%%%%%%%%
\begin{lem}\label{Heatkernel-lem}
Let $k_\gamma(\xi,t) =  \left( |\xi|+\sqrt{t} \right)^{\gamma}$.
Then
there exist $c_i$ $(i=1,2,3,4)$ such that
% \[
%  \frac{c_1}{V(\xi,\sqrt{t})}\exp\left( -\frac{c_3|x-\xi|^2}{t} \right) < p(x,\xi,t) <
%  \frac{c_2}{V(\xi,\sqrt{t})}\exp\left( -\frac{c_4|x-\xi|^2}{t} \right),
% \]
% namely
% \begin{equation}\label{GKernel-eq}
\[ 
 c_1\left( \frac{k_\gamma(\xi,t)^2}{t^{n/2}} \right)
 \exp\left( -\frac{c_3|x-\xi|^2}{t} \right)
\leq
 p(x,\xi,t)
\leq
 c_2\left( \frac{k_\gamma(\xi,t)^2}{t^{n/2}} \right)
 \exp\left( -\frac{c_4|x-\xi|^2}{t} \right),
\]
\end{lem}
%%%%%%%%%%%%%%%%%%%%%%%%%%%%%%%%%%%%%%%%%%%%%%%%%%%%%%%%%%%

Furthermore
it is known that
$c_4$ in Lemma \ref{Heatkernel-lem} can be chosen arbitrarily close to $1/4$
(see (5.2.17) and Theorem 5.3.3 in \cite{Coste}),
which is stated as follows.

%%%%%%%%%%%%%%%%%%%%%%%%%%%%%%%%%%%%%%%%%%%%%%%%%%%%%%%%%%%
\begin{lem}\label{HeatkernelNew-lem}
Let $k_\gamma(x,t)$ be as in Lemma {\rm\ref{Heatkernel-lem}}.
For any $\epsilon>0$ there exists $c_\epsilon>0$ such that
\[
 p(x,\xi,t)
\leq
 c_\epsilon
 \left( \frac{k_\gamma(\xi,t)^2}{t^{n/2}} \right)
 \exp\left( -\frac{|x-\xi|^2}{4(1+\epsilon)t} \right).
\]
\end{lem}
%%%%%%%%%%%%%%%%%%%%%%%%%%%%%%%%%%%%%%%%%%%%%%%%%%%%%%%%%%%

%%%%%%%%%%%%%%%%%%%%%%%%%%%%%%%%%%%%%%%%%%%%%%%%%%%%%%%%%%%
\subsection{Backward type linear parabolic equations}
\label{Backward-sec}
%%%%%%%%%%%%%%%%%%%%%%%%%%%%%%%%%%%%%%%%%%%%%%%%%%%%%%%%%%%

In this subsection,
we study the backward type linear parabolic equations.
\begin{equation}\label{Backward-eq}
\begin{cases}
\dis
 b_s = \frac{1}{{\cal C}}\text{div}\left( {\cal C}\nabla b \right),
 & (y,s)\in\R_+^n\times(0,\infty),
\\ \dis
 \pa_\nu b = 0,
 & (y,s)\in\pa\R_+^n\times(0,\infty),
\\ \dis
 b(y,0) = b_0(y),
 & y\in\pa\R_+^n,
\end{cases}
\end{equation}
where
\[
 {\cal C}(y) = \sigma(y)^2\rho(y).
\]
Here we introduce suitable functional spaces related to \eqref{Backward-eq}. 
Set
\[
 L_{\cal C}^p =
 \left\{b\in L_{\text{loc}}^p(\R_+^n); \int_{\R_+^n}|b(y)|^p{\cal C}(y)dy
 \right\},
\hspace{7.5mm}
 H_{\cal C}^1 =
 \left\{ b\in L_{\cal C}^2; \nabla b\in L_{\cal C}^2 \right\}.
\]
Let $H_{\cal C}^*$ be the dual space of $H_{\cal C}^*$
and define a operator ${\cal A}_0$: $D({\cal A}_0)\to H_{\cal C}^*$
with $D({\cal A}_0)=H_{\cal C}^1$ by
\[
 {}_{H_{\cal C}^*}\langle {\cal A}_0b,\eta \rangle_{H_{\cal C}}
:=
 -(\nabla b,\nabla \eta)_{\cal C}.
\]
In the same manner as before,
$e^{{\cal A}_0t}b_0$ defines a semigroup on $L_{\cal C}^2$.
Then it is verified that
$e^{{\cal A}_0s}$ is symmetric on $L_{\cal C}^2$,
that is
\begin{equation}\label{ANew-eq}
 (e^{{\cal A}_0s}b_1,b_2)_{\cal C}=(b_1,e^{{\cal A}_0s}b_2)_{\cal C},
\hspace{7.5mm} b_1,b_2\in L_{\cal C}^2,\hspace{2.5mm} t>0.
\end{equation}
To obtain the heat kernel of \eqref{Backward-eq}, 
we change variables.
\[
 z(x,t) = b\left( (1-t)^{-1/2}x,-\log(1-t) \right).
\]
Then $z(x,t)$ satisfies \eqref{z-eq} with the initial data $z_0(x)=b_0(x)$.
Since $z(x,t)$ is given by
\[
 z(x,t) =
 \int_{\R_+^n}p(x,\xi,t)b_0(\xi){\cal B}(\xi)d\xi,
\]
by using $x=e^{-s/2}y$ and $t=1-e^{-s}$,
we see that
\[
 (e^{{\cal A}_0t}b_0)(y) = z\left( e^{-s/2}y,1-e^{-s} \right) =
 \int_{\R_+^n}\Theta(y,\xi,s)b_0(\xi){\cal B}(\xi)d\xi,
\]
where $\Theta(y,\xi,s)$ is given by
\[
 \Theta(y,\xi,s) = p\left( e^{-s/2}y,\xi,1-e^{-s} \right).
\]
Then
from Lemma \ref{Boundp-lem},
we find that
$\Theta(\cdot,\xi,\cdot)\in C^{2,1}(\R_+^n\times(0,\infty))$
satisfies \eqref{Backward-eq} for any fixed $\xi\in\R_+^n$.
Furthermore
from Lemma \ref{Boundp-lem},
for any $r$ and $s_0\in(0,1)$ there exists $c(r,s_0)>0$ such that
\begin{equation}\label{GammaBound-eq}
 |\Theta_s(y,\xi,s)|+|\nabla_y\Theta(y,\xi,s)|+|D_y^2\Theta(y,\xi,s)| \leq c(r,s_0)(1+|y|)
\end{equation}
for $|y|>r$, $\xi\in\R_+^n$, $s_0<s<s_0^{-1}$.
Moreover
by virtue of \eqref{ANew-eq},
it holds that
\begin{equation}\label{Gamma1-eq}
 \Theta(y,\xi,s)e^{-|y|^2/4} = \Theta(\xi,y,s)e^{-|\xi|^2/4}.
\end{equation}
%%%%%%%%%%%%%%%%%%%%%%%%%%%%%%%%%%%%%%%%%%%%%%%%%%%%%%%%%%%
Next
we consider non-homogeneous problems.
\begin{equation}\label{NonhomoBackward-eq}
\begin{cases}
\dis
 b_s = \frac{1}{{\cal C}}\text{div}\left( {\cal C}\nabla b \right),
 & (y,s)\in\R_+^n\times(0,\infty),
\\ \dis
 \pa_\nu b = f(y',s),
 & (y',s)\in\pa\R_+^n\times(0,\infty),
\\ \dis
 b(y,0) = 0,
 & y\in\pa\R_+^n.
\end{cases}
\end{equation}
Let $s'>0$ and assume $f(y',s)\in BC(\pa\R_+^n\times[0,s'))$.
Furthermore
let
$b(y,s)\in BC(\overline{\R_+^n}\times[0,s'))\cap BC^{2,1}(\R_+^n\times[\epsilon,s'))$
for any $\epsilon>0$
be a classical solution of \eqref{NonhomoBackward-eq}.
We denote by $\theta_\epsilon(y)$ a cut off function such that
$\theta_\epsilon(y)=0$ if $|y|<\epsilon$, $\theta_\epsilon(y)=1$ if $|y|>2\epsilon$
and $|D^i\theta_\epsilon(y)|<c\epsilon^{-i}$ ($i=1,2$).
Here we put
\[
 A^{\epsilon}(y,s,\tau) =
 e^{|y|^2/4}\int_{\R_+^n}\Theta(\mu,y,s-\tau)\theta_\epsilon(\mu)b(\mu,\tau){\cal C}(\mu)d\mu,
\hh 0<\tau<s.
\]
From \eqref{GammaBound-eq},
we find that
$A^\epsilon(y,s,\tau)$ is differentiable with respect to $\tau$ for $0<\tau<s$
and it satisfies
\[
\begin{array}{lll}
\dis
 A_\tau^\epsilon(y,s,\tau)
\hspace{-2mm}&=&\hspace{-2mm} \dis
 -e^{|y|^2/4}\int_{\R_+^n}\Theta_s(\mu,y,s-\tau)\theta_\epsilon(\mu)b(\mu,\tau){\cal C}(\mu)d\mu
\\
\hspace{-2mm}&&\hspace{-2mm} \dis \hspace{20mm}
+
 e^{|y|^2/4}\int_{\R_+^n}\Theta(\mu,y,s-\tau)\theta_\epsilon(\mu)b_s(\mu,\tau){\cal C}(\mu)d\mu
\end{array}
\]
for $0<\tau<s$ and $y\in\R_+^n$.
Then
since $\Theta(\cdot,y,\cdot)$ satisfies \eqref{Backward-eq} for any fixed $y\in\R_+^n$,
by using $\pa_\nu\theta_\epsilon=0$ on $\pa\R_+^n$,
we see that
\[
\begin{array}{l}
\dis
 \int_{\R_+^n}\Theta_s(\mu,y,s-\tau)\theta_\epsilon b\h {\cal C}d\mu
=
 -\int_{\R_+^n}\nabla_\mu\Theta(\mu,y,s-\tau)\cdot\nabla(\theta_\epsilon b){\cal C}d\mu
\\[4mm] \dis \hspace{10mm}
=
 -\int_{\pa\R_+^n}\Theta(\mu',y,s-\tau)\theta_\epsilon f(\mu',\tau){\cal C}d\mu'
+
 \int_{\R_+^n}\Theta(\mu,y,s-\tau)\theta_\epsilon b_s\h{\cal C}d\mu
\\[4mm] \dis \hspace{25mm}
+
 \int_{\R_+^n}\Theta(\mu,y,s-\tau)
 \left(
 2\nabla\theta_\epsilon\cdot\nabla b
 +
 \left( \Delta\theta_\epsilon+\frac{\nabla\theta_\epsilon\cdot\nabla{\cal C}}{\cal C} \right)b
 \right)
 {\cal C}d\mu.
\end{array}
\]
Here
we denote the last term by $R_\epsilon(y,s,\tau)$.
Integrating over $(\delta,s-\delta)$ both sides,
we get
\begin{equation}\label{Aepsilondelta-eq}
\begin{array}{lll}
\dis
 A^\epsilon(y,s,s-\delta)
\hspace{-2mm}&=&\hspace{-2mm} \dis
 A^\epsilon(y,s,\delta)
+
 e^{|y|^2/4}
 \left(
 \int_\delta^{s-\delta}d\tau\int_{\pa\R_+^n}\Theta(\mu',y,s-\tau)\theta_\epsilon f\h{\cal C}d\mu'
 \right.
\\[4mm]
\hspace{-2mm}&&\hspace{50mm} \dis
 \left.
 -
 \int_\delta^{s-\delta}R_\epsilon(y,s,\tau)d\tau
 \right).
\end{array}
\end{equation}
Then
since  ${\cal C}(\mu)\sim|\mu|^{-2\gamma}$ for $|y|\leq1$,
we see that
\[
\begin{array}{lll}
 |R_\epsilon|
\hspace{-2mm}&\leq&\hspace{-2mm} \dis
 c\sup_{\delta<\tau<s-\delta}\|b(\cdot,\tau)\|_{W^{1,\infty}(\R_+^n)}
 \int_{B_{2\epsilon}}
 \left(
 |\nabla\theta_\epsilon|+|\Delta\theta_\epsilon|+\frac{|\nabla\theta_\epsilon|}{|\mu|}
 \right)
 |\mu|^{-2\gamma}
 d\mu
\\[5mm]
\hspace{-2mm}&\leq&\hspace{-2mm} \dis
 c\sup_{\delta<\tau<s-\delta}\|b(\cdot,\tau)\|_{W^{1,\infty}(\R_+^n)}
 \left( \epsilon^{-2}\cdot\epsilon^{n-2\gamma}+\epsilon^{-1}\cdot\epsilon^{n-1-2\gamma} \right).
\end{array}
\]
Hence
since $b(y,s)\in BC^{2,1}(\R_+\times(\delta,s))$ and $2\gamma<n-2$,
it holds that
\[
 \lim_{\epsilon\to0}\int_\delta^{s-\delta}R_\epsilon(y,s,\tau)d\tau = 0
\hspace{5mm}\text{for any } \delta>0.
\]
Therefore
since $b(y,s)\in BC(\overline{\R_+^n}\times[0,s'))$ and $f(y',s)\in BC(\pa\R_+^n\times[0,s'))$,
by taking $\epsilon\to0$ in \eqref{Aepsilondelta-eq},
we obtain
\[
 A(y,s,s-\delta)
=
 A(y,s,\delta)
+
 e^{|y|^2/4}
 \int_\delta^{s-\delta}d\tau\int_{\pa\R_+^n}\Theta(\mu',y,s-\tau)f\h{\cal C}d\mu'
\]
for any fixed $\delta>0$,
where $A(y,s,\tau)$ is defined by
\[
 A(y,s,\tau) =
 e^{|y|^2/4}\int_{\R_+^n}\Theta(\mu,y,s-\tau)b(\mu,\tau){\cal C}(\mu)d\mu.
\]
Here
from \eqref{Gamma1-eq},
we note that
\[
 A(y,s,s-\delta)
=
 \int_{\R_+^n}\Theta(y,\mu,\delta)b(\mu,s-\delta){\cal B}(\mu)d\mu
=
  e^{{\cal A}_0\delta}b(s-\delta).
\]
Furthermore
we recall that $b(y,0)\equiv0$.
Therefore
since $b(y,s)\in BC(\overline{\R_+^n}\times[0,s'))$ and $b(y',s)\in BC(\pa\R_+^n\times[0,s'))$,
we take $\delta\to0$ to obtain
\[
 b(y,s)
=
 \int_0^sd\tau\int_{\pa\R_+^n}\Theta(y,\mu',s-\tau)f(\mu',\tau){\cal B}(\mu')d\mu'.
\]
Summing up the above facts,
we obtain the Duhamel principle for solution of \eqref{NonhomoBackward-eq}.

%%%%%%%%%%%%%%%%%%%%%%%%%%%%%%%%%%%%%%%%%%%%%%%%%%%%%%%%%%%
\begin{pro}\label{Rep-pro}
Let $f(y,s)\in BC(\pa\R_+^n\times[0,s')$
and
$b(y,s)$ be a solution of \eqref{NonhomoBackward-eq} with $b(y,0)=b_0(y)$.
Then
if $b(y,s)\in BC(\overline{\R_+^n}\times[0,s'))\cap BC^{2,1}(\R_+^n\times[\epsilon,s'))$
for any $\epsilon>0$,
it is expressed by
\[
 b(y,s) =
 \int_{\R_+^n}\Theta(y,\xi,s)b_0(\xi){\cal B}(\xi)d\xi +
 \int_0^sd\tau\int_{\pa\R_+^n}\Theta(y,\xi',s-\tau)f(\xi',\tau){\cal B}(\xi')d\xi'.
\]
Furthermore
it holds that
\[
 \Theta(y,\xi,s)
\leq
 c\frac{\left( |\xi|+\sqrt{1-e^{-s}} \right)^{2\gamma}}{(1-e^{-s})^{n/2}}
 \exp\left( -\frac{|e^{-s/2}y-\xi|^2}{6(1-e^{-s})} \right).
\]
\end{pro}
%%%%%%%%%%%%%%%%%%%%%%%%%%%%%%%%%%%%%%%%%%%%%%%%%%%%%%%%%%%

%%%%%%%%%%%%%%%%%%%%%%%%%%%%%%%%%%%%%%%%%%%%%%%%%%%%%%%%%%%
\begin{proof}
The Duhamel principle follows from the above argument.
Furthermore
by using Lemma \ref{HeatkernelNew-lem} with $\epsilon=1/2$,
then we obtain the upper bound of $\Theta(y,\xi,s)$. 
\end{proof}
%%%%%%%%%%%%%%%%%%%%%%%%%%%%%%%%%%%%%%%%%%%%%%%%%%%%%%%%%%%

Finally
we provide the $L^\infty$-$L^2$ type estimate for solutions of \eqref{Backward-eq}.

%%%%%%%%%%%%%%%%%%%%%%%%%%%%%%%%%%%%%%%%%%%%%%%%%%%%%%%%%%%
\begin{lem}\label{L^inftyL^2-lem}
There exists $c>0$ such that
\[
 |e^{{\cal A}_0s}b_0|
<
 \frac{c\|b_0\|_{\cal C}}{(1-e^{-s})^{n/2}}
\hspace{5mm}\mathrm{for}\ |y|<e^{s/2}.
\]
\end{lem}
%%%%%%%%%%%%%%%%%%%%%%%%%%%%%%%%%%%%%%%%%%%%%%%%%%%%%%%%%%%

%%%%%%%%%%%%%%%%%%%%%%%%%%%%%%%%%%%%%%%%%%%%%%%%%%%%%%%%%%%
\begin{proof}
Put $r_1=\sqrt{7}/(\sqrt{7}-\sqrt{6})$.
Then
a direct computation shows that
\begin{equation}\label{c'yxi-eq}
 \frac{|e^{-s/2}y-\xi|^2}{6}
\geq
 \frac{|\xi|^2}{7}
\hspace{5mm}\text{for } |\xi|>r_1,\ |y|<e^{s/2}.
\end{equation}
From Proposition \ref{Rep-pro},
we observe that
\begin{equation}\label{Asexpress-eq}
 e^{{\cal A}_0s}b_0
=
 \left( \int_0^{r_1}+\int_{r_1}^\infty \right)
 \Theta(y,\xi,s)b_0(\xi){\cal B}(\xi)
 d\xi.
\end{equation}
By the Schwarz inequality,
the first integral is estimated by
\[
 \int_0^{r_1}
 \Theta(y,\xi,s)b_0(\xi){\cal B}(\xi)
 d\xi
\leq
 \left(
 \int_0^{r_1}
 \Theta(y,\xi,s)^2{\cal B}(\xi)
 d\xi
 \right)^{1/2}
 \left(
 \int_0^{r_1}
 b_0(y)^2{\cal B}(\xi)
 d\xi
 \right)^{1/2}.
\]
Here
we note that ${\cal B}(\xi)\leq c|\xi|^{-2\gamma}$ and $\sqrt{1-e^{-s}}<1$.
Therefore
Proposition \ref{Rep-pro} implies
\[
\begin{array}{lll}
\dis
 \int_0^{r_1}
 \Theta(y,\xi,s)^2{\cal B}(\xi)
 d\xi
\hspace{-2mm}&\leq&\hspace{-2mm} \dis
 \frac{c}{(1-e^{-s})^n}
 \int_0^{r_1}
 \left( |\xi|+\sqrt{1-e^{-s}} \right)^{4\gamma}
 |\xi|^{-2\gamma}
 d\xi
\\[4mm]
\hspace{-2mm}&\leq&\hspace{-2mm} \dis
 \frac{c(r_1+1)^{4\gamma}}{(1-e^{-s})^n}
 \int_0^{r_1}
 |\xi|^{-2\gamma}
 d\xi.
\end{array}
\]
Since $\gamma<(n-2)/2$,
we obtain
\[
 \int_0^{r_1}
 \Theta(y,\xi,s)^2{\cal B}(\xi)
 d\xi
\leq
 \frac{c}{(1-e^{-s})^n}.
\]
Furthermore
it holds from ${\cal B}(\xi)=e^{|\xi|^2/4}{\cal C}(\xi)$ that
\[
 \int_0^{r_1}
 b_0(y)^2{\cal B}(\xi)
 d\xi
\leq
 e^{r_1^2/4}\int_0^{r_1}
 b_0(y)^2{\cal C}(\xi)
 d\xi
\leq
 e^{r_1^2/4}\|b_0\|_{\cal C}^2.
\]
Therefore
we obtain
\[
 \int_0^{r_1}
 \Theta(y,\xi,s)b_0(\xi){\cal B}(\xi)
 d\xi
\leq
 \frac{c\|b_0\|_{\cal C}}{(1-e^{-s})^{n/2}}.
\]
Next we estimate the second integral in \eqref{Asexpress-eq}.
By \eqref{c'yxi-eq},
we note that
\[
 \exp\left( -\frac{|e^{-s/2}y-\xi|^2}{6(1-e^{-s})} \right)
\leq
 \exp\left( -\frac{|\xi|^2}{7(1-e^{-s})} \right)
\leq
 e^{-|\xi|^2/7}
\]
for $|\xi|>r_1$ and $|y|<e^{s/2}$.
Therefore
since $1/7=1/8+1/56$,
by the Schwarz inequality,
we see that
\[
\begin{array}{l}
\dis
 \int_{r_1}^\infty
 \Theta(y,\xi,s)b_0(\xi){\cal B}(\xi)
 d\xi
\leq
 c\int_{r_1}^\infty
 \frac{\dis\left( |\xi|+\sqrt{1-e^{-s}} \right)^{2\gamma}}{(1-e^{-s})^{n/2}}
 e^{-|\xi|^2/7}|b_0(\xi)|{\cal B}(\xi)
 d\xi
\\[4mm] \dis \hspace{10mm}
\leq
 c\left(
 \int_{r_1}^\infty
 \frac{\dis\left( |\xi|+\sqrt{1-e^{-s}} \right)^{4\gamma}}{(1-e^{-s})^{n}}
 e^{-|\xi|^2/28}{\cal B}(\xi)
 d\xi
 \right)^{1/2}
 \left(
 \int_{r_1}^\infty
 b_0(\xi)^2e^{-|\xi|^2/4}{\cal B}(\xi)
 d\xi
 \right)^{1/2}.
\end{array}
\]
Since
$\sqrt{1-e^{-s}}<1$ and ${\cal C}(\xi)=e^{-|\xi|^2/4}{\cal B}(\xi)$,
we conclude
\[
\begin{array}{lll}
\dis
 \int_{r_1}^\infty
 \Theta(y,\xi,s)b_0(\xi){\cal B}(\xi)
 d\xi
\hspace{-2mm}&\leq&\hspace{-2mm} \dis
 \frac{c}{{(1-e^{-s})^{n}}}
 \left(
 \int_{r_1}^\infty
 \dis\left( |\xi|+1 \right)^{4\gamma}|\xi|^{-2\gamma}
 e^{-|\xi|^2/28}
 d\xi
 \right)^{1/2}
 \|b_0\|_{\cal C}
\\[6mm]
\hspace{-2mm}&\leq&\hspace{-2mm} \dis
 \frac{c\|b_0\|_{\cal C}}{{(1-e^{-s})^{n/2}}}.
\end{array}
\]
Thus the proof is completed.
\end{proof}
%%%%%%%%%%%%%%%%%%%%%%%%%%%%%%%%%%%%%%%%%%%%%%%%%%%%%%%%%%%

%%%%%%%%%%%%%%%%%%%%%%%%%%%%%%%%%%%%%%%%%%%%%%%%%%%%%%%%%%%
%%%%%%%%%%%%%%%%%%%%%%%%%%%%%%%%%%%%%%%%%%%%%%%%%%%%%%%%%%%
\section{Functional setting}\label{Functionalsetting-sec}
%%%%%%%%%%%%%%%%%%%%%%%%%%%%%%%%%%%%%%%%%%%%%%%%%%%%%%%%%%%
%%%%%%%%%%%%%%%%%%%%%%%%%%%%%%%%%%%%%%%%%%%%%%%%%%%%%%%%%%%

In this section,
we introduce a functional space for a fixed point theorem and
provide fundamental estimates for Section \ref{Shorttime-sec}\h-\h{Section} \ref{Exterior-sec}.
Here
we use the same transformation as in Section \ref{Heatkernel-sec}.
\[
 b(y,s) = \Phi(y,s)/\sigma(y),
\]
where $\sigma(y)$ is a function given at the beginning of Section \ref{Heatkernel-sec}.
Then $b(y,s)$ satisfies
\begin{equation}\label{Fullb-eq}
\begin{cases}
\dis
 b_s =
\frac{1}{\cal C}\nabla\left( {\cal C}\nabla b \right) + \left( \frac{\gamma-m}{2} \right)b,
 & (y,s)\in\R_+^n\times(s_1,\infty),
\\[2mm] \dis
 \pa_\nu b = f(\Phi)/\sigma,
 & (y,s)\in\pa\R_+^n\times(s_1,\infty),
\\[1mm]
 b(y,s_1) = b_1(y) := \Phi_0(y)/\sigma(y),
  & y\in\R_+^n,
\end{cases}
\end{equation}
where we replaced the initial time $s_T$ by $s_1$.
We recall that ${\cal K}=qV|_{\theta=\pi/2}^{q-1}$ and
\begin{equation}\label{f(Phi)-eq}
 f(\Phi) = (\Phi+U_\infty)^q-U_\infty^q-{\cal K}r^{-1}\Phi.
\end{equation}
Let $e^{{\cal A}s}$ be the semigroup of \eqref{Fullb-eq} with replaced its boundary condition by
zero Neumann boundary condition.
Therefore
$e^{{\cal A}(s-s_1)}$ is written by
\[
 e^{{\cal A}(s-s_1)} = e^{(\gamma-m)(s-s_1)/2}e^{{\cal A}_0(s-s_1)},
\]
where $e^{{\cal A}_0(s-s_1)}$ is the semigroup defined in Section \ref{Backward-sec}.
Furthermore
by Proposition \ref{Rep-pro},
$b(y,s)$ is expressed by
\begin{equation}\label{b(y,s)express-eq}
 b(y,s)
=
 e^{{\cal A}(s-s_1)}b_1 + \int_{s_1}^se^{{\cal B}(s-\tau)}F(\tau)d\tau,
\end{equation}
where $F(\xi,s) = f(\Phi(\xi,s))/\sigma(\xi')$ and
% $e^{{\cal A}s}$, $e^{{\cal B}s}$ are given by
\[
\begin{array}{c}
\dis
 e^{{\cal A}s}b_1
=
 \int_{\R_+^n}\Gamma(y,\xi,s)b_1(\xi){\cal B}(\xi)d\xi,
\hspace{7.5mm}
 e^{{\cal B}s}F
=
 \int_{\pa\R_+^n}\Gamma(y,\xi',s)F(\xi'){\cal B}(\xi')d\xi'
\end{array}
\]
with $\Gamma(y,\xi,s)=e^{(\gamma-m)s/2}\Theta(y,\xi,s)$.
%%%%%%%%%%%%%%%%%%%%%%%%%%%%%%%%%%%%%%%%%%%%%%%%%%%%%%%%%%%

Here we prepare notations for convenience.
Throughout this section,
we fix $\ell\in\N$ such that $\lambda_{1\ell}>0$ and put
\[
 \lambda^* = \lambda_{1\ell},
\hspace{10mm}
 \omega = \lambda^*/(\gamma-m)>0.
\]
Let $k_\alpha=\alpha^{-(\gamma-m)/m}k_1$ be defined in Theorem \ref{JLsuper-thm} (iv)
and $c_{1\ell}$  be defined in Lemma \ref{2A-lem}.
We choose $\alpha>0$ such that
\[
 k_\alpha=c_{1\ell}.
\]
We fix $K>0$ and $\sigma\in(0,1/2)$ such that
\[
 K = e^{a\omega s_1}\hspace{3mm} (0<a<1),
\hspace{10mm}
 \max\left\{\frac{\lambda^*}{2\lambda^*+1},\frac{1}{2q}\right\} < \sigma < 1/2.
\]
Moreover
we fix $H\in(0,K)$ and $\varrho \in(0,1/2)$ such that
\[
 H = e^{a'\omega s_1}\hspace{3mm} (0<a'\ll a),
\hspace{10mm}
 \sigma < \varrho < 1/2.
\]
Here
we do not give an explicit form of $a'$.
The constant $a'$ will be chosen sufficiently small
in Section \ref{Functionalsetting-sec}\h-\h{Section} \ref{Exterior-sec}
to obtain the desired estimates.
%%%%%%%%%%%%%%%%%%%%%%%%%%%%%%%%%%%%%%%%%%%%%%%%%%%%%%%%%%%
To apply a fixed point theorem,
we define a suitable functional space.
Let $c_{1\ell}>0$ be defined in Lemma \ref{2A-lem} and put $\epsilon_0=c_{1\ell}/4$.
Furthermore
let $m_0=m_0(q,n)$ be a constant given in \eqref{m_0-eq}.
We define
\begin{equation}\label{Assume1-eq}
\begin{cases}
\dis
 \varphi(y,s) < U_\infty(y)
& \text{for } r<Ke^{-\omega s_1},
\\[2mm]
 \left| \varphi(y,s)-U_\infty(y)+e^{-\lambda^*s}\phi_{1\ell}(y) \right|
<
 \epsilon_0e^{-\lambda^*s}e_1(\theta)\left( r^{-\gamma}+r^{2\lambda^*-m} \right)
&
 \text{for } r\in(Ke^{-\omega s},e^{\sigma s}),
\\[2mm]
 |\varphi(y,s)| < m_0U_\infty(y)
&
 \text{for } r>e^{\sigma s}.
\end{cases}
\end{equation}
Then
from Lemma \ref{2A-lem},
we recall that $\phi_{1\ell}(y)=(1+o(1))c_{1\ell}e_1(\theta)r^{-\gamma}$ for $r\ll1$.
Hence
if $\varphi(y,s)$ satisfies \eqref{Assume1-eq} for $s_1<s<s_2$,
then it holds that
\begin{equation}\label{AssumeMinus-eq}
 -\frac{3e^{-\lambda^*s}}{2}c_{1\ell}e_1(\theta)r^{-\gamma}
<
 \varphi(y,s) - U_\infty(y)
<
 -\frac{e^{-\lambda^*s}}{2}c_{1\ell}e_1(\theta)r^{-\gamma}
\end{equation}
for $r=Ke^{-\omega s}$ and $s_1<s<s_2$.
We will construct a solution $\varphi(y,s)$ of \eqref{varphi-eq} which belongs to
\[
 A_{s_1,s_2}
=
 \Bigl\{
 \varphi(y,s)\in C(\overline{\R_+^n}\times[s_1,s_2]);
 \varphi(y,s) \text{ satisfies } \text{\eqref{Assume1-eq} for } s\in(s_1,s_2) 
 \Bigr\}.
\]
For simplicity,
we define $\Pi\subset\N^2$ by
\[
 \Pi = \{(i,j)\in\N^2;\ \lambda_{ij}<\lambda^* \}.
\]
To construct such a solution,
we define $\phi_\ell^*(y,s_1)$ by
\[
 \phi_\ell^*(y,s_1)
=
\begin{cases}
 \dis
 e^{\lambda^*s_1}
 \left(
 U_\infty(y)-e^{m\omega  s_1}U_\alpha(e^{\omega  s_1}y)+
 \sum_{(i,j)\in\Pi}d_{ij}\phi_{ij}(y)
 \right)
 \\
 \hspace{75mm} \text{if } 0\leq r<He^{-\omega  s_1},
 \\ \dis
 e^{\lambda^*s_1}
 \Biggl(
 \Bigl( H+1-e^{\omega  s_1}r \Bigr)
 \Bigl( U_\infty(y)-e^{m\omega  s_1}U_\alpha(e^{\omega  s_1}y)+\sum_{(i,j)\in\Pi}d_{ij}\phi_{ij}(y) \biggr)
 \\[2mm] \hspace{20mm} \dis
 +\Bigl( e^{\omega  s_1}r-H \Bigr)e^{-\lambda^* s_1}\phi_{1\ell}(y)
 \Biggr)
\hspace{5mm} \text{if } He^{-\omega  s_1}<r<(H+1)e^{-\omega  s_1},
 \\[2mm] \dis
 \phi_{1\ell}(y)
 \hh \text{if } (H+1)e^{-\omega  s_1}<r<e^{\varrho s_1},
 \\[3mm] \dis
 e^{\lambda^*s_1}
 \left(
 \Bigl( r-e^{\varrho s_1} \Bigr)
 \Bigl( U_\infty(y)+\sum_{(i,j)\in\Pi}d_{ij}\phi_{ij}(y) \Bigr)
 +
 e^{-\lambda^*s_1}\Bigl( e^{\varrho s_1}+1-r \Bigr)\phi_{1\ell}(y)
 \right)
 \\[4mm]
 \hspace{80mm} \text{if } e^{\varrho s_1}<r<e^{\varrho s_1}+1,
 \\ \dis
 e^{\lambda^*s_1}
 \left( U_\infty(y)+\sum_{(i,j)\in\Pi}d_{ij}\phi_{ij}(y) \right)
 \hh \text{if } r>e^{\varrho s_1}+1.
\end{cases}
\]
Throughout this paper,
we denote $\phi_\ell^*(y,s_1)$ by $\phi_\ell^*(y)$ for simplicity.
We choose the initial data $\varphi(y,s_1)$ as
\[
 \varphi(y,s_1)-U_\infty(y) =
 \sum_{(i,j)\in\Pi}d_{ij}\phi_{ij}(y) - e^{-\lambda^*s_1}\phi_\ell^*(y).
\]
Then
by definition of $\phi_\ell^*$,
we easily see that $\varphi(y,s_1)\in BC(\overline{\R_+^n})$.
From now,
we denote by $\varphi(y,s)$
a classical solution of \eqref{varphi-eq} with the initial data $\varphi(y,s_1)$
and set
\[
 \Phi(y,s)=\varphi(y,s)-U_\infty(y),
\hspace{7.5mm}
 b(y,s) = \Phi(y,s)/\sigma(y).
\]
Furthermore
for simplicity of notations,
we put
\[
 d_{\text{max}}=\max_{(i,j)\in\Pi}|d_{ij}|.
\]

%%%%%%%%%%%%%%%%%%%%%%%%%%%%%%%%%%%%%%%%%%%%%%%%%%%%%%%%%%%
\subsection{Fixed point theorem}
%%%%%%%%%%%%%%%%%%%%%%%%%%%%%%%%%%%%%%%%%%%%%%%%%%%%%%%%%%%

Since $e_1(\theta)$ is strictly positive on $S_+^n$ and $\lambda_{ij}<\lambda^*$ for $(i,j)\in\Pi$,
from Lemma \ref{2A-lem},
there exists $c_1>0$ such that
\begin{equation}\label{c_1-eq}
 \sum_{(i,j)\in\Pi}|\phi_{ij}(y)|
\leq
 c_1e_1(\theta)\left( |y|^{-\gamma}+|y|^{2\lambda^*-m} \right)
\hspace{5mm}\text{for } y\in\R_+^n.
\end{equation}
Here
we fix $\epsilon_1=\epsilon_0/2c_1$.
Let $N$ be the number of elements of $\Pi$ and define
\[
 U_{s_1,s_2} = \left\{ d\in\R^N;\ |d|<\epsilon_1e^{-\lambda^*s_1},\ \varphi(y,s)\in A_{s_1,s_2} \right\}.
\]
For the case $U_{s_1,s_2}\not=\emptyset$,
we define an operator $P$: $U_{s_1,s_2}\to\R^N$ by
\begin{equation}\label{PDef-eq}
 P_{ij}(d;s_2) = (\varphi(s_2),\phi_{ij})_\rho
\hspace{5mm}\text{for } (i,j)\in\Pi.
\end{equation}
By the continuous dependence on the initial data,
we find $U_{s_1,s_2}$ is a open set in $\R^N$.
Furthermore
by definition of $U_{s_1,s_2}$,
it holds that $U_{s_1,s_2'}\subset U_{s_1,s_2}$ if $s_2<s_2'$.
Now
we will see that $P_{ij}(d;s_2)$ can be defined on $\bar{U}_{s_1,s_2}$ and $P_{ij}\in C(\bar{U}_{s_1,s_2})$.
In fact $\{d_j\}_{j\in\N}\subset U_{s_1,s_2}$ be a sequence converging to $d_*\in\pa U_{s_1,s_2}$.
Let $\varphi_j(y,s)$ and $\varphi_*(y,s)$ be a solution corresponding to $d=d_j$ and $d=d_*$, respectively.
Then
since $\varphi_j(y,s)\in A_{s_1,s_2}$,
by using Lemma \ref{underUinfty-lem},
we easily see that for $s_1<s<s_2$
\begin{equation}\label{varphi_j-eq}
 \varphi_j(y,s)
<
\begin{cases}
 U_{\beta(s)}
& \text{for } |y|<Ke^{-\omega s},
\\[2mm]
 U_\infty(y) - e^{-\lambda s}\phi_{1\ell}(y) +
 \epsilon_0e^{-\lambda^*s}e_1(\theta)\left( r^{-\gamma}+r^{2\lambda^*-m} \right)
& \text{for } Ke^{-\omega s}<|y|<e^{\sigma s},
\\[2mm]
 m_0U_\infty(y)
& \text{for } |y|>e^{\sigma s}.
\end{cases}
\end{equation}
Therefore
by the continuous dependence on the initial data and the unique solvability of solutions,
$\varphi_*(y,s)$ turns out to be defined on $s\in(s_1,s_2)$ and
satisfies \eqref{varphi_j-eq} with replaced `$<$' by `$\leq$'.
As a consequence,
$P_{ij}(d;s_2)$ can be defined for $d\in\pa U_{s_1,s_2}$.
Furthermore
we easily see that $P_{ij}(d;s_2)\in C(\bar{U}_{s_1,s_2})$.
To apply a fixed point theorem,
the following proposition plays a crucial role.

%%%%%%%%%%%%%%%%%%%%%%%%%%%%%%%%%%%%%%%%%%%%%%%%%%%%%%%%%%%
\begin{pro}\label{Main-pro}
There exists $s_1>0$ such that
if $P(d;s_2)=0$ with $d\in\bar{U}_{s_1,s_2}\not=\emptyset$ for $s_2>s_1$,
then it holds that for $s_1<s<s_2$
\[
\begin{cases}
\dis
 \varphi(y,s) < U_{\beta(s)}(y)
& \mathrm{for}\ r<Ke^{-\omega s_1},
\\[2mm] \dis
 \left| \varphi(y,s)-U_\infty(y)+e^{-\lambda^*s}\phi_{1\ell}(y) \right|
<
 \frac{\epsilon_0}{2}e^{-\lambda^*s}e_1(\theta)\left( r^{-\gamma}+r^{2\lambda^*-m} \right)
&
 \mathrm{for}\ r\in(Ke^{-\omega s},e^{\sigma s}),
\\[2mm] \dis
 |\varphi(y,s)| < \left( \frac{1+3m_0}{4} \right)U_\infty(y)
&
 \mathrm{for}\ r>e^{\sigma s},
\end{cases}
\]
where $\beta(s)=\beta_0e^{m\omega s}$ and $\beta_0>\alpha$ is a constant given in
{\rm Lemma \ref{underUinfty-lem}}.
\end{pro}
%%%%%%%%%%%%%%%%%%%%%%%%%%%%%%%%%%%%%%%%%%%%%%%%%%%%%%%%%%%

%%%%%%%%%%%%%%%%%%%%%%%%%%%%%%%%%%%%%%%%%%%%%%%%%%%%%%%%%%%
\begin{proof}
This proposition follows from Lemma \ref{underUinfty-lem},
Proposition \ref{Shorttime-pro}, Proposition \ref{Longtime1-pro}\h-\h\ref{Longtime3-pro}
and Proposition \ref{Exterior1-pro}\h-\h\ref{Exterior2-pro}.
\end{proof}
%%%%%%%%%%%%%%%%%%%%%%%%%%%%%%%%%%%%%%%%%%%%%%%%%%%%%%%%%%%

As a consequence,
we obtain the following main result.

%%%%%%%%%%%%%%%%%%%%%%%%%%%%%%%%%%%%%%%%%%%%%%%%%%%%%%%%%%%
\begin{pro}\label{Main2-pro}
Let $s_1>0$ be as in {\rm Proposition \ref{Main-pro}}.
Then
there exists $d^*\in\{d\in\R^N;\ |d|<\epsilon_1\}$ such that
a solution $\varphi(y,s)$ corresponding to $d=d^*$
is defined on $(s_1,\infty)$ and satisfies $\varphi(y,s)\in A_{s_1,s_2}$ for any $s_2>s_1$.
\end{pro}
%%%%%%%%%%%%%%%%%%%%%%%%%%%%%%%%%%%%%%%%%%%%%%%%%%%%%%%%%%%

To show the above proposition,
we prepare several lemmas.
Here
we put
\[
 D_{s_1,s_2}= \left\{(d,s);\ d\in U_{s_1,s},\ s_1<s<s_2 \right\},
\hspace{7.5mm}
 D_{s_1,s_2}[s]= \left\{d;\ (d,s)\in D_{s_1,s_2} \right\}.
\]
Then
it is clear that $D_{s_1,s_2}[s]=U_{s_1,s}$.
Here
we will see that
$D_{s_1,s_2}$ is a open set in $\R^N\times(s_1,s_2)$.
In fact,
let $(d,s)\in D_{s_1,s_2}$.
Then
by the continuous dependence on the initial data and
by the continuity of solutions,
there exists $\delta>0$ such that $(d-\delta,d+\delta)\times(s,s+\delta)\subset D_{s_1,s_2}$.
Furthermore
since $U_{s_1,s'}\subset U_{s_1,s}$ if $s'>s$,
it holds that $(d-\delta,d+\delta)\times(s_1,s)\subset D_{s_1,s_2}$,
which assures the claim.
Let $P$ be the operator defined by \eqref{PDef-eq}.
Then
$P$ can be naturally considered as an operator $D_{s_1,s_2}\to\R^N$.
We define its map by $(d,s)\longmapsto P(d,s)$.
Then
we find that $P:D_{s_1,s_2}\to\R^N$ is continuous.
Furthermore
by the same reason as above,
$P$ can be defined on $\bar{D}_{s_1,s_2}$ and $P\in C(\bar{D}_{s_1,s_2})$.

%%%%%%%%%%%%%%%%%%%%%%%%%%%%%%%%%%%%%%%%%%%%%%%%%%%%%%%%%%%
\begin{lem}\label{1-lem}
Let $s_1>0$ be given in {\rm Proposition \ref{Main-pro}} and $U_{s_1,s_2}\not=\emptyset$.
Then
$P(d;\tau)=0$ has no solutions on $\pa U_{s_1,\tau}$ for any $\tau\in[s_1,s_2]$.
\end{lem}
%%%%%%%%%%%%%%%%%%%%%%%%%%%%%%%%%%%%%%%%%%%%%%%%%%%%%%%%%%%

%%%%%%%%%%%%%%%%%%%%%%%%%%%%%%%%%%%%%%%%%%%%%%%%%%%%%%%%%%%
\begin{proof}
Let $P(d_0;\tau)=0$ and $\varphi(y,s)$ be a solution corresponding to $d=d_0$.
Then
from Proposition \ref{Main-pro},
we see that $\varphi(y,s)$ satisfies estimates given in Proposition \ref{Main-pro}
with replaced $s_2$ by $\tau$.
Hence
if $|d'-d_0|$ is small enough,
a solution corresponding to $d=d'$ belongs to $A_{s_1,\tau}$.
Furthermore
from Lemma \ref{d_max-lem},
we note that $|d|<\epsilon_1$.
Therefore
it follows that $d_0\in U_{s_1,\tau}$.
Thus the proof is completed.
\end{proof}
%%%%%%%%%%%%%%%%%%%%%%%%%%%%%%%%%%%%%%%%%%%%%%%%%%%%%%%%%%%

%%%%%%%%%%%%%%%%%%%%%%%%%%%%%%%%%%%%%%%%%%%%%%%%%%%%%%%%%%%
\begin{lem}\label{2-lem}
Let $s_1>0$ be given in {\rm Proposition \ref{Main-pro}}.
Then
if $U_{s_1,s_2}\not=\emptyset$ for some $s_2>s_1$,
then 
there exists $d\in U_{s_1,s_2}$ such that $P(d;s_2)=0$.
\end{lem}
%%%%%%%%%%%%%%%%%%%%%%%%%%%%%%%%%%%%%%%%%%%%%%%%%%%%%%%%%%%

%%%%%%%%%%%%%%%%%%%%%%%%%%%%%%%%%%%%%%%%%%%%%%%%%%%%%%%%%%%
\begin{proof}
We recall that $P$: $\bar{D}_{s_1,s_2}\to\R^N$ is continuous
and $P(d;\tau)=0$ has no solutions on $\pa D_{s_1,s_2}[\tau]$ for any $\tau\in[s_1,s_2]$
(see Lemma \ref{1-lem}).
Then
we can define deg$(P(d;\tau),D_{s_1,s_2}[\tau],0)$ for $\tau\in[s_1,s_2]$.
By the homotopy invariance property (see Theorem 2.2.4 \cite{Lloyd}),
we get
\[
 \text{deg}(P(d;s_2),D_{s_1,s_2}[s_2],0) = \text{deg}(P(d;s_1),D_{s_1,s_2}[s_1],0).
\]
Since $D_{s_1,s_2}[s]=U_{s_1,s}$,
this implies
$\text{deg}(P(d;s_2),U_{s_1,s_2},0)=\text{deg}(P(d;s_1),U_{s_1,s_1},0)$.
Here
we recall that $\epsilon_1=\epsilon_0/2c_1$,
Then
from Lemma \ref{phi-initial-lem} and Lemma \ref{underUinfty-lem},
we find that $U_{s_1,s_1}=\{d\in\R^N;\ |d|<\epsilon_1e^{-\lambda^*s_1}\}$.
Furthermore
by definition of $\varphi(y,s_1)$,
it holds that for $(i,j)\in\Pi$
\[
 P_{ij}(d;s_1)
=
 d_{ij} - (e^{-\lambda^*s_1}\phi_\ell^*,\phi_{ij})_\rho
=
 d_{ij} - e^{-\lambda^*s_1}(\phi_\ell^*-\phi_{1\ell},\phi_{ij})_\rho.
\]
Then
it follows from Lemma \ref{(*-ell)-lem} that
\[
 |P_{ij}(d;s_1)-d_{ij}| \leq ce^{-(\lambda^*+\delta)s_1}.
\]
This implies
$|P(d;s_1)-\text{Id}(d)| \leq ce^{-(\lambda^*+\delta)s_1}$.
Hence
by the homotopy invariance property (see Theorem 2.1.2 (2) \cite{Lloyd}),
we obtain
\[
 \text{deg}(P(d;s_2),U_{s_1,s_2},0)
=
 \text{deg}(P(d;s_1),U_{s_1,s_1},0)
=
 \text{deg}(\text{Id},B_{\epsilon_1},0),
\]
 where $B_{\epsilon_1}=\{d\in\R^N;\ |d|<\epsilon_1\}$.
Therefore
there exists $d\in U_{s_1,s_2}$ such that $P(d;s_2)=0$,
which completes the proof.
\end{proof}
%%%%%%%%%%%%%%%%%%%%%%%%%%%%%%%%%%%%%%%%%%%%%%%%%%%%%%%%%%%

%%%%%%%%%%%%%%%%%%%%%%%%%%%%%%%%%%%%%%%%%%%%%%%%%%%%%%%%%%%
\begin{lem}\label{3-lem}
Let $s_1>0$ be given in {\rm Proposition \ref{Main-pro}}.
Then
it holds that $U_{s_1,s_2}\not=\emptyset$ for any $s_2>s_1$.
\end{lem}
%%%%%%%%%%%%%%%%%%%%%%%%%%%%%%%%%%%%%%%%%%%%%%%%%%%%%%%%%%%

%%%%%%%%%%%%%%%%%%%%%%%%%%%%%%%%%%%%%%%%%%%%%%%%%%%%%%%%%%%
\begin{proof}
Fix $s_2>s_1$.
Since $\varphi(y,s_1)$ satisfies estimates given in Lemma \ref{phi-initial-lem},
by the continuous dependence on the initial data,
there exists $\delta>0$ such that
$U_{s_1,s_1+\delta}=\{d\in\R^N;\ |d|<\epsilon_1e^{-\lambda^*s_1}\}$.
Now
we define
\[
 s^*=\sup\left\{ s>s_1;\ U_{s_1,s}\not=\emptyset \right\}.
\]
Suppose $s^*<s_2$.
Then
there exists a sequence $\{\tau_j\}_{j\in\N}\subset(s_1,s^*)$
such that $\tau_j\to s^*$ and $U_{s_1,\tau_j}\not=\emptyset$.
By Lemma \ref{2-lem},
there exists $d_j\in U_{s_1,\tau_j}$ such that $P(d_j,\tau_j)=0$.
Let $\varphi_j(y,s)$ be a solution corresponding to $d=d_j$.
Then
from Proposition \ref{Main-pro},
$\varphi_j(y,s)$ satisfies estimates given in Proposition \ref{Main-pro}
with replaced $s_2$ by $\tau_j$.
Hence
by a parabolic regularity theory,
we see that
$\varphi_j(y,s)$ converges to some function $\varphi^*(y,s)$
satisfying estimates given in Proposition \ref{Main-pro}
with replaced $s_2$ by $s^*$ and with replaced `$<$' by `$\leq$'.
Then
by the continuity of $\varphi^*(y,s)$ in time,
we find that
there exists $\delta>0$ such that $\varphi^*(y,s)\in A_{s_1,s^*+\delta}$.
However
this contradicts the assumption,
which completes the proof.
\end{proof}
%%%%%%%%%%%%%%%%%%%%%%%%%%%%%%%%%%%%%%%%%%%%%%%%%%%%%%%%%%%

%%%%%%%%%%%%%%%%%%%%%%%%%%%%%%%%%%%%%%%%%%%%%%%%%%%%%%%%%%%
\begin{proof}[\bf Proof of Proposition \ref{Main2-pro}]
Let $s_1>0$ be given in Proposition \ref{Main-pro} and put $\tau_j=s_1+j$.
Then
Lemma \ref{3-lem} implies $U_{s_1,\tau_j}\not=\emptyset$.
Therefore
from Lemma \ref{2-lem},
there exists $d_j\in U_{s_1,\tau_j}$ such that $P(d_j;\tau_j)=0$.
We denote by $\varphi_j(y,s)$ a solution corresponding to $d=d_j$.
Then
from Proposition \ref{Main-pro},
there exist a limiting function
$\varphi^*(y,s)\in C(\overline{\R_+^n}\times[s_1,\infty))$
and
a subsequence $\{d_j\}_{j\in\N}$
which is denoted by the same symbol such that
$\varphi_j(y,s)\to\varphi^*(y,s)$ uniformly on any compact set in $\overline{\R_+^n}\times[s_1,\infty)$.
Then
this $\varphi^*(y,s)$ turns out to be the desired solution 
satisfying estimates given in Proposition \ref{Main-pro} with replaced $s_2$ by $\infty$.
Therefore
the proof is completed.
\end{proof}
%%%%%%%%%%%%%%%%%%%%%%%%%%%%%%%%%%%%%%%%%%%%%%%%%%%%%%%%%%%

%%%%%%%%%%%%%%%%%%%%%%%%%%%%%%%%%%%%%%%%%%%%%%%%%%%%%%%%%%%
\begin{proof}[\bf Proof of Theorem \ref{1-thm}]
Let $s_1$ and $\varphi(y,s)$ be given in Proposition \ref{Main2-pro}.
Here
we recall that $U_\beta(y)<U_\beta(0)=\beta$ for $y\in\R_+^n$.
Therefore
by Proposition \ref{Main2-pro},
we see that
\[
 \varphi(y,s) < U_{\beta(s)}(y) \leq \beta(s)=\beta_0e^{m\omega s}
\hspace{5mm}\text{for }|y|<Ke^{-\omega s}.
\]
Furthermore
by using
$|\phi_{1\ell}(y)|<c(|y|^{-\gamma}+|y|^{2\lambda^*-m})$
(see Lemma \ref{2A-lem})
and
$\lambda^*=(\gamma-m)\omega$,
we get
\[
\begin{array}{lll}
\dis
 \varphi(y,s)
\hspace{-2mm}&\leq&\hspace{-2mm} \dis
 U_\infty(y) + ce^{-\lambda^*s}\left( |y|^{-\gamma}+|y|^{2\lambda^*-m} \right)
\\[4mm]
\hspace{-2mm}&\leq&\hspace{-2mm} \dis
 cK^{-m}e^{m\omega s}
+
 c\begin{cases}
 K^{-\gamma}e^{m\omega s} & \text{for } Ke^{-\omega s}<|y|<1,\\
 e^{-\sigma m s}e^{-(1-2\sigma)\lambda^*s} & \text{for } 1<|y|<e^{\sigma s}. 
 \end{cases}
\end{array}
\]
Finally
we provide estimates in the range $|y|>e^{\sigma s}$.
Then
Proposition \ref{Main2-pro} implies
\[
 \varphi(y,s) \leq cr^{-m} \leq ce^{-m\sigma s}
\hspace{5mm}\text{for } |y|>e^{\sigma s}.
\]
As a consequence,
we conclude
\[
 \|\varphi(s)\|_\infty \leq ce^{m\omega s}.
\]
Furthermore
from \eqref{AssumeMinus-eq} and $\lambda^*=(\gamma-m)\omega$,
we obtain
\[
\begin{array}{lll}
\dis
 \|\varphi(s)\|_\infty
\hspace{-2mm}&\geq&\hspace{-2mm} \dis
 \varphi(y,s)|_{|y|=Ke^{-\omega s}}
\geq
 U_\infty(y) - ce^{-\lambda^*s}e_1(\theta)r^{-\gamma}
\\[2mm]
\hspace{-2mm}&\geq&\hspace{-2mm} \dis
 K^{-m}V(\theta)e^{m\omega s} - \frac{3}{2}K^{-\gamma}c_{1\ell}e_1(\theta)e^{m\omega s}.
\end{array}
\]
Since $K\gg1$,
from the above estimates,
there exist $c_1>0$ and $c_2>0$ such that
\[
 c_1e^{m\omega s} \leq \|\varphi(s)\|_\infty \leq c_2e^{m\omega s},
\]
which completes the proof.
\end{proof}
%%%%%%%%%%%%%%%%%%%%%%%%%%%%%%%%%%%%%%%%%%%%%%%%%%%%%%%%%%%

%%%%%%%%%%%%%%%%%%%%%%%%%%%%%%%%%%%%%%%%%%%%%%%%%%%%%%%%%%%
\subsection{Fundamental estimates}
%%%%%%%%%%%%%%%%%%%%%%%%%%%%%%%%%%%%%%%%%%%%%%%%%%%%%%%%%%%

%%%%%%%%%%%%%%%%%%%%%%%%%%%%%%%%%%%%%%%%%%%%%%%%%%%%%%%%%%%
\begin{lem}\label{phi-initial-lem}
There exists a continuous function $\nu(s)$ on $\R_+$ satisfying $\lim_{s\to\infty}\nu(s)=0$
such that if $|d|<\epsilon_1e^{-\lambda^*s_1}$,
then it holds that
\[
\begin{array}{c}
 \Phi(y,s_1)
= 
 -U_\infty(y) + e^{m\omega s_1}U_\alpha(e^{\omega s_1}y)
\hspace{5mm} \mathrm{for}\ 0<r<He^{-\omega s_1},
\\[3mm]
 \left| \Phi(y,s_1)+e^{-\lambda^*s_1}\phi_{1\ell}(y) \right|
\leq
 \left( \nu(s_1)e^{-\lambda^*s_1}+c_1d_{\max} \right)e_1(\theta)
 (r^{-\gamma}+r^{2\lambda^*-m})
\\[2mm] \dis \hspace{75mm}
 \mathrm{for }\ He^{-\omega s_1}<|y|<e^{\varrho s_1},
\\[2mm] \dis
 -U_\infty(y) \leq \Phi(y,s) \leq ce^{-(1-2\varrho)\lambda^*s_1}U_\infty(y)
\hspace{5mm} \mathrm{for}\ r>e^{\varrho s_1}.
\end{array}
\]
\end{lem}
%%%%%%%%%%%%%%%%%%%%%%%%%%%%%%%%%%%%%%%%%%%%%%%%%%%%%%%%%%%

%%%%%%%%%%%%%%%%%%%%%%%%%%%%%%%%%%%%%%%%%%%%%%%%%%%%%%%%%%%
\begin{proof}
The equality in $0<r<He^{-\omega s_1}$ follows directly from definition of $\phi_\ell^*$.
Next
we provide the estimate in $He^{-\omega s_1}<r<e^{\varrho s_1}$.
We recall that
$U_\alpha(y)=U_\infty(y)-k_\alpha(1+o(1))e_1(\theta)r^{-\gamma}$ for $r\gg1$
(see Theorem \ref{JLsuper-thm})
and
$\phi_{1\ell}(y)=c_{1\ell}e_1(\theta)r^{-\gamma}$ for $r\ll1$
(see Lemma \ref{2A-lem}).
Therefore
since $k_\alpha=c_{1\ell}$,
we see that
\[
\begin{array}{lll}
\dis
 e^{-\lambda^*s_1}\phi_\ell^*
\hspace{-2mm}&=&\hspace{-2mm} \dis
 (H+1-e^{\omega s_1}r)
 \left( k_\alpha(1+o(1))e^{-\lambda^*s_1}e_1(\theta)r^{-\gamma}+\sum_{(i,j)\in\Pi}d_{ij}\phi_{ij} \right)
\\[6mm] \dis
\hspace{-2mm}&&\hspace{5mm} \dis
 +
 (e^{\omega s_1}r-H)c_{1\ell}(1+o(1))e^{-\lambda^*s_1}e_1(\theta)r^{-\gamma}
\\
\hspace{-2mm}&=&\hspace{-2mm} \dis
 c_{1\ell}(1+o(1))e^{-\lambda^*s_1}e_1(\theta)r^{-\gamma}
+
 (H+1-e^{\omega s_1}r)
 \left( \sum_{(i,j)\in\Pi}d_{ij}\phi_{ij} \right)
\end{array}
\]
for $He^{-\omega s_1}<r<(H+1)e^{-\omega s_1}$.
Furthermore
by definition of $\phi_\ell^*$,
it follow that $e^{-\lambda^*s_1}\phi_\ell^*=e^{-\lambda^*s_1}\phi_{1\ell}$
for $(H+1)e^{-\omega s_1}<r<e^{\varrho s_1}$.
Therefore
since $\Phi(y,s_1)=\sum_{(i,j)\in\Pi}d_{ij}\phi_{ij}-e^{-\lambda^*s_1}\phi_\ell^*$,
by using \eqref{c_1-eq},
we obtain the estimate in $He^{-\omega s_1}<r<e^{\varrho s_1}$.
Finally
we consider the region in $r>e^{\varrho s_1}$.
Since $\varphi(y,s_1)=U_\infty+\sum_{(i,j)\in\Pi}d_{ij}\phi_{ij}-e^{-\lambda\*s_1}\phi_\ell^*$,
we see that
\[
 \varphi(y,s_1) = (1-r+e^{\varrho s_1})
 \left(
 U_\infty + \sum_{(i,j)\in\Pi}d_{ij}\phi_{ij} - e^{-\lambda^*s_1}\phi_{1\ell}
 \right)
\]
for $e^{\varrho s_1}<r<e^{\varrho s_1}+1$.
Then since
$\phi_{ij}(y)\sim r^{2\lambda_{ij}-m}$ for $r\gg1$ (see Lemma \ref{2A-lem}).
we observe that
$e^{-\lambda^*s_1}\sum_{(i,j)\in\Pi}(|\phi_{ij}|+|\phi_{1\ell}|)\leq ce^{-(1-2\varrho)\lambda^*s_1}r^{-m}$
for $e^{\varrho s_1}<r<e^{\varrho s_1}+1$.
Therefore
since $0<\varrho<1/2$,
we get if $|d|<\epsilon_1e^{-\lambda^*s_1}$
\[
 0 \leq \varphi(y,s_1) \leq \left( 1+ce^{-(1-2\varrho)\lambda^*s_1} \right)U_\infty
\]
for $e^{\varrho s_1}<r<e^{\varrho s_1}+1$.
Furthermore
by definition of $\phi_\ell^*$,
we find $\varphi(y,s_1)=0$ for $r>e^{\varrho s_1}+1$.
Combining these estimates,
we obtain the estimates in $r>e^{\varrho s_1}$,
which completes the proof.
\end{proof}
%%%%%%%%%%%%%%%%%%%%%%%%%%%%%%%%%%%%%%%%%%%%%%%%%%%%%%%%%%%

%%%%%%%%%%%%%%%%%%%%%%%%%%%%%%%%%%%%%%%%%%%%%%%%%%%%%%%%%%%
\begin{lem}\label{underUinfty-lem}
There exist $\beta_0>\alpha$ and $s_0>0$ such that
if $s_1>s_0$, $|d|<\epsilon_1e^{-\lambda^*s_1}$ and $\varphi(y,s)\in A_{s_1,s_2}$,
then it holds that for $s\in(s_1,s_2)$
\[
 \varphi(y,s) \leq U_{\beta(s)}(y)
\hspace{5mm}\mathrm{if}\ r<Ke^{-\omega s}
\]
with $\beta(s)=\beta_0e^{\omega s}$.
\end{lem}
%%%%%%%%%%%%%%%%%%%%%%%%%%%%%%%%%%%%%%%%%%%%%%%%%%%%%%%%%%%

%%%%%%%%%%%%%%%%%%%%%%%%%%%%%%%%%%%%%%%%%%%%%%%%%%%%%%%%%%%
\begin{proof}
Since $\varphi(y,s)\in A_{s_1,s_2}$,
we recall from \eqref{AssumeMinus-eq} that
\begin{equation}\label{As-eq}
 \varphi(y,s) \leq U_\infty(y)-\frac{e^{-\lambda^*s_1}}{2}c_{1\ell}e_1(\theta)r^{-\gamma}
\hspace{5mm}\text{for } |y|=Ke^{-\omega s},\ s_1<s<s_2.
\end{equation}
Furthermore
since
$e^{m\omega s_1}U_\alpha(e^{\omega s_1}y)=U_{\alpha_1}(y)$
with $\alpha_1=\alpha e^{m\omega s_1}$
(see Theorem \ref{JLsuper-thm}),
Lemma \ref{phi-initial-lem} implies
\[
 \varphi(y,s_1)= U_{\alpha_1}(y)
\hspace{5mm}\text{for } r<He^{-\omega s_1}
\]
with $\alpha_1=\alpha e^{m\omega s_1}$.
Since $\phi_{1\ell}(y)=c_{1\ell}(1+o(1))e_1(\theta)r^{-\gamma}$ for $r\ll1$ (see Lemma \ref{2A-lem}),
from Lemma \ref{phi-initial-lem},
we see that
$\varphi(y,s_1)\leq U_\infty(y)+e^{-\lambda^*s_1}(-c_{1\ell}+o(1)+\nu(s_1)+
 c_1d_{\max}e^{\lambda^*s_1})e_1(\theta)r^{-\gamma}$
for $He^{-\omega s_1}<r<Ke^{-\omega s_1}$.
Then
since $|d|<\epsilon_1e^{-\lambda^*s_1}$ with $\epsilon_1=c_{1\ell}/8c_1$,
there exists $s_0>0$ such that if $s_1>s_0$,
then it holds that
\[
 \varphi(y,s)
\leq
 U_\infty(y)-\frac{e^{-\lambda^*s_1}}{2}c_{1\ell}e_1(\theta)r^{-\gamma}
\hspace{5mm}\text{for }
 He^{-\omega s_1}<r<Ke^{-\omega s_1}.
\]
We fix $\beta'>\alpha$ such that
$k_{\beta'}<k_\alpha/2$ (see Theorem \ref{JLsuper-thm} (iv)).
Then
since $k_\alpha=c_{1\ell}$,
there exists $s_0'>0$ such that if $s>s_1>s_0'$
\begin{equation}\label{As2-eq}
 e^{m\omega s}U_\beta(e^{\omega s}y)
\geq
 U_\infty(y) - \frac{e^{-\lambda^*s}}{2}c_{1\ell}e_1(\theta)r^{-\gamma}
\hspace{5mm} \text{for } r>He^{-\omega s}.
\end{equation}
Put $\beta_1=\beta'e^{m\omega s_1}$.
Then
from Theorem \ref{JLsuper-thm},
we observe that $e^{m\omega s_1}U_\beta(e^{\omega s_1}y)=U_{\beta_1}(y)$.
Since $\beta_1>\alpha_1$,
Theorem \ref{JLsuper-thm} implies $U_{\beta_1}(y)>U_{\alpha_1}(y)$ for $y\in\R_+^n$,
which implies
\begin{equation}\label{As3-eq}
 \varphi(y,s_1) \leq U_{\beta_1}(y)
\hspace{5mm} \text{for } |y|<Ke^{-\omega s_1}.
\end{equation}
Here
we put $Z(y,s)=U_{\beta(s)}(y)$ with $\beta(s)=\beta'e^{m\omega s}$.
From Theorem \ref{JLsuper-thm},
we note that $\pa_sU_{\beta(s)}(y)\geq0$ for $(y,s)\in\R_+^n\times(s_1,\infty)$. 
Furthermore
we put $\Lambda(\epsilon)=1+\epsilon$,
then we get from Theorem \ref{JLsuper-thm}
\[
 \pa_\epsilon U_{\Lambda(\epsilon)}(y) =
 \pa_\epsilon\left( \Lambda(\epsilon)U_1\left( \Lambda(\epsilon)^{1/m}y \right) \right) =
 U_1\left( \Lambda^{1/m}y \right) +
 \frac{\Lambda^{1/m}}{m}y\cdot(\nabla U_1)\left( \Lambda^{1/m}y \right).
\]
Since $\pa_\epsilon U_{\Lambda(\epsilon)}(y)\geq0$ for $y\in\R_+^n$,
taking $\epsilon=0$,
we find
\[
 y\cdot\nabla U_1(y) + mU_1(y) \geq 0
\hspace{5mm}\text{for } y\in\R_+^n.
\]
Since $U_\alpha(y)=\alpha U_1(\alpha^{1/m}y)$,
$U_\alpha(y)$ also satisfies the above inequality for any $\alpha>0$. 
Therefore
it holds that
\[
\begin{cases}
 \dis
 Z_s - \Delta Z +\frac{y}{2}\cdot\nabla Z + \frac{m}{2}Z \geq0
& \text{for } (y,s)\in\R_+^n\times(s_1,\infty),
 \\
 \pa_\nu Z = Z^q & \text{for }(y,s)\in\pa\R_+^n\times(s_1,\infty).
\end{cases}
\]
Applying a comparison lemma in $\{(y,s);\ |y|<Ke^{-\omega s},\ s_1<s<s_2\}$
with \eqref{As-eq}-\eqref{As3-eq},
we obtain the conclusion.
\end{proof}
%%%%%%%%%%%%%%%%%%%%%%%%%%%%%%%%%%%%%%%%%%%%%%%%%%%%%%%%%%%

%%%%%%%%%%%%%%%%%%%%%%%%%%%%%%%%%%%%%%%%%%%%%%%%%%%%%%%%%%%
\begin{lem}\label{Tana-eq}
There exists $c>0$ such that
if $|d|<\epsilon_1e^{-\lambda^*s_1}$ and $\varphi(y,s)\in A_{s_1,s_2}$,
then it holds that for $s_1<s<s_2$
\[
 |\Phi(y,s)| \leq c
 \begin{cases}
 \min\left\{ r^{-m},e^{-\lambda^*s}r^{-\gamma} \right\}
 & \mathrm{if}\ r<Ke^{-\omega  s},\\
 e^{-\lambda^*s}r^{-\gamma}
 & \mathrm{if}\ Ke^{-\omega  s}<r<1,\\
 e^{-\lambda^*s}r^{2\lambda^*-m}
 & \mathrm{if}\ 1<r<e^{\sigma s},\\
 r^{-m} & \mathrm{if}\ r>e^{\sigma s}.
 \end{cases}
\]
\end{lem}
%%%%%%%%%%%%%%%%%%%%%%%%%%%%%%%%%%%%%%%%%%%%%%%%%%%%%%%%%%%

%%%%%%%%%%%%%%%%%%%%%%%%%%%%%%%%%%%%%%%%%%%%%%%%%%%%%%%%%%%
\begin{proof}
By definition of $A_{s_1,s_2}$,
it is sufficient to show the estimate in $|y|<Ke^{-\omega s}$.
Since $\Phi(y,s)$ satisfies \eqref{Phi-eq} and $f(\Phi)>0$,
we observe that
$\pa_s\Phi = \Delta\Phi-\frac{y}{2}\cdot\nabla\Phi-m\Phi$ 
in $\R_+^n\times(s_1,\infty)$
and
$\pa_\nu\Phi\geq Kr^{-1}\Phi$
on $\pa\R_+^n\times(s_1,\infty)$.
To construct a sub-solution,
we put
\[
 \zeta_\kappa(y,s) = -\kappa e^{-\lambda^*s}\phi_{1\ell}(y).
\]
Then
it is clear that
$\zeta_\kappa(y,s)$ satisfies
$\pa_s\zeta_\kappa=\Delta\zeta_\kappa-\frac{y}{2}\cdot\nabla\zeta_\kappa-m\zeta_\kappa$ 
in $\R_+^n\times(s_1,\infty)$
and
$\pa_\nu\zeta_\kappa=Kr^{-1}\zeta_\kappa$
on $\pa\R_+^n\times(s_1,\infty)$.
Here
we will apply a comparison argument in the parabolic cylinder $|y|<Ke^{-\omega s}$, $s>s_1$.
Since $\varphi(y,s_1)=e^{m\omega s_1}U_\alpha(e^{\omega s_1}y)$ for $r<He^{-\omega s_1}$,
by Theorem \ref{JLsuper-thm},
there exists $r_0>0$ such that
$\varphi(y,s_1)\geq U_\infty(y)-2k_\alpha e^{-\lambda^*s_1}e_1(\theta)|y|^{-\gamma}$
for $e^{\omega s_1}|y|>r_0$.
Furthermore
from Lemma \ref{phi-initial-lem} with $|d|<\epsilon_1e^{-\lambda^*s_1}$,
it holds that
$\Phi(y,s_1)>-(1+\nu(s_1)+c_1\epsilon_1)e^{-\lambda^*s_1}e_1(\theta)|y|^{-\gamma}$
for $He^{-\omega s_1}<|y|<Ke^{-\omega s_1}$.
Therefore
since $\phi_{1\ell}(y)=c_{1\ell}(1+o(1))e_1(\theta)|y|^{-\gamma}$ for $|y|\ll1$ (see Lemma \ref{2A-lem}),
there exists $\kappa_0>0$ such that if $\kappa>\kappa_0$,
then it holds that $\Phi(y,s_1)>\zeta_\kappa(y,s_1)$
for $r_0e^{-\omega s_1}<|y|<Ke^{-\omega s_1}$. 
Furthermore
by using $\phi_{1\ell}(y)=c_{1\ell}(1+o(1))e_1(\theta)|y|^{-\gamma}$ for $|y|\ll1$ again,
% there exists $c'>0$ such that $\phi_{1\ell}(y)>c'|y|^{-\gamma}$ for $|y|<r_0$.
we get
% This implies 
\[
\begin{array}{lll}
\dis
 e^{-\lambda^*s_1}\phi_{1\ell}(y)
\hspace{-2mm}&\geq&\hspace{-2mm} \dis
 c_{1\ell}(1+o(1))e^{-\lambda^*s_1}e_1(\theta)|y|^{-\gamma}
\\[2mm]
\hspace{-2mm}&\geq&\hspace{-2mm} \dis
 c_{1\ell}(1+o(1))e^{-\lambda^*s_1}e_1(\theta)
 \left( r_0e^{-\omega s_1} \right)^{-(\gamma-m)}|y|^{-m}
\\[2mm]
\hspace{-2mm}&=&\hspace{-2mm} \dis
 c_{1\ell}r_0^{-(\gamma-m)}(1+o(1))e_1(\theta)|y|^{-m}
\hspace{7.5mm}\text{for } |y|<r_0e^{-\omega s_1}.
\end{array}
\]
Hence
there exists $\kappa_1>0$ such that if $\kappa>\kappa_1$,
then it holds that
$\zeta_\kappa(y,s_1)<-U_\infty(y)$ for $|y|<r_0e^{-\omega s_1}$.
Therefore by the above argument,
if $\kappa>\max\{\kappa_0,\kappa_1\}$,
then it holds that
\[
 \Phi(y,s_1)>\zeta_\kappa(y,s_1)
\hspace{5mm}\text{for } |y|<Ke^{-\omega s_1}.
\]
Next
we check boundary conditions on $|y|=Ke^{-\omega s}$ for $s\in[s_1,s_2]$.
Since $\varphi(y,s)\in A_{s_1,s_2}$ and
$\phi_{1\ell}(y)=c_{1\ell}(1+o(1))e_1(\theta)|y|^{-\gamma}$ for $|y|\ll1$,
there exists $\kappa_2>0$ such that
if $\kappa>\kappa_2$,
then it holds that
\[
 \Phi(y,s) \geq \zeta_\kappa(y,s)
\hspace{5mm}\text{on } |y|=Ke^{-\omega s},\ s_1<s<s_2.
\]
Therefore
since $\Phi(y,s),\zeta_\kappa(y,s)\in C([s_1,s_2];H_\rho^1(B_1))$,
a comparison lemma implies
$\Phi(y,s)\geq \zeta_\kappa(y,s)$ for $|y|<Ke^{-\omega s}$ and $s_1\leq s\leq s_2$
if $\kappa>\max\{\kappa_0,\kappa_1,\kappa_2\}$. 
Since $|\Phi(y,s)|<U_\infty(y,s)$ for $|y|<Kr^{-\omega s}$,
we obtain the conclusion.
\end{proof}
%%%%%%%%%%%%%%%%%%%%%%%%%%%%%%%%%%%%%%%%%%%%%%%%%%%%%%%%%%%

As a consequence,
a nonlinear term $f(\Phi(y,s))$ defined in \eqref{f(Phi)-eq} is estimated as follows.

%%%%%%%%%%%%%%%%%%%%%%%%%%%%%%%%%%%%%%%%%%%%%%%%%%%%%%%%%%%
\begin{lem}\label{f(Phi)-lem}
There exists $c>0$ such that
if $|d|<\epsilon_1e^{-\lambda^*s_1}$ and $\varphi(y,s)\in A_{s_1,s_2}$,
then it holds that for $s_1<s<s_2$
\[
 |f(\Phi(y,s))| \leq c
 \begin{cases}
 \min\left\{ r^{-(m+1)}, e^{-2\lambda^*s}r^{-1+m-2\gamma} \right\}
 & \mathrm{if}\ r<Ke^{-\omega  s},\\
 e^{-2\lambda^*s}r^{-1+m-2\gamma}
 & \mathrm{if}\ Ke^{-\omega  s}<r<1,\\
 e^{-2\lambda^*s}r^{-(m+1)+4\lambda^*}
 & \mathrm{if}\ 1<r<e^{\sigma s},\\
 r^{-(m+1)} & \mathrm{if}\ r>e^{\sigma s}.
 \end{cases}
\]
\end{lem}
%%%%%%%%%%%%%%%%%%%%%%%%%%%%%%%%%%%%%%%%%%%%%%%%%%%%%%%%%%%

%%%%%%%%%%%%%%%%%%%%%%%%%%%%%%%%%%%%%%%%%%%%%%%%%%%%%%%%%%%
\begin{proof}
By using $f(\Phi)=(\Phi+U_\infty)^q-U_\infty^q-Kr^{-1}\Phi$ and $\Phi=\varphi-U_\infty$,
we verify that
\[
 f(\Phi) = q(q-1)
 \left(
 \int_0^1\int_0^1
 \left( \theta_1\theta_2\varphi+(1-\theta_1\theta_2)U_\infty \right)^{q-2}
 \theta_1d\theta_1d\theta_2
 \right)
 \Phi^2.
\]
Since $\varphi(y,s)$ is positive,
we get
\[
 f(\Phi) \leq q(q-1)
\begin{cases}
 \dis
 \left(
 \int_0^1\int_0^1
 (1-\theta_1\theta_2)^{q-2}\theta_1d\theta_1d\theta_2
 \right)
 U_\infty^{q-2}\Phi^2
 & \text{if } 1<q<2,
 \\[4mm] \dis
 \left(
 \int_0^1\int_0^1\theta_1d\theta_1d\theta_2
 \right)
 (\varphi+U_\infty)^{q-2}\Phi^2
 & \text{if } q\geq2.
\end{cases}
\]
Therefore
Lemma \ref{Tana-eq} proves this lemma.
\end{proof}
%%%%%%%%%%%%%%%%%%%%%%%%%%%%%%%%%%%%%%%%%%%%%%%%%%%%%%%%%%%

%%%%%%%%%%%%%%%%%%%%%%%%%%%%%%%%%%%%%%%%%%%%%%%%%%%%%%%%%%%
\begin{lem}\label{A'F-lem}
Let $e^{{\cal B}s}$ be defined at the beginning of this section and
$F(\xi')\in BC(\pa\R_+^n)$.
Then it holds that
\[
 \left( e^{{\cal B}s}F,\eta_{ij} \right)_{\cal C}
=
 e^{-\lambda_{ij}s}(F,\eta_{ij})_{{\cal C},\pa\R_+^n}
\hspace{5mm}\mathrm{for}\ s>0.
\]
Furthermore
it holds that
\[
 e^{{\cal B}(s+s')}F
=
 e^{{\cal A}s}e^{{\cal B}s'}F
\hspace{5mm}\mathrm{for}\ s,s'>0.
\]
\end{lem}
%%%%%%%%%%%%%%%%%%%%%%%%%%%%%%%%%%%%%%%%%%%%%%%%%%%%%%%%%%%

%%%%%%%%%%%%%%%%%%%%%%%%%%%%%%%%%%%%%%%%%%%%%%%%%%%%%%%%%%%
\begin{proof}
By using \eqref{Gamma1-eq},
we easily obtain the first formula.
Furthermore
by using
\[
 \int_{\R_+}\Gamma(y,\xi,s)\Gamma(\xi,z',s'){\cal B}(\xi)d\xi
=
 \Gamma(y,z',s+s')\
\hspace{5mm}\text{for } z'\in\pa\R_+^n,\ s,s'>0,
\]
we obtain the the second formula.
\end{proof}
%%%%%%%%%%%%%%%%%%%%%%%%%%%%%%%%%%%%%%%%%%%%%%%%%%%%%%%%%%%

%%%%%%%%%%%%%%%%%%%%%%%%%%%%%%%%%%%%%%%%%%%%%%%%%%%%%%%%%%%
\begin{lem}\label{(*-ell)-lem}
There exist $c>0$ and $\delta>0$ such that
if $|d|<\epsilon_1e^{-\lambda^*s_1}$,
then it holds that
\[
 \|\phi_\ell^*-\phi_{1\ell}\|_\rho
\leq
 ce^{-\delta s_1}.
\]
\end{lem}
%%%%%%%%%%%%%%%%%%%%%%%%%%%%%%%%%%%%%%%%%%%%%%%%%%%%%%%%%%%

%%%%%%%%%%%%%%%%%%%%%%%%%%%%%%%%%%%%%%%%%%%%%%%%%%%%%%%%%%%
\begin{proof}
By definition of $\phi_\ell^*$,
it is easily seen that
\[
 \|\phi_\ell^*-\phi_{1\ell}\|_\rho^2
\leq
 2\left( \int_0^{(H+1)e^{-\omega s_1}}+\int_{e^{\varrho s_1}}^\infty \right)
 \left( \phi_\ell^*(y)^2+\phi_{1\ell}(y)^2 \right)\rho(y)dy.
\]
Here
from Theorem \ref{JLsuper-thm},
we see that
\[
 |\phi_\ell^*(y)|
<
\begin{cases}
\dis
 U_\infty(y) + e^{\lambda s_1}d_{\max}\sum_{(i,j)\in\Pi}|\phi_{ij}(y)|
\hspace{5mm} \text{for } |y|<He^{-\omega s_1},
\\[2mm] \dis
 2k_\alpha e_1(\theta)|y|^{-\gamma}
 +
 e^{\lambda s_1}d_{\max}\sum_{(i,j)\in\Pi}|\phi_{ij}(y)|
 +
 |\phi_{1\ell}(y)|,
 &
\\[2mm] \dis \hspace{50mm}
 \text{for } He^{-\omega s_1}<|y|<(H+1)e^{-\omega s_1}.
\\[2mm] \dis
 e^{\lambda^*s_1}\left( U_\infty(y)+d_{\max}\sum_{(i,j)\in\Pi}|\phi_{ij}| \right) +|\phi_{1\ell}|
\hspace{5mm} \text{for } |y|>e^{\varrho s_1}.
\end{cases}
\]
Therefore
since
$|\phi_{ij}(y)|<c_{ij}|y|^{-\gamma}$ for $|y|\ll1$ (see Lemma \ref{2A-lem}),
by using $m<\gamma<(n-2)/2$ and $H>1$,
we obtain
\[
\begin{array}{lll}
\dis
 \int_0^{(H+1)e^{-\omega s_1}}
 \left( \phi_\ell^*(y)^2+\phi_{1\ell}(y)^2 \right)\rho(y)dy
\hspace{-2mm}&\leq&\hspace{-2mm} \dis
 c\left( 1+e^{-2\lambda^*s_1}d_{\max}^2 \right)
 \int_0^{2He^{-\omega s_1}}
 \left( |y|^{-2m}+|y|^{-2\gamma} \right)dy
\\[4mm]
\hspace{-2mm}&=&\hspace{-2mm} \dis
 c\left( 1+e^{-2\lambda^*s_1}d_{\max}^2 \right)H^{n-2\gamma}e^{-(n-2\gamma)\omega s_1}.
\end{array}
\]
Furthermore
since $|\phi_{ij}(y)|<c_{ij}|y|^{2\lambda_{ij}-m}$ for $|y|\gg1$ (see Lemma \ref{2A-lem}),
we see that
\[
\begin{array}{lll}
\dis
 \int_{e^{\varrho s_1}}^\infty
 \left( \phi_\ell^*(y)^2+\phi_{1\ell}(y)^2 \right)\rho(y)dy
\hspace{-2mm}&=&\hspace{-2mm} \dis
 c\int_{e^{\varrho s_1}}^\infty
 \left(
 e^{2\lambda^*s_1}|y|^{-2m}+|y|^{4\lambda^*-2m}
 \right)
 \rho(y)dy
\\[4mm]
\hspace{-2mm}&\leq&\hspace{-2mm} \dis
 ce^{2\lambda^*s_1}
 \int_{e^{\varrho s_1}}^\infty
 |y|^{4\lambda^*-2m}\rho(y)
 dy
\\[4mm]
\hspace{-2mm}&=&\hspace{-2mm} \dis
 ce^{2\lambda^*s_1}
 \int_{e^{\varrho s_1}}^\infty
 r^{4\lambda^*-2m+n-1}e^{-r^2/4}dr.
\end{array}
\]
Since
$r^{4\lambda^*-2m+n-1}e^{-r^2/4}\leq cre^{-r^2/8}$ for $r>1$,
it follows that
\[
\begin{array}{lll}
\dis
 \int_{e^{\varrho s_1}}^\infty
 \left( \phi_\ell^*(y)^2+\phi_{1\ell}(y)^2 \right)\rho(y)dy
\leq
 ce^{2\lambda^*s_1}\exp\left( -\frac{e^{2\varrho s_1}}{8} \right).
\end{array}
\]
Therefore
by definition of $H$,
it follows that
\[
\begin{array}{lll}
\dis
 \|\phi_\ell^*-\phi_{1\ell}\|_\rho^2
\hspace{-2mm}&\leq&\hspace{-2mm} \dis
 c\left(
 H^{n-2\gamma}e^{-(n-2\gamma)\omega s_1}+e^{2\lambda^*s_1}\exp\left( -\frac{e^{2\varrho s_1}}{8} \right)
 \right)
\\[5mm]
\hspace{-2mm}&\leq&\hspace{-2mm} \dis
 c\left(
 e^{-(n-2\gamma)(\omega-a')s_1}+e^{2\lambda^*s_1}\exp\left( -\frac{e^{2\varrho s_1}}{8} \right)
 \right).
\end{array}
\]
Since $0<a'\ll1$,
we obtain the conclusion.
\end{proof}
%%%%%%%%%%%%%%%%%%%%%%%%%%%%%%%%%%%%%%%%%%%%%%%%%%%%%%%%%%%

%%%%%%%%%%%%%%%%%%%%%%%%%%%%%%%%%%%%%%%%%%%%%%%%%%%%%%%%%%%
\begin{lem}\label{Feta_ij-lem}
There exists $\delta>0$ such that for any $(i,j)\in\N^2$
there exists $c_{ij}>0$ such that
if $\varphi(y,s)\in A_{s_1,s_2}$,
then it holds that
\[
 |(F(s),\eta_{ij})_{{\cal C},\pa\R_+^n}|
\leq
 c_{ij}e^{-(\lambda^*+\delta)s}
\hspace{5mm}\mathrm{for}\ s_1<s<s_2.
\]
\end{lem}
%%%%%%%%%%%%%%%%%%%%%%%%%%%%%%%%%%%%%%%%%%%%%%%%%%%%%%%%%%%

%%%%%%%%%%%%%%%%%%%%%%%%%%%%%%%%%%%%%%%%%%%%%%%%%%%%%%%%%%%
\begin{proof}
Since $H<K$,
Lemma \ref{f(Phi)-lem} implies
\[
 |F(\xi',\tau)|
\leq
 c
 \begin{cases}
 |\xi'|^{-m-1+\gamma} & \text{if } |\xi'| <He^{-\omega\tau},
\\[1mm]
 e^{-2\lambda^*\tau}|\xi'|^{-1+m-\gamma}
 & \text{if } He^{-\omega\tau}<|\xi'| <1.
 \end{cases}
\]
From Lemma \ref{2A-lem} and Lemma \ref{3A-lem},
there exists $m_{ij}>0$ such that
$|\eta_{ij}(\xi)|<m_{ij}(1+|\xi|^{2\lambda_{ij}-m+\gamma})$ with $\lambda_{ij}\geq-(\gamma-m)/2$,
Hence
by using $|{\cal C}(\xi')|<c|\xi'|^{-2\gamma}\rho(\xi')$,
we get
\[
\begin{array}{l}
\dis
 \int_{\pa\R_+^n}
 F(\xi',s)\eta_{ij}{\cal C}\
 d\xi'
\leq
 c_{ij}\left(
 \int_0^{He^{-\omega s}}
 |\xi'|^{-m-1-\gamma}
 d\xi'
 +
 e^{-2\lambda^*s}\int_{He^{-\omega s}}^1
 |\xi'|^{-1+m-3\gamma}
 d\xi'
 \right.
\\[4mm] \dis \hspace{15mm}
 \left.
 +
 e^{-2\lambda^*s}\int_1^{e^{\sigma s}}
 |\xi'|^{-2m-1+4\lambda^*+2\lambda_{ij}}
 \rho(\xi')
 d\xi'
 +
 \int_{e^{\sigma s}}^\infty
 |\xi'|^{-2m-1+2\lambda_{ij}}
 \rho(\xi')
 d\xi'
 \right).
\end{array}
\]
Put $2d=(n-2)-2\gamma$.
Then
since $-1-2\gamma-d+(n-1)=d>0$,
% we can fix $0<d<\gamma-m$ such that $2\gamma+d<n-2$.
% Then
we obtain
\[
\begin{array}{l}
\dis
 \int_{\pa\R_+^n}
 F(\xi',s)\eta_{ij}{\cal C}\
 d\xi'
\leq
 c_{ij}(He^{-\omega s})^{-m-\gamma+n-2}
 +
 ce^{-2\lambda^*s}(He^{-\omega s})^{-(\gamma-m)+d}
\\[3mm] \dis \hspace{5mm}
\times
 \int_{He^{-\omega s}}^1
 |\xi'|^{-1-2\gamma-d}
 d\xi'
+
 c_{ij}e^{-2\lambda^*s}
+
 c_{ij}e^{-2\lambda^*s}
 \int_{e^{\sigma s}}^\infty
 |\xi'|^{2\lambda^*/\sigma}|\xi'|^{-2m-1+2\lambda_{ij}}
 \rho(\xi')
 d\xi'
\\[6mm] \dis \hspace{2mm}
\leq
 c_{ij}\left(
 H^{n-2-(\gamma+m)}
 e^{-(n-2-2\gamma)\omega s}e^{-\lambda^*s}
 +
 H^{-(\gamma-m)+d}e^{-(\lambda^*+d\omega)s}
+
 e^{-2\lambda^*s}
 \right).
\end{array}
\]
Therefore
since $H=e^{a's_1}$ with $0<a'\ll1$,
we obtain the conclusion.
\end{proof}
%%%%%%%%%%%%%%%%%%%%%%%%%%%%%%%%%%%%%%%%%%%%%%%%%%%%%%%%%%%

%%%%%%%%%%%%%%%%%%%%%%%%%%%%%%%%%%%%%%%%%%%%%%%%%%%%%%%%%%%
\begin{lem}\label{d_max-lem}
There exist $\delta>0$ and $c>0$ such that
if $\varphi(y,s)\in A_{s_1,s_2}$ and
$(\Phi(y,s_2),\phi_{ij})_\rho=0$ for any $(i,j)\in\Pi$,
then it holds that
\[
 d_{\max} < ce^{-(\lambda^*+\delta)s_1}.
\]
\end{lem}
%%%%%%%%%%%%%%%%%%%%%%%%%%%%%%%%%%%%%%%%%%%%%%%%%%%%%%%%%%%

%%%%%%%%%%%%%%%%%%%%%%%%%%%%%%%%%%%%%%%%%%%%%%%%%%%%%%%%%%%
\begin{proof}
From \eqref{b(y,s)express-eq} and Lemma \ref{A'F-lem},
we get for $(i,j)\in\Pi$
\[
\begin{array}{lll}
\dis
 0
\hspace{-2mm}&=&\hspace{-2mm} \dis
 (e^{-A(s_2-s_1)}b_1,\eta_{ij})_{\cal C} +
 \int_{s_1}^s\left( e^{{\cal B}(s_2-\tau)}F(\tau),\eta_{ij} \right)_{\cal C}
 d\tau
\\[4mm]
\hspace{-2mm}&=&\hspace{-2mm} \dis
 e^{-\lambda_{ij}(s_2-s_1)}(b_1,\eta_{ij})_{\cal C}
+
 \int_{s_1}^se^{-\lambda_{ij}(s_2-\tau)}
 d\tau
 \int_{\pa\R_+^n}
 F(\xi',\tau)\eta_{ij}(\xi'){\cal C}(\xi')
 d\xi'.
\end{array}
\]
Here by definition of $\varphi(y,s_1)$,
we see that
\[
 (b_1,\eta_{ij})_{\cal C}
=
 (\Phi(s_1),\phi_{ij})_\rho
=
 d_{ij} - e^{-\lambda^*s_1}(\phi_\ell^*,\phi_{ij})_\rho.
\]
This implies for $(i,j)\in\Pi$
\[
 d_{ij}
=
 -\int_{s_1}^se^{\lambda_{ij}(\tau-s_1)}
 d\tau
 \int_{\pa\R_+^n}
 F(\xi',\tau)\eta_{ij}(\xi'){\cal C}(\xi')
 d\xi'
+
 e^{-\lambda^*s_1}(\phi_\ell^*,\phi_{ij})_\rho.
\]
Then
from Lemma \ref{Feta_ij-lem} and $\lambda^*>\lambda_{ij}$ for $(i,j)\in\Pi$,
we obtain
\[
 \int_{s_1}^se^{\lambda_{ij}(\tau-s)}
 d\tau
 \int_{\pa\R_+^n}|F\eta_{ij}|{\cal C}
 d\xi'
\leq
 ce^{-(\lambda^*+\delta)s_1}.
\]
Next
we estimate $(\phi_\ell^*,\phi_{ij})_\rho$.
Since $(\phi_{1\ell},\phi_{ij})_\rho=0$ for $(i,j)\in\Pi$,
it holds that
\[
 (\phi_\ell^*,\phi_{ij})_\rho = (\phi_\ell^*-\phi_{1\ell},\phi_{ij})_\rho
\hspace{5mm}\text{for } (i,j)\in\Pi.
\]
Therefore
Lemma \ref{(*-ell)-lem} implies
\[
|(\phi_\ell^*,\phi_{ij})_\rho| < ce^{-\delta s_1}.
\]
Thus the proof is completed.
\end{proof}
%%%%%%%%%%%%%%%%%%%%%%%%%%%%%%%%%%%%%%%%%%%%%%%%%%%%%%%%%%%

We divide $b(y,s)$ several parts to estimate each term separately.
By definition of $\Phi(y,s_1)$,
$b(y,s)$ is rewritten by
\begin{equation}\label{b(s)express-eq}
\begin{array}{lll}
\dis
 b(y,s)
\hspace{-2mm}&=&\hspace{-2mm} \dis
 e^{{\cal A}(s-s_1)}b_1
+
 \int_{s_1}^se^{-{\cal B}(s-\tau)}F(\tau)d\tau ,
\\[4mm]
\hspace{-2mm}&=&\hspace{-2mm} \dis
 -e^{-\lambda^*s}(\phi_\ell^*,\phi_{1\ell})_\rho\eta_{1\ell}
+
 e^{-{\cal A}(s-s_1)}
 \left( b_1+e^{-\lambda^*s_1}(\phi_\ell^*,\phi_{1\ell})_\rho\eta_{1\ell} \right)
+
 S_2
\\[4mm]
\hspace{-2mm}&=&\hspace{-2mm} \dis
 -e^{-\lambda^*s}\eta_{1\ell}
+
 -e^{-\lambda^*s}(\phi_\ell^*-\phi_{1\ell},\phi_{1\ell})_\rho\eta_{1\ell}
+
 S_1 + S_2.
\end{array}
\end{equation}
To provide estimates of $S_i$ ($i=1,2$),
we consider three cases separately. 
\[
\begin{array}{lll}
\dis
 O_{\text{Short}}
\hspace{-2mm}&=&\hspace{-2mm} \dis
 \left\{ (y,s);\ Ke^{-\omega s}<r<e^{\sigma s},\ s_1<s<s_1+1 \right\},
\\[3mm]
 O_{\text{Long}}
\hspace{-2mm}&=&\hspace{-2mm} \dis
 \left\{ (y,s);\ Ke^{-\omega s}<r<e^{\sigma s},\ s>s_1+1 \right\},
\\[3mm] 
 O_{\text{Ext}}
\hspace{-2mm}&=&\hspace{-2mm} \dis
 \left\{ (y,s);\ r>Ke^{\sigma s},\ s>s_1 \right\}.
\end{array}
\]

%%%%%%%%%%%%%%%%%%%%%%%%%%%%%%%%%%%%%%%%%%%%%%%%%%%%%%%%%%%
%%%%%%%%%%%%%%%%%%%%%%%%%%%%%%%%%%%%%%%%%%%%%%%%%%%%%%%%%%%
\section{Short time estimates}\label{Shorttime-sec}
%%%%%%%%%%%%%%%%%%%%%%%%%%%%%%%%%%%%%%%%%%%%%%%%%%%%%%%%%%%
%%%%%%%%%%%%%%%%%%%%%%%%%%%%%%%%%%%%%%%%%%%%%%%%%%%%%%%%%%%

In this section,
we provide the estimate in
\[
 O_{\text{Short}}=\left\{ (y,s);\ Ke^{-\omega s}<r<e^{\sigma s},\ s_1<s<s_1+1 \right\}.
\]
Therefore
throughout this section,
we always assume $(y,s)\in O_{\text{Short}}$.
Furthermore
$|d|<\epsilon_1e^{-\lambda^*s_1}$ and $\varphi(y,s)\in A_{s_1,s_2}$
are also assumed in this section.
In this setting,
$\Gamma(y,\xi,s)$ is dominated by $\Theta(y,\xi,s)$ as follows.
\begin{equation}\label{GammaTheta-eq}
 \Gamma(y,\xi,s-s_1) \leq e^{(\gamma-m)/2}\Theta(y,\xi,s-s_1)
\hspace{5mm}\text{for } s_1<s<s_1+1.
\end{equation}

%%%%%%%%%%%%%%%%%%%%%%%%%%%%%%%%%%%%%%%%%%%%%%%%%%%%%%%%%%%
\begin{lem}\label{ShortS_1-lem}
Let $\nu(s)$ be given in Lemma {\rm\ref{phi-initial-lem}}.
Then
there exist $\delta>0$ and $c>0$ such that
\[
 |S_1(y,s)|
\leq
 c\left( \left( \nu(s_1)+e^{-\delta s_1} \right)e^{-\lambda^*s}+d_{\max} \right)
 \left( 1+|y|^{2\lambda^*-m+\gamma} \right)
\hh
 \mathrm{in}\ O_{\mathrm{Short}}.
\]
\end{lem}
%%%%%%%%%%%%%%%%%%%%%%%%%%%%%%%%%%%%%%%%%%%%%%%%%%%%%%%%%%%

%%%%%%%%%%%%%%%%%%%%%%%%%%%%%%%%%%%%%%%%%%%%%%%%%%%%%%%%%%%
\begin{proof}
For simplicity of notations,
we put
\[
 X=\xi/\sqrt{1-e^{-(s-s_1)}},
\hspace{7.5mm}
 Y=y/\sqrt{1-e^{-(s-s_1)}},
\hspace{7.5mm}
 Z=e^{-(s-s_1)/2}y/\sqrt{1-e^{-(s-s_1)}}.
\]
By definition of $\varphi(y,s_1)$,
it is easily seen that
\begin{equation}\label{S_1express-eq}
\begin{array}{llll}
\dis
 S_1
\hspace{-2mm}&=&\hspace{-2mm} \dis
 e^{{\cal A}(s-s_1)}\left( b_1+e^{-\lambda^*s_1}(\phi_\ell^*,\phi_{1\ell})_\rho\eta_{1\ell} \right)
\\[2mm]
\hspace{-2mm}&=&\hspace{-2mm} \dis
 e^{{\cal A}(s-s_1)}\left( b_1+e^{-\lambda^*s_1}\eta_{1\ell} \right)
+
 e^{-\lambda^*s}(\phi_\ell^*-\phi_{1\ell},\phi_{1\ell})_\rho
 \eta_{1\ell}.
\end{array}
\end{equation}
Since $\eta_{ij}(y)=\phi_{ij}(y)/\sigma(y)$ ($\sigma(y)\sim|y|^{-\gamma}$)
and $|\phi_{ij}(y)|<c(|y|^{-\gamma}+|y|^{2\lambda_{ij}-m})$,
by using Lemma \ref{(*-ell)-lem},
we obtain
\[
 |(\phi_\ell^*-\phi_{1\ell},\phi_{1\ell})_\rho\eta_{1\ell}|
\leq
 ce^{-\delta s_1}\left( 1+|y|^{2\lambda_{ij}-m+\gamma} \right)
\leq
 ce^{-\delta s_1}\left( 1+|y|^{2\lambda^*-m+\gamma} \right).
\]
Here
we used
$-(\gamma-m)/2\leq\lambda_{ij}<\lambda^*$ (see Lemma \ref{3A-lem}).
Therefore
it is sufficient to estimate the first term on the right-hand side in \eqref{S_1express-eq},
which is denoted by $S_1'$.
For simplicity,
here we denote $\Theta(y,\xi,s-s_1)$ by $\Theta$.
Since $b(y,s_1)+e^{-\lambda^*s_1}\eta_{1\ell}=(\Phi(y,s_1)+e^{-\lambda^*s_1}\phi_{ij})/\sigma(y)$,
by Lemma \ref{phi-initial-lem},
we see that
\[
\begin{array}{lll}
\dis
 |S_1'|
\hspace{-2mm}&\leq&\hspace{-2mm} \dis
 c\int_0^{He^{-\omega s_1}}
 \Theta\cdot
 \left( |\xi|^{-m+\gamma}+e^{-\lambda^*s_1} \right)
 {\cal B}d\xi
+c
 \left( \nu(s_1)e^{-\lambda^*s_1}+d_{\text{max}} \right)
\\[4mm] \hspace{-2mm}&&\hspace{-2mm} \dis
\times
 \left(
 \int_{He^{-\omega s_1}}^1
 \Theta\cdot{\cal B}d\xi
 +
 \int_1^{e^{\varrho s_1}}
 \Theta\cdot|\xi|^{2\lambda^*-m+\gamma}{\cal B}d\xi
 \right)
+
 c\int_{e^{\varrho s_1}}^\infty
 \Theta\cdot |\xi|^{-m+\gamma}{\cal B}d\xi
\\[6mm] \hspace{-2mm}&=:&\hspace{-2mm}\dis
 cI_1 + c\left( \nu(s_1)e^{-\lambda^*s_1}+d_{\text{max}} \right)(I_2 + I_3) + cI_4.
\end{array}
\]
%%%%%%%%%%%%%%%%%%%%%%%%%%%%%%%%%%%%%%%%%%%%%%%%%%%%%%%%%%%
First
we provide the estimate of $I_1$.
Since $|y|>Ke^{-\omega s}$, $s_1<s<s_1+1$ and $H\ll K$,
it holds that
\[
 |e^{-(s-s_1)/2}y-\xi|\geq c|y|\geq c(|y|+|\xi|)
\hspace{5mm}\text{for } |\xi|<He^{-\omega s_1}.
\]
Then
we see that
\[
\begin{array}{lll}
\dis
 I_1
\hspace{-2mm}&\leq&\hspace{-2mm} \dis
 \frac{\dis c\left( (He^{-\omega s_1})^{\gamma-m}+e^{-\lambda^*s_1} \right)}{(1-e^{-(s-s_1)})^{n/2}}
 e^{-c|Y|^2}
 \int_0^{He^{-\omega s_1}}
 e^{-c|X|^2}
 \left( 1+\frac{1}{|X|} \right)^{2\gamma}
 d\xi
\\[4mm]
\hspace{-2mm}&\leq&\hspace{-2mm} \dis
 \frac{\dis cH^{\gamma-m}e^{-\lambda^*s_1}e^{-c|Y|^2}\left( He^{-\omega s_1} \right)}
 {(1-e^{-(s-s_1)})^{n/2}}
 \int_0^{He^{-\omega s_1}}
 e^{-c|X|^2}
 \left( 1+\frac{1}{|X|} \right)^{2\gamma}|\xi|^{-1}
 d\xi
\\[4mm]
\hspace{-2mm}&\leq&\hspace{-2mm} \dis
 \frac{cH^{\gamma-m+1}e^{-\lambda^*s_1}e^{-\omega s_1}}{|y|}
 \left( |Y|e^{-c|Y|^2} \right)
 \int_{\R_+^n}
 \left( 1+\frac{1}{|z|} \right)^{2\gamma}|z|^{-1}
 e^{-c|z|^2}dz.
\end{array}
\]
Since $|Y|e^{-c|Y|^2}$ is uniformly bounded on  $|Y|>0$,
from $|y|>Ke^{-\omega s}$,
it holds that
\[
 I_1
\leq
 c\left( \frac{H^{\gamma-m+1}}{K} \right)
 e^{-\lambda^*s_1}.
\]
Next
we estimate $I_2$.
Since $b(y,s)=e^{{\cal A}_0(s-s_1)}b_0$ is a solution of
$b_s=\frac{1}{\cal C}\nabla({\cal C}\nabla b)$ in $\R_+^n\times(s_1,\infty)$,
$\pa_\nu b=0$ on $\pa\R_+^n\times(s_1,\infty)$
and
$b(y,s_1)=b_0(y)$ in $\R_+^n$,
we see that
\[
 I_2
=
 \int_{He^{-\omega s_1}}^1\Theta\cdot{\cal B}d\xi
\leq
 \int_{\R_+^n}\Theta\cdot{\cal B}d\xi
=
 e^{{\cal A}_0(s-s_1)}1
=
 1.
\]
Furthermore
to calculate $I_3$,
we divide it into two parts.
\[
 I_3 =
 \left( \int_1^{\max\{1,2e^{-(s-s_1)/2}|y|\}}+\int_{\max\{1,2e^{-(s-s_1)/2}|y|\}}^{e^{\varrho s}} \right)
 \Theta\cdot|\xi|^{2\lambda^*-m+\gamma}{\cal B}d\xi
 =: I_3' + I_3''.
\]
Then $I_3'$ is estimated by
\[
 I_3'
\leq
 \frac{c|y|^{2\lambda^*-m+\gamma}}{(1-e^{-(s-s_1)})^{n/2}}
 \int_1^{\max\{1,2e^{-(s-s_1)/2}|y|\}}e^{-c|Z|^2}d\xi
\leq
 c|y|^{2\lambda^*-m+\gamma}.
\]
Furthermore
since $|e^{-(s-s_1)/2}y-\xi|\geq|\xi|/2$ for $|\xi|>\max\{1,2e^{-(s-s_1)/2}|y|\}$,
we see that
\[
\begin{array}{lll}
\dis
 I_3''
\hspace{-2mm}&\leq&\hspace{-2mm} \dis
 \frac{c}{(1-e^{-(s-s_1)})^{n/2}}
 \int_{\max\{1,2e^{-(s-s_1)/2}|y|\}}^{e^{\varrho s}}
 \left( 1+|X|^{-1} \right)^{2\gamma}
 e^{-c|X|^2}|\xi|^{2\lambda^*-m+\gamma}
 d\xi
\\[6mm]
\hspace{-2mm}&\leq&\hspace{-2mm} \dis
 c\left( 1-e^{-(s-s_1)} \right)^{(2\lambda^*-m+\gamma)/2}
 \int_{\R^n}
 \left( 1+\frac{1}{|z|} \right)^{2\gamma}
 e^{-c|z|^2}|z|^{2\lambda^*-m+\gamma}dz.
\end{array}
\]
Since $\lambda^*>0$ and $\gamma>m$,
we find that $I_3''$ is uniformly bounded.
Finally
we provide the estimate of $I_4$.
Since $\sigma<\varrho $,
it holds that $|e^{-(s-s_1)/2}y-\xi|>|\xi|/2$ for $|\xi|>e^{\varrho s_1}$.
Therefore since $\gamma>m$,
we see that
\[
\begin{array}{lll}
\dis
 I_4
\hspace{-2mm}&\leq&\hspace{-2mm} \dis
 c\int_{e^{\varrho s_1}}^\infty
 \frac{\dis\left( 1+|X|^{-1} \right)^{2\gamma}}{(1-e^{-(s-s_1)})^{n/2}}
 |\xi|^{-m+\gamma}e^{-c|X|^2}
 d\xi
\\[4mm] \hspace{-2mm}&\leq&\hspace{-2mm} \dis
 c\left( 1-e^{-(s-s_1)} \right)^{(\gamma-m)/2}
 \int_{e^{\varrho s_1}/\sqrt{1-e^{-(s-s_1)}}}^\infty
 |z|^{\gamma-m}e^{-c|z|^2}
 d\xi
\\[4mm] \hspace{-2mm}&\leq&\hspace{-2mm} \dis
 c\int_{e^{\varrho s_1}}^\infty
 |z|^{\gamma-m}e^{-c|z|^2}
 d\xi.
\end{array}
\]
As a consequence,
by using $r^{\gamma-m+n-1}e^{-cr^2}\leq cre^{-c'r^2}$ for $r>1$,
we obtain
\[
 I_4
\leq
 c\int_{e^{\varrho s_1}}^\infty
 r^{\gamma-m}e^{-c'r^2}r^{n-1}dr
=
 c\int_{e^{\varrho s_1}}^\infty
 re^{-c'r^2}dr
=
 ce^{-c'e^{2\varrho s_1}}.
\]
Combining the above estimates,
we conclude
\[
 |S_1'|
\leq
 c\left(
 \left( \frac{H^{\gamma-m+1}}{K} \right)e^{-\lambda^*s_1}
 +
 \left( \nu(s_1)e^{-\lambda^*s_1}+d_{\text{max}} \right)
 \left( 1+|y|^{2\lambda^*-m+\gamma} \right)
 +
 e^{-ce^{-2\varrho s_1}}
 \right).
\]
Here
we recall that $H=e^{a's_1}$ and $K=e^{as_1}$ with $a'\ll a$.
Therefore
since $s_1<s<s_1+1$,
we obtain the conclusion.
\end{proof}
%%%%%%%%%%%%%%%%%%%%%%%%%%%%%%%%%%%%%%%%%%%%%%%%%%%%%%%%%%%

Here
we consider a more general form of an integral instead of $S_2$ for a later argument.
\[
 T(y,s) = \int_{\mu_1}^sd\tau\int_{\R_+^n}\Gamma(y,\xi',s-\tau)F(\xi',\tau)d\xi',
\]
where $\mu_1\in(s_1,s_2)$.
If we take $\mu_1=s_1$,
then $T(y,s)$ coincides with $S_2(y,s)$.

%%%%%%%%%%%%%%%%%%%%%%%%%%%%%%%%%%%%%%%%%%%%%%%%%%%%%%%%%%%
\begin{lem}\label{ShortS_2-lem}
There exist $\delta>0$ and $c>0$ independent of $\mu_1$ such that
if $s_1\leq\mu_1<s<\mu_1+1$,
then it holds that
\[
 |T(y,s)| \leq ce^{-\delta s_1}e^{-\lambda^*s}\left( 1+|y|^{2\lambda^*-m+\gamma} \right)
\hspace{5mm}
 \mathrm{for}\ Ke^{-\omega s}<|y|<e^{\sigma s}.
\]
\end{lem}
%%%%%%%%%%%%%%%%%%%%%%%%%%%%%%%%%%%%%%%%%%%%%%%%%%%%%%%%%%%

%%%%%%%%%%%%%%%%%%%%%%%%%%%%%%%%%%%%%%%%%%%%%%%%%%%%%%%%%%%
\begin{proof}
We divide the integral of $T(y,s)$ into four parts.
\[
\begin{array}{lll}
\dis
 T(y,s)
\hspace{-2mm}&=&\hspace{-2mm} \dis
 \int_{\mu_1}^sd\tau
 \left(
 \int_0^{He^{-\omega \tau}}+\int_{He^{-\omega \tau}}^1+
 \int_1^{e^{\sigma\tau}}+\int_{e^{\sigma\tau}}^\infty
 \right)
 \Gamma(y,\xi',s-\tau)F(\xi',\tau){\cal B}d\xi'
\\[6mm]
\hspace{-2mm}&=:&\hspace{-2mm} \dis
 I_1+I_2+I_3+I_4.
\end{array}
\]
For simplicity of notations,
we put
\[
 X'=\xi'/\sqrt{1-e^{-(s-\tau)}},
\hspace{7.5mm}
 Y=y/\sqrt{1-e^{-(s-\tau)}},
\hspace{7.5mm}
 Z=e^{-(s-\tau)/2}y/\sqrt{1-e^{-(s-\tau)}}.
\]
Furthermore
$\tau$ is always assumed to be $\mu_1<\tau<s$.
Then since $\mu_1\leq s\leq\mu_1+1$,
$\Gamma(y,\xi,s-\tau)$ is dominated by $\Theta(y,\xi,s-\tau)$ as \eqref{GammaTheta-eq}.
First
we provide the estimate of $I_1$.
In this case,
since $H\ll K$,
we find that
\[
 |e^{-(s-\tau)/2}y-\xi'|\geq c(|y|+|\xi'|).
\]
Therefore
we get from Lemma \ref{f(Phi)-lem}
\begin{equation}\label{APEC-eq}
\begin{array}{l}
\dis
 I_1
\leq
 c\int_{\mu_1}^s\frac{e^{-c|Y|^2}d\tau}{(1-e^{-(s-\tau)})^{n/2}}
 \int_0^{He^{-\omega \mu_1}}
 \left( 1+|X'|^{-1} \right)^{2\gamma}e^{-c|X'|^2}|\xi'|^{-(m+1)+\gamma}
 d\xi'.
\end{array}
\end{equation}
Then
the space integral is estimated by
\[
\begin{array}{l}
\dis
 \int_0^{He^{-\omega\mu_1}}
 \left( 1+|X'|^{-1} \right)^{2\gamma}e^{-c|X'|^2}|\xi'|^{-(m+1)+\gamma}
 d\xi'
\\[2mm] \dis \hspace{5mm}
\leq
 \left( He^{-\omega\mu_1} \right)^{\gamma-m}
 \int_0^{He^{-\omega\mu_1}}
 \left( 1+|X'|^{-1} \right)^{2\gamma}e^{-c|X'|^2}|\xi'|^{-1}
 d\xi'
\\[4mm] \dis \hspace{5mm}
=
 cH^{\gamma-m}e^{-\lambda^*\mu_1}
 \left( 1-e^{-(s-\tau)} \right)^{(n-2)/2}
 \int_0^{He^{-\omega\mu_1}/\sqrt{1-e^{-(s-\tau)}}}
 \dis\left( 1+\frac{1}{|z'|} \right)^{2\gamma}
 \frac{e^{-c|z'|^2}dz'}{|z'|}.
\end{array}
\]
Plugging this estimate into \eqref{APEC-eq},
we get
\[
 I_1
\leq
 cH^{\gamma-m}e^{-\lambda^*\mu_1}
 \int_{\mu_1}^s
 \frac{e^{-c|Y|^2}d\tau}{1-e^{-(s-\tau)}}
 \int_0^{He^{-\omega\mu_1}/\sqrt{1-e^{-(s-\tau)}}}
 \dis\left( 1+\frac{1}{|z'|} \right)^{2\gamma}
 \frac{e^{-c|z'|^2}dz'}{|z'|}.
\]
Put $2d=(n-2)-2\gamma>0$.
Then
it holds that
\[
\begin{array}{l}
\dis
 I_1
\leq
 cH^{\gamma-m}e^{-\lambda^*\mu_1}(He^{-\omega\mu_1})^d
 \int_{\mu_1}^s
 \frac{e^{-c|Y|^2}d\tau}{(1-e^{-(s-\tau)})^{1+d/2}}
\\[4mm] \dis \hspace{20mm}
\times
 \int_0^{He^{-\omega\mu_1}/\sqrt{1-e^{-(s-\tau)}}}
 \dis\left( 1+\frac{1}{|z'|} \right)^{2\gamma}
 \frac{e^{-c|z'|^2}dz'}{|z'|^{1+d}}.
\end{array}
\]
Here by changing variables as
\begin{equation}\label{varchange-eq}
 |y|^2/(1-e^{-(s-\tau)})=\mu
\hspace{7.5mm}
 \left( \mu^2d\tau=e^{s-\tau}|y|^2d\mu \right),
\end{equation}
we see that
\[
\begin{array}{l}
\dis
\int_{\mu_1}^s
 \frac{e^{-c|Y|^2}d\tau}{(1-e^{-(s-\tau)})^{1+d/2}}
 \int_0^{He^{-\omega\mu_1}/\sqrt{1-e^{-(s-\tau)}}}
 \dis\left( 1+\frac{1}{|z'|} \right)^{2\gamma}
 \frac{e^{-c|z'|^2}dz'}{|z'|^{1+d}}
\\[4mm] \dis \hspace{10mm}
\leq
 e^{s-\mu_1}|y|^{-d}
 \int_{|y|^2/(1-e^{-(s-\mu_1)})}^{\infty}
 \frac{e^{-c\mu}d\mu}{\mu^{1-d/2}}
 \int_0^{He^{-\omega\mu_1}\sqrt{\mu}/|y|}
 \left( 1+\frac{1}{|z'|} \right)^{2\gamma}
 \frac{e^{-c|z'|^2}dz'}{|z'|^{1+d}}.
\end{array}
\]
Since $-2\gamma+n-2-d=d>0$,
we observe that
\[
 \int_0^{\infty}
 \frac{e^{-c\mu}d\mu}{\mu^{1-d/2}}
 \int_{\R_+^n}
 \left( 1+\frac{1}{|z'|} \right)^{2\gamma}
 |z'|^{-1-d}e^{-c|z'|^2}dz'
<
 \infty.
\]
Therefore
since $|y|>Ke^{-\omega s}$ and $\mu_1<s<\mu_1+1$,
by definition of $H$ and $K$,
we obtain
\[
\begin{array}{lll}
\dis
 I_1
\hspace{-2mm}&=&\hspace{-2mm} \dis
 cH^{\gamma-m+d}e^{s-\mu_1}e^{-\lambda^*\mu_1}e^{-d\omega\mu_1}|y|^{-d}
\leq
 c\left( \frac{H^{\gamma-m+d}}{K^d} \right)
 e^{-\lambda^*\mu_1}
\\[2mm]
\hspace{-2mm}&=&\hspace{-2mm} \dis
 ce^{-(da-(\gamma-m+d)a')s_1}
 e^{-\lambda^*\mu_1}.
\end{array}
\]
Next
we give the estate of $I_2$.
Then
we divide the integral of $I_2$ into two parts.
\[
 I_2
\leq
 \int_{\mu_1}^sd\tau
 \left( \int_{He^{-\omega\tau}}^{4|y|}+\int_{4|y|}^1 \right)
 \Theta(y,\xi's-\tau)|F(\xi',\tau)|
 d\xi'
=:
I_2' + I_2''.
\]
From Lemma \ref{f(Phi)-lem},
we get
\[
\begin{array}{lll}
\dis
 I_2'
\hspace{-2mm}&\leq&\hspace{-2mm} \dis
 ce^{-2\lambda^*\mu_1}
 \int_{\mu_1}^sd\tau
 \int_{He^{-\omega\tau}}^{4|y|}
 \frac{\dis\left( 1+|X'|^{-1} \right)^{2\gamma}}{(1-e^{-(s-\tau)})^{n/2}}
 e^{-c|Z-X'|^2}|\xi'|^{-1+m-\gamma}
 d\xi'
\\[4mm]
\hspace{-2mm}&\leq&\hspace{-2mm} \dis
 ce^{-2\lambda^*\mu_1}
 \int_{\mu_1}^s(He^{-\omega\tau})^{-(\gamma-m)}
 d\tau
 \int_0^{4|y|}
 \frac{\dis\left( 1+|X'|^{-1} \right)^{2\gamma}}{(1-e^{-(s-\tau)})^{n/2}}
 e^{-c|Z-X'|^2}|\xi'|^{-1}
 d\xi'
\\[4mm]
\hspace{-2mm}&\leq&\hspace{-2mm} \dis
 cH^{-(\gamma-m)}e^{-\lambda^*\mu_1}
 \int_{\mu_1}^s
 \frac{d\tau}{1-e^{-(s-\tau)}}
 \int_0^{4|y|/\sqrt{1-e^{-(s-\tau)}}}
 \left( 1+|z'|^{-1} \right)^{2\gamma}
 \frac{e^{-c|Z-z'|^2}dz'}{|z'|}.
\end{array}
\]
Here
we put
\begin{equation}\label{YmuDef-eq}
 Y_\mu = \sqrt{\mu}\left( 1-\frac{|y|^2}{\mu} \right)^{1/2}\frac{y}{|y|}.
\end{equation}
By changing variables as \eqref{varchange-eq},
we see that
\[
\begin{array}{l}
\dis
 \int_{\mu_1}^s
 \frac{d\tau}{1-e^{-(s-\tau)}}
 \int_0^{4|y|/\sqrt{1-e^{-(s-\tau)}}}
 \left( 1+|z'|^{-1} \right)^{2\gamma}
 e^{-c|Z-z'|^2}\frac{dz'}{|z'|}
\\[6mm] \dis \hspace{10mm}
\leq
 c\int_{|y|^2/(1-e^{-(s-\mu_1)})}^\infty
 \frac{d\mu}{\mu}
 \int_0^{4\sqrt{\mu}}
 \left( 1+|z'|^{-1} \right)^{2\gamma}|z'|^{-1}
 e^{-c|Y_\mu-z'|^2}\frac{dz'}{|z'|}.
\end{array}
\]
We divide the above integral into two parts. 
\[
\begin{array}{l}
\dis
\int_{|y|^2/(1-e^{-(s-\mu_1)})}^\infty
 \frac{d\mu}{\mu}
 \int_0^{4\sqrt{\mu}}
 \left( 1+|z'|^{-1} \right)^{2\gamma}|z'|^{-1}
 e^{-c|Y_\mu-z'|^2}\frac{dz'}{|z'|}
\\[6mm] \dis \hspace{15mm}
=
 \left(
 \int_{\min\{|y|^2/(1-e^{-(s-\mu_1)}),1\}}^1+\int_{\max\{|y|^2/(1-e^{-(s-\mu_1)}),1\}}^{\infty}
 \right)
 \frac{d\mu}{\mu}
\\[6mm] \dis \hspace{60mm}
\times
 \int_0^{4\sqrt{\mu}}
 \left( 1+|z'|^{-1} \right)^{2\gamma}
 e^{-c|Y_\mu-z'|^2}\frac{dz'}{|z'|}
\\ \dis \hspace{15mm}
=: A_1+A_2.
\end{array}
\]
Since $-2\gamma+n-2>0$,
we see that
\[
 A_1
\leq
 c\int_0^1\frac{d\mu}{\mu}
 \int_0^{4\sqrt{\mu}}|z'|^{-(2\gamma+1)}dz'
\leq
 c\int_0^1
 \mu^{-1+(-2\gamma+n-2)/2}d\mu
\leq
 c.
\]
To estimate $A_2$,
we divide the space integral into two parts.
Let
\begin{equation}\label{D_1D_2-eq}
\begin{array}{c}
\dis
 D_1 =
 \left\{
 |z'|<4\sqrt{\mu};\ \left| Y_\mu-z' \right|>|Y_\mu|/2
 \right\},
\\[3mm] \dis
 D_2 =
 \left\{
 |z'|<4\sqrt{\mu};\ \left| Y_\mu-z' \right|<|Y_\mu|/2
 \right\}.
\end{array}
\end{equation}
Furthermore
we put
\[
 A_2(D_i)
=
 \int_{\max\{|y|^2/(1-e^{-(s-\mu_1)}),1\}}^{\infty}\frac{d\mu}{\mu}
 \int_{D_i}
 \left( 1+|z'|^{-1} \right)^{2\gamma}|z'|^{-1}
 e^{-c|Y_\mu-z'|^2}
 dz'.
\]
Then
it is clear that
$A_2\leq A_2(D_1)+A_2(D_2)$.
Here
we note that
$2|Y_\mu-z'|>|Y_\mu|>e^{-(s-\mu_1)/2}\sqrt{\mu}$
for $z'\in D_1$ and $\mu\geq|y|^2/(1-e^{-(s-\mu_1)})$.
Therefore
$A_2(D_1)$ is estimated by
\[
 A_2(D_1)
\leq
 \int_1^\infty
 \exp\left( -ce^{-(s-\mu_1)}\mu \right)
 d\mu
 \int_0^{4\sqrt{\mu}}
 \left( 1+|z'|^{-1} \right)^{2\gamma}|z'|^{-1}
 dz'.
\]
Since $-2\gamma+n-2>0$ and $\mu_1<s<\mu_1+1$,
we find that $A_2(D_1)$ is uniformly bounded.
Next we estimate $A_2(D_2)$.
Then
it is easily verified that $|Y_\mu|<2|z'|$ for $z'\in D_2$.
Furthermore
by definition of $Y_\mu$,
it holds that
$|Y_\mu|>e^{-(s-\mu_1)/2}\sqrt{\mu}$ for $\mu\geq|y|^2/(1-e^{-(s-\mu_1)})$.
Therefore
it holds that
\begin{equation}\label{D_2'-eq}
 D_2
\subset
 \left\{ z'\in\R^{n-1};\ e^{-(s-\mu_1)/2}\sqrt{\mu}/2<|z'|<4\sqrt{\mu} \right\}
\hspace{5mm}\text{for }
 \mu\geq\frac{|y|^2}{1-e^{-(s-\mu_1)}}.
\end{equation}
As a consequence,
we see that
\[
\begin{array}{lll}
\dis
 A_2(D_2)
\hspace{-2mm}&\leq&\hspace{-2mm} \dis
 \int_{\max\{|y|^2/(1-e^{-(s-\mu_1)}),1\}}^{\infty}
 \frac{d\mu}{\mu}
 \int_{e^{-(s-\mu_1)/2}\sqrt{\mu}/2}^{4\sqrt{\mu}}
 \left( 1+|z'|^{-1} \right)^{2\gamma}|z'|^{-1}
 e^{-c|Y_\mu-z'|^2}
 dz'
\\[6mm]
\hspace{-2mm}&\leq&\hspace{-2mm} \dis
 c\int_{\max\{|y|^2/(1-e^{-(s-\mu_1)}),1\}}^{\infty}
 \mu^{-1}\left( 1+\mu^{-1} \right)^{\gamma}\mu^{-1/2}
 d\mu
 \int_{\R^{n-1}}
 e^{-c|Y_\mu-z'|^2}
 dz'
\\[6mm]
\hspace{-2mm}&\leq&\hspace{-2mm} \dis
 c\int_1^\infty\mu^{-3/2}d\mu.
\end{array}
\]
Combining the above inequalities,
we obtain
\[
 I_2'
\leq
 cH^{-(\gamma-m)}e^{-\lambda^*\mu_1}(A_2(D_1) + A_2(D_2) \leq
 cH^{-(\gamma-m)}e^{-\lambda^*\mu_1}.
\]
Next
we provide the estimate of $I_2''$.
Then
it is easily verified that
$|e^{-(s-\tau)/2}y-\xi'|\geq c(|y|+|\xi'|)$ for $|\xi'|>4|y|$.
Then
since $\mu_1<s<\mu_1+1$,
we see that
\[
\begin{array}{lll}
\dis
 I_2''
\hspace{-2mm}&\leq&\hspace{-2mm} \dis
 ce^{s-\mu_1}e^{-2\lambda^*\mu_1}
 \int_{\mu_1}^s
 \frac{e^{-c|Y|^2}d\tau}{(1-e^{-(s-\tau)})^{n/2}}
 \int_{4|y|}^1
 \left( 1+|X'|^{-1} \right)^{2\gamma}
 \frac{e^{-c|X'|^2}dz'}{|\xi'|^{1-m+\gamma}}
 d\xi'
\\[6mm]
\hspace{-2mm}&\leq&\hspace{-2mm} \dis
 ce^{-2\lambda^*\mu_1}
 \int_{\mu_1}^s
 \frac{e^{-c|Y|^2}d\tau}{(1-e^{-(s-\tau)})^{(2+\gamma-m)/2}}
 \int_{4|y|/\sqrt{1-e^{-(s-\tau)}}}^\infty
 \left( 1+\frac{1}{|z'|} \right)^{2\gamma}
 \frac{e^{-c|z'|^2}dz'}{|z'|^{1-m+\gamma}}.
\end{array}
\]
By changing variables as \eqref{varchange-eq},
we get
\[
 I_2''
\leq
 ce^{-2\lambda^*\mu_1}|y|^{-(\gamma-m)}
 \int_0^\infty
 \frac{e^{-c\mu}d\mu}{\mu^{1-(\gamma-m)/2}} 
 \int_{4\sqrt{\mu}}^\infty
 \left( 1+\frac{1}{|z'|} \right)^{2\gamma}
 \frac{e^{-c|z'|^2}dz'}{|z'|^{1-m+\gamma}}.
\]
Here
we change the order of integrals.
\[
\begin{array}{l}
\dis
 \int_0^\infty
 \frac{e^{-c\mu}d\mu}{\mu^{1-(\gamma-m)/2}}
 \int_{4\sqrt{\mu}}^\infty
 \left( 1+\frac{1}{r} \right)^{2\gamma}r^{m-\gamma+n-3}
 e^{-cr^2}dr
\\[6mm] \dis \hspace{20mm}
=
 \int_0^\infty
 \left( 1+\frac{1}{r} \right)^{2\gamma}r^{m-\gamma+n-3}
 e^{-cr^2}dr
 \int_0^{16r^2}
 \frac{e^{-c\mu}}{\mu^{1-(\gamma-m)/2}}d\mu
\\[6mm] \dis \hspace{20mm}
\leq
 c\int_0^\infty
 \left( 1+\frac{1}{r} \right)^{2\gamma}r^{n-3}
 e^{-cr^2}dr.
\end{array}
\]
Since $-2\gamma+n-2>0$,
the above integral is finite.
Therefore
since $|y|>Ke^{-\omega s}$ and $\mu_1<s<\mu_1+1$,
we obtain
\[
 I_2'' \leq ce^{-2\lambda^*\mu_1}|y|^{-(\gamma-m)} \leq
 cK^{-(\gamma-m)}e^{-\lambda^*\mu_1}.
\]
As a consequence,
we conclude
\[
 I_2 \leq c\left( H^{-(\gamma-m)}+K^{-(\gamma-m)} \right)e^{-\lambda^*\mu_1}.
\]
Now
we calculate $I_3$.
We divide the integral into two parts.
\[
\begin{array}{lll}
\dis
 I_3
\hspace{-2mm}&\leq&\hspace{-2mm} \dis
 ce^{-2\lambda^*\mu_1}
 \int_{\mu_1}^s\frac{d\tau}{(1-e^{-(s-\tau)})^{n/2}}
 \left( \int_1^{\max\{4|y|,1\}}+\int_{4|y|}^{e^{\sigma\tau}} \right) 
\\[6mm]
\hspace{-2mm}&&\hspace{30mm} \dis
\times
 \left( 1+|X'|^{-1} \right)^{2\gamma}
 e^{-c|Z-X'|^2}|\xi'|^{4\lambda^*-m+\gamma-1}d\xi'
\\[2mm]
\hspace{-2mm}&=&\hspace{-2mm} \dis
 ce^{-2\lambda^*s_1}(A_1+A_2).
\end{array}
\]
In the first inequality,
we applied Lemma \ref{f(Phi)-lem}.
First we estimate $A_1$.
Since $A_1=0$ if $4|y|<1$,
we assume $4|y|>1$.
Then
by changing variables as in the estimate of $I_2$,
we get
\[
\begin{array}{lll}
\dis
 A_1
\hspace{-2mm}&\leq&\hspace{-2mm} \dis
 \int_{\mu_1}^s
 (1-e^{-(s-\tau)})^{(4\lambda^*-m+\gamma-2)/2}
 d\tau
 \int_0^{4|y|/\sqrt{1-e^{-(s-\tau)}}} 
\\[3mm] && \dis \hspace{25mm}
\times
 \left( 1+|z'|^{-1} \right)^{2\gamma}
 e^{-c|Z-z'|^2}|z'|^{4\lambda^*-m+\gamma-1}dz'
\\[2mm]
\hspace{-2mm}&\leq&\hspace{-2mm} \dis
 c|y|^{4\lambda^*+\gamma-m}
 \int_{|y|^2/(1-e^{-(s-\mu_1)})}^\infty
 d\mu
 \int_0^{4\sqrt{\mu}}
 \frac{|z'|^{4\lambda^*-m+\gamma-1}}{\mu^{1+(4\lambda^*+\gamma-m)/2}}
 e^{-c|Y_\mu-z'|^2}
 dz',
\end{array}
\]
where $Y_\mu$ is given by \eqref{YmuDef-eq}.
Since $4|y|>1$,
we note that $|y|^2/(1-e^{-(s-\mu_1)})>1/16$.
Therefore
it holds that
\[
 A_1
\leq
 c|y|^{4\lambda^*+\gamma-m}
 \int_{1/16}^\infty
 d\mu
 \int_0^{4\sqrt{\mu}}
 \frac{|z'|^{4\lambda^*-m+\gamma-1}}{\mu^{1+(4\lambda^*+\gamma-m)/2}}
 e^{-c|Y_\mu-z'|^2}
 dz'.
\]
To estimate the above integral,
we divide it into two parts.
\[
\begin{array}{lll}
 A_1
\hspace{-2mm}&\leq&\hspace{-2mm} \dis
 c|y|^{4\lambda^*+\gamma-m}
 \int_{1/16}^\infty d\mu
 \left( \int_{D_1}+\int_{D_2} \right)
 \frac{|z'|^{4\lambda^*-m+\gamma-1}}{\mu^{1+(4\lambda^*+\gamma-m)/2}}
 e^{-c|Y_\mu-z'|^2}
 dz'
\\[5mm]
\hspace{-2mm}&=:&\hspace{-2mm} \dis
 c|y|^{4\lambda^*+\gamma-m}
 (A_1(D_1) + A_1(D_2)),
\end{array}
\]
where $D_1$ and $D_2$ are given by \eqref{D_1D_2-eq}.
By the same calculation as in the estimate of $I_2$,
we see that
\[
 A_1(D_1)
\leq
 \int_{1/16}^\infty
 e^{-c\mu}d\mu
 \int_0^{4\sqrt{\mu}}
 \frac{|z'|^{4\lambda^*-m+\gamma-1}}{\mu^{1+(4\lambda^*+\gamma-m)/2}}
 dz'
=
 c\int_{1/16}^\infty
 \mu^{(n-4)/2}e^{-c\mu}d\mu.
\]
On the other hand,
by using \eqref{D_2'-eq},
we verify that
\[
 A_1(D_2)
\leq
 c\int_{1/16}^\infty\mu^{-3/2}d\mu
 \int_{e^{-(s-\mu_1)/2}\sqrt{\mu}/2}^{4\sqrt{\mu}}
 e^{-c|Y_\mu-z'|^2}dz'
\leq
 c\int_{1/16}^\infty\mu^{-3/2}d\mu.
\]
Therefore
since $|y|<e^{\sigma s}$,
we obtain
\[
 A_1 \leq c|y|^{4\lambda^*+\gamma-m} \leq ce^{2\sigma\lambda^* s}|y|^{2\lambda^*+\gamma-m}.
\]
Next
we provide the estimate of $A_2$.
Here
we note that
$|e^{-(s-\tau)}y-\xi'|>c|\xi'|$ for $|\xi'|>4|y|$.
Therefore
we see that
\[
 A_2
\leq
 \int_{\mu_1}^s
 (1-e^{-(s-\tau)})^{(4\lambda^*-m+\gamma-2)/2}
 d\tau
 \int_{\R^{n-1}}
 \left( 1+|z'|^{-1} \right)^{2\gamma}
 e^{-c|z'|^2}|z'|^{4\lambda^*-m+\gamma-1}dz'.
\]
Since $\gamma>m$ and $2\gamma<(n-2)$,
$A_2$ turns out to be bounded.
Thus
since $\mu_1<s<\mu_1+1$,
we obtain
\[
 I_3
\leq
 ce^{-2\lambda^*\mu_1}
 \left( e^{2\sigma\lambda^*s}|y|^{2\lambda^*-m+\gamma}+1 \right)
\leq
 ce^{-(1-2\sigma)\lambda^*\mu_1}e^{-\lambda^*\mu_1}\left(|y|^{2\lambda^*-m+\gamma}+1 \right).
\]
Finally
we provide the estimate of $I_4$.
We consider two cases (i) $|y|<e^{\sigma s}/4$ and (ii) $|y|>e^{\sigma s}/4$ separately.
For the case (i),
it holds that $|e^{-(s-\mu_1)/2}y-\xi'|>c(|\xi'|+e^{\sigma \tau})$
for $|\xi'|>e^{\sigma\tau}$.
Therefore
by Lemma \ref{f(Phi)-lem} and $\gamma>m$,
we get
\[
\begin{array}{lll}
\dis
 I_4
\hspace{-2mm}&\leq&\hspace{-2mm} \dis
 \int_{\mu_1}^s
 \frac{\dis\exp\left( -\frac{ce^{2\sigma\tau}}{(1-e^{-(s-\tau)})} \right)}{(1-e^{-(s-\tau)})^{n/2}}
 d\tau
 \int_{e^{\sigma\tau}}^\infty
 \left( 1+|X'|^{-1} \right)^{2\gamma}|\xi'|^{-m+\gamma-1}
 e^{-c|X'|^2}
 d\xi'
\\[4mm]
\hspace{-2mm}&\leq&\hspace{-2mm} \dis
 c
 \int_{\mu_1}^s
 \frac{\dis\exp\left( -\frac{ce^{2\sigma\tau}}{(1-e^{-(s-\tau)})} \right)}
 {(1-e^{-(s-\tau)})^{1-(\gamma-m)/2}}
 d\tau
 \int_{\R^{n-1}}
 |z'|^{-m+\gamma-1}e^{-c|z'|^2}
 dz'
\\[5mm]
\hspace{-2mm}&\leq&\hspace{-2mm} \dis
 c\exp\left( -\frac{ce^{2\sigma \mu_1}}{(1-e^{-(s-\mu_1)})} \right)
 \int_{\mu_1}^s
 \frac{d\tau}{(1-e^{-(s-\tau)})^{1-(\gamma-m)/2}}.
\end{array}
\]
Since $\mu_1<s<\mu_1+1$ and $\gamma>m$,
it follows that
\[
 I_4 \leq ce^{-ce^{-2\sigma \mu_1}}.
\]
Next
we consider the case (ii).
Then
$I_4$ is estimated by
\[
\begin{array}{lll}
\dis
 I_4
\hspace{-2mm}&\leq&\hspace{-2mm} \dis
 \int_{\mu_1}^s
 \frac{d\tau}{(1-e^{-(s-\tau)})^{n/2}}
 \int_{e^{\sigma\tau}}^\infty
 \left( 1+|X'|^{-1} \right)^{2\gamma}|\xi'|^{-m+\gamma-1}
 e^{-c|Z-X'|^2}
 d\xi'
\\[6mm]
\hspace{-2mm}&=&\hspace{-2mm} \dis
 \int_{\mu_1}^s
 \frac{d\tau}{(1-e^{-(s-\tau)})^{(m-\gamma+2)/2}}
 \int_{e^{\sigma\tau}/\sqrt{1-e^{-(s-\tau)}}}^\infty
 \frac{e^{-c|Z-z'|^2}dz'}{|z'|^{m-\gamma+1}}
\\[6mm]
\hspace{-2mm}&=&\hspace{-2mm} \dis
 \int_{\mu_1}^s
 \frac{d\tau}{(1-e^{-(s-\tau)})^{(m-\gamma+2)/2}}
 \left(
 \int_{e^{\sigma\tau}/\sqrt{1-e^{-(s-\tau)}}}^{4|y|/\sqrt{1-e^{-(s-\tau)}}}
 +
 \int_{4|y|/\sqrt{1-e^{-(s-\tau)}}}^\infty
 \right)
 \frac{e^{-c|Z-z'|^2}dz'}{|z'|^{m-\gamma+1}}
\\[4mm]
\hspace{-2mm}&=:&\hspace{-2mm} \dis
 I_4'+I_4''.
\end{array}
\]
First
we provide the estimate of $I_4'$.
\[
\begin{array}{lll}
\dis
 I_4'
\hspace{-2mm}&\leq&\hspace{-2mm} \dis
 c|y|^{\gamma-m}
 \int_{\mu_1}^s
 \frac{d\tau}{1-e^{-(s-\tau)}}
 \int_{e^{\sigma\tau}/\sqrt{1-e^{-(s-\tau)}}}^{4|y|/\sqrt{1-e^{-(s-\tau)}}}
 |z'|^{-1}e^{-c|Z-z'|^2}
 dz'
\\[6mm]
\hspace{-2mm}&\leq&\hspace{-2mm} \dis
 c|y|^{\gamma-m}
 \int_{\mu_1}^s
 \frac{d\tau}{1-e^{-(s-\tau)}}
 \left( \frac{e^{\sigma\tau}}{\sqrt{1-e^{-(s-\tau)}}} \right)^{-1}
 \int_{\R^{n-1}}e^{-c|Z-z'|^2}
 dz'
\\[4mm]
\hspace{-2mm}&\leq&\hspace{-2mm} \dis
 ce^{-\sigma \mu_1}|y|^{\gamma-m}
 \int_{\mu_1}^s
 \frac{d\tau}{\sqrt{1-e^{-(s-\tau)}}}.
\end{array}
\]
Since $|y|>e^{\sigma s}/4$,
it holds that
$e^{2\sigma\lambda^*s}\leq4^{2\lambda^*}|y|^{2\lambda^*-m+\gamma}$.
Therefore
we obtain
\[
 I_4' \leq ce^{-(2\lambda^*+1)\sigma \mu_1}|y|^{2\lambda^*-m+\gamma}.
\]
Furthermore
since $|Z-z'|>c|z'|$ for $|z'|>4|y|/\sqrt{1-e^{-(s-\mu_1)}}$,
we get
\[
\begin{array}{lll}
\dis
 I_4''
\hspace{-2mm}&\leq&\hspace{-2mm} \dis
 \int_{\mu_1}^s
 \frac{d\tau}{(1-e^{-(s-\tau)})^{(m-\gamma+2)/2}}
 \int_{4|y|/\sqrt{1-e^{-(s-\tau)}}}^\infty
 \frac{e^{-c|z'|^2}dz'}{|z'|^{m-\gamma+1}}
\\[6mm]
\hspace{-2mm}&\leq&\hspace{-2mm} \dis
 c\int_{\mu_1}^s
 \frac{\left( \sqrt{1-e^{-(s-\tau)}}/|y| \right)^{2\lambda^*+1}}{(1-e^{-(s-\tau)})^{(m-\gamma+2)/2}}
 d\tau
 \int_{4|y|/\sqrt{1-e^{-(s-\tau)}}}^\infty
 |z'|^{2\lambda^*+1}
 \frac{e^{-c|z'|^2}dz'}{|z'|^{m-\gamma+1}}
\\[6mm]
\hspace{-2mm}&\leq&\hspace{-2mm} \dis
 c|y|^{-(2\lambda^*+1)}
 \int_{\mu_1}^s
 \frac{d\tau}{(1-e^{-(s-\tau)})^{(m-\gamma-1-2\lambda^*)/2}}
 \int_{\R^{n-1}}
 |z'|^{2\lambda^*+\gamma-m}e^{-c|z'|^2}dz'.
\end{array}
\]
Since $|y|>e^{\sigma s}/4$ and $\mu_1<s<\mu_1+1$,
we obtain
\[
 I_4'' \leq ce^{-(2\lambda^*+1)\sigma s}.
\]
As a consequence,
since $\mu_1<s<\mu_1+1$,
we conclude
\[
 I_4 \leq
 \left( e^{-(2\lambda^*+1)\sigma \mu_1}+e^{-ce^{-2\sigma \mu_1}} \right)
 \left( 1+|y|^{2\lambda^*-m+\gamma} \right).
\]
By definition of $\sigma$,
we note that $(2\lambda^*+1)\sigma>\lambda^*$.
Thus
combining the above estimates,
we obtain the conclusion.
\end{proof}
%%%%%%%%%%%%%%%%%%%%%%%%%%%%%%%%%%%%%%%%%%%%%%%%%%%%%%%%%%%

Therefore
we obtain the desired estimate in $O_{\text{Short}}$.

%%%%%%%%%%%%%%%%%%%%%%%%%%%%%%%%%%%%%%%%%%%%%%%%%%%%%%%%%%%
\begin{pro}\label{Shorttime-pro}
Let $\nu(s)$ be given in Lemma {\rm\ref{phi-initial-lem}}.
Then
there exist $c>0$ and $\delta>0$ such that 
if $|d|<\epsilon_1e^{-\lambda^*s_1}$ and $\varphi(y,s)\in A_{s_1,s_2}$,
then it holds that
\[
\begin{array}{lll}
\dis
 \left| b(y,s)+e^{-\lambda^*s}\eta_{1\ell}(y) \right|
\hspace{-2mm}&=&\hspace{-2mm} \dis
 c\left( \nu(s_1)+d_{\max}e^{\lambda^*s_1}+e^{-\delta s_1} \right)
 e^{-\lambda^*s}\left( 1+|y|^{2\lambda^*-m+\gamma} \right)
\\[4mm]
\hspace{-2mm}&&\hspace{-2mm} \dis
\hspace{45mm}\mathrm{for}\ (y,s)\in O_{\mathrm{Short}}.
\end{array}
\]
\end{pro}
%%%%%%%%%%%%%%%%%%%%%%%%%%%%%%%%%%%%%%%%%%%%%%%%%%%%%%%%%%%

%%%%%%%%%%%%%%%%%%%%%%%%%%%%%%%%%%%%%%%%%%%%%%%%%%%%%%%%%%%
\begin{proof}
Applying Lemma \ref{(*-ell)-lem} and Lemma \ref{ShortS_1-lem}-Lemma \ref{ShortS_2-lem}
to each term in \eqref{b(s)express-eq},
we obtain the conclusion.
\end{proof}
%%%%%%%%%%%%%%%%%%%%%%%%%%%%%%%%%%%%%%%%%%%%%%%%%%%%%%%%%%%

%%%%%%%%%%%%%%%%%%%%%%%%%%%%%%%%%%%%%%%%%%%%%%%%%%%%%%%%%%%
%%%%%%%%%%%%%%%%%%%%%%%%%%%%%%%%%%%%%%%%%%%%%%%%%%%%%%%%%%%
\section{Long time estimates}\label{Longtime-sec}
%%%%%%%%%%%%%%%%%%%%%%%%%%%%%%%%%%%%%%%%%%%%%%%%%%%%%%%%%%%
%%%%%%%%%%%%%%%%%%%%%%%%%%%%%%%%%%%%%%%%%%%%%%%%%%%%%%%%%%%

In this section,
we provide the estimate in
\[
 O_{\text{Long}}
=
 \left\{ (y,s);\ Ke^{-\omega s}<|y|<e^{\sigma s},\ s_1+1<s<s_2 \right\}.
\]
Throughout this section,
we always assume
\begin{equation}\label{AssumeAPi-eq}
 \varphi(y,s)\in A_{s_1,s_2},
\hspace{7.5mm}
 (b(s_2),\eta_{ij})_{\cal C} = 0
\hspace{3mm}\text{for } (i,j)\in\Pi.
\end{equation}
We recall that $b(y,s)$ is expressed by
\begin{equation}\label{express-eq}
 b(s) =
 e^{{\cal A}(s-s_1)}b_1 + \int_{s_1}^se^{{\cal B}(s-\tau)}F(\tau)d\tau.
\end{equation}
From $(b(s_2),\eta_{ij})_{\cal C} = 0$ and \eqref{express-eq},
we easily see that
\[
 e^{-\lambda_{ij}(s_2-s_1)}(b_1,\eta_{ij})_{\cal C}
+
 \int_{s_1}^{s_2}
 (e^{{\cal B}(s_2-\tau)}F(\tau),\eta_{ij})_{\cal C}d\tau = 0
\hspace{5mm}\text{for } (i,j)\in\Pi.
\]
Therefore
by using
$e^{\lambda_{ij}(s_2-s)}(e^{{\cal B}(s_2-\tau)}F(\tau),\eta_{ij})_{\cal C}=
 (e^{{\cal B}(s-\tau)}F(\tau),\eta_{ij})_{\cal C}$,
(see Lemma \ref{A'F-lem}),
we get
\[
\begin{array}{l}
\dis
 e^{{\cal A}(s-s_1)}b_1
=
 e^{{\cal A}(s-s_1)}
 \left(
 b_1-\sum_{(i,j)\in\Pi}(b_1,\eta_{ij})_{\cal C}\eta_{ij}
 \right)
+
 \sum_{(i,j)\in\Pi}e^{-\lambda_{ij}(s-s_1)}(b_1,\eta_{ij})_{\cal C}\eta_{ij}
\\[6mm] \dis \hspace{5mm}
=
 e^{{\cal A}(s-s_1)}
 \left(
 b_1-\sum_{(i,j)\in\Pi}(b_1,\eta_{ij})_{\cal C}\eta_{ij}
 \right)
-
 \sum_{(i,j)\in\Pi}
 \left(
 \int_{s_1}^{s_2}
 (e^{{\cal B}(s-\tau)}F(\tau),\eta_{ij})_{\cal C}d\tau
 \right)
 \eta_{ij}.
\end{array}
\]
%%%%%%%%%%%%%%%%%%%%%%%%%%%%%%%%%%%%%%%%%%%%%%%%%%%%%%%%%%%
% •ÏX"_
%%%%%%%%%%%%%%%%%%%%%%%%%%%%%%%%%%%%%%%%%%%%%%%%%%%%%%%%%%%
As a consequence,
we obtain from \eqref{express-eq}
\[
\begin{array}{lll}
\dis
 b(s)
\hspace{-2mm}&=&\hspace{-2mm} \dis
 e^{{\cal A}(s-s_1)}
 \left(
 b_1-\sum_{(i,j)\in\Pi}(b_1,\eta_{ij})_{\cal C}\eta_{ij}
 \right)
+
 \int_{s_1}^s
 e^{{\cal B}(s-\tau)}F(\tau)d\tau
\\[6mm]
\hspace{-2mm}&&\hspace{20mm} \dis
-
 \sum_{(i,j)\in\Pi}
 \left(
 \int_{s_1}^{s_2}
 \left( e^{{\cal B}(s-\tau)}F(\tau),\eta_{ij} \right)_{\cal C}d\tau
 \right)
 \eta_{ij}
\\[6mm]
\hspace{-2mm}&=&\hspace{-2mm} \dis
 e^{{\cal A}(s-s_1)}
 \left(
 b_1-\sum_{(i,j)\in\Pi}(b_1,\eta_{ij})_{\cal C}\eta_{ij}
 \right)
-
 \sum_{(i,j)\in\Pi}
 \left(
 \int_{s}^{s_2}
 \left( e^{-{\cal B}(s-\tau)}F(\tau),\eta_{ij} \right)_{\cal C}
 d\tau
 \right)
 \eta_{ij}
\\[4mm]
\hspace{-2mm}&&\hspace{0mm} \dis
 +
 \left(
 \int_{s_1}^s
 e^{{\cal B}(s-\tau)}F(\tau)
 d\tau
 - 
 \sum_{(i,j)\in\Pi}
 \left(
 \int_{s_1}^se^{{\cal B}(s-\tau)}F(\tau)
 d\tau,
 \eta_{ij}
 \right)_{\cal C}\eta_{ij}
 \right).
\end{array}
\]
% In the second equality,
% we used Lemma \ref{A'F-lem}.
Here
We put
\[
 \bar{\Pi} = \{(i,j)\in\N^2;\ \lambda_{ij}\leq\lambda^*\}.
\]
Furthermore
we define
\[
 Pb
=
 b-\sum_{(i,j)\in\bar{\Pi}}(b,\eta_{ij})_{\cal C}\eta_{ij}.
\]
Then
it holds that
\begin{equation}\label{bLong-eq}
\begin{array}{l}
\dis
 b(s)
=
 e^{{\cal A}(s-s_1)}P b_1
+
 e^{-\lambda^*(s-s_1)}(b_1,\eta_{1\ell})_{\cal C}\eta_{1\ell}
-
 \sum_{(i,j)\in\Pi}
 \left(
 \int_{s}^{s_2}
 \left( e^{-{\cal B}(s-\tau)}F(\tau),\eta_{ij} \right)_{\cal C}
 d\tau
 \right)
 \eta_{ij}
\\[6mm] \dis \hspace{15mm}
+
 \left(
 \int_{s_1}^s
 \left( e^{-{\cal B}(s-\tau)}F(\tau),\eta_{1\ell} \right)_{\cal C}
 d\tau
 \right)
 \eta_{1\ell}
+
 P\left(
 \int_{s_1}^s
 e^{{\cal B}(s-\tau)}F(\tau)d\tau
 \right).
\end{array}
\end{equation}

%%%%%%%%%%%%%%%%%%%%%%%%%%%%%%%%%%%%%%%%%%%%%%%%%%%%%%%%%%%
\begin{lem}\label{Long1-lem}
There exist $\delta>0$ and $c>0$ such that
for $y\in\R_+^n$ and $s_1<s<s_2$
\[
\left|
 e^{-\lambda^*(s-s_1)}(b_1,\eta_{1\ell})_{\cal C}\eta_{1\ell}
+
 e^{-\lambda^*s}\eta_{1\ell}
\right| 
<
 ce^{-\delta s_1}e^{-\lambda^*s}
 \left( 1+ |y|^{2\lambda^*-m+\gamma} \right).
\]
\end{lem}
%%%%%%%%%%%%%%%%%%%%%%%%%%%%%%%%%%%%%%%%%%%%%%%%%%%%%%%%%%%

%%%%%%%%%%%%%%%%%%%%%%%%%%%%%%%%%%%%%%%%%%%%%%%%%%%%%%%%%%%
\begin{proof}
Since $\|\eta_{1\ell}\|_{\cal C}=1$,
we easily see that
\[
 e^{-\lambda^*(s-s_1)}(b_1,\eta_{1\ell})_{\cal C}\eta_{1\ell}
+
 e^{-\lambda^*s}\eta_{1\ell}
=
 e^{-\lambda^*s}
 \left( e^{\lambda^*s_1}b_1+\eta_{1\ell},\eta_{1\ell} \right)_{\cal C}
 \eta_{1\ell}.
\]
Here
we recall that
$b_1(y)=\Phi(y,s_1)/\sigma(y)=\sum_{(i,j)\in\Pi}d_{ij}\eta_{ij}-e^{-\lambda^*s_1}\eta_\ell^*(y)$.
Hence
it holds that
\[
\begin{array}{lll}
\dis
 \left( e^{\lambda^*s_1}b_1+\eta_{1\ell},\eta_{1\ell} \right)_{\cal C}
\hspace{-2mm}&=&\hspace{-2mm} \dis
 \left(
 e^{\lambda s_1}\left( \sum_{(i,j)\in\Pi}d_{ij}\eta_{ij}-e^{-\lambda^*s_1}\eta_\ell^*
 \right)
+
 \eta_{1\ell},\eta_{1\ell} \right)_{\cal C}
\\[8mm]
\hspace{-2mm}&=&\hspace{-2mm} \dis
 (-\phi_\ell^*+\phi_{1\ell},\phi_{1\ell})_\rho.
\end{array}
\]
Therefore
by Lemma \ref{(*-ell)-lem},
we obtain the conclusion.
\end{proof}
%%%%%%%%%%%%%%%%%%%%%%%%%%%%%%%%%%%%%%%%%%%%%%%%%%%%%%%%%%%

%%%%%%%%%%%%%%%%%%%%%%%%%%%%%%%%%%%%%%%%%%%%%%%%%%%%%%%%%%%
\begin{lem}\label{Long3-lem}
There exist $\delta>0$ and $c>0$ such that
for $y\in\R_+^n$ and $s_1<s<s_2$
\[
\left|
 \sum_{(i,j)\in\Pi}
 \left(
 \int_{s}^{s_2}
 \left( e^{-{\cal B}(s-\tau)}F(\tau),\eta_{ij} \right)_{\cal C}
 d\tau
 \right)
 \eta_{ij}
\right|
 \leq
 ce^{-(\lambda^*+\delta)s}
 \left( 1+|y|^{2\lambda^*-m+\gamma} \right).
\]
\end{lem}
%%%%%%%%%%%%%%%%%%%%%%%%%%%%%%%%%%%%%%%%%%%%%%%%%%%%%%%%%%%

%%%%%%%%%%%%%%%%%%%%%%%%%%%%%%%%%%%%%%%%%%%%%%%%%%%%%%%%%%%
\begin{proof}
From Lemma \ref{A'F-lem},
we recall that
\[
 \left( e^{-{\cal B}(s-\tau)}F(\tau),\eta_{ij} \right)_{\cal C}
=
 e^{-\lambda_{ij}(s-\tau)}\left( F(\tau),\eta_{ij} \right)_{{\cal C},\pa\R_+^n}.
\]
Then
since $\varphi(y,s)\in A_{s_1,s_2}$ and $\lambda_{ij}<\lambda^*$ for $(i,j)\in\Pi$,
we obtain from Lemma \ref{Feta_ij-lem}
\[
 \int_{s}^{s_2}
 e^{-\lambda_{ij}(s-\tau)}
 \left| (F(\tau),\eta_{ij})_{{\cal C},\pa\R+^n} \right|
 d\tau
\leq
 c\int_{s}^{s_2} e^{-\lambda_{ij}(s-\tau)}e^{-(\lambda^*+\delta)\tau}d\tau
\leq
 ce^{-(\lambda^*+\delta)s}.
\]
Furthermore
since $-(\gamma-m)/2\leq\lambda_{ij}<\lambda^*$ for $(i,j)\in\Pi$,
it holds that
$|\eta_{ij}|<c(1+|y|^{2\lambda_{ij}-m+\gamma})<c(1+|y|^{2\lambda^*-m+\gamma})$
for $(i,j)\in\Pi$.
Therefore
the proof is completed.
\end{proof}
%%%%%%%%%%%%%%%%%%%%%%%%%%%%%%%%%%%%%%%%%%%%%%%%%%%%%%%%%%%

%%%%%%%%%%%%%%%%%%%%%%%%%%%%%%%%%%%%%%%%%%%%%%%%%%%%%%%%%%%
\begin{lem}\label{Long2-lem}
There exist $\delta>0$ and $c>0$ such that
for $y\in\R_+^n$ and $s_1<s<s_2$
\[
 \left|
 \left(
 \int_{s_1}^s
 \left( e^{-{\cal B}(s-\tau)}F(\tau),\eta_{1\ell} \right)_{\cal C}
 d\tau
 \right)
 \eta_{1\ell}
 \right|
\leq
 ce^{-\delta s_1}e^{-\lambda^*s}
 \left( 1+|y|^{2\lambda^*-m+\gamma} \right).
\]
\end{lem}
%%%%%%%%%%%%%%%%%%%%%%%%%%%%%%%%%%%%%%%%%%%%%%%%%%%%%%%%%%%

%%%%%%%%%%%%%%%%%%%%%%%%%%%%%%%%%%%%%%%%%%%%%%%%%%%%%%%%%%%
\begin{proof}
By the same way as in the proof of Lemma \ref{Long3-lem},
we note that
\[
 \left( e^{-{\cal B}(s-\tau)}F(\tau),\eta_{1\ell} \right)_{\cal C}
=
 e^{-\lambda_{1\ell}(s-\tau)}\left( F(\tau),\eta_{1\ell} \right)_{{\cal C},\pa\R_+^n}.
\]
Then
from Lemma \ref{Feta_ij-lem},
we see that
\[
 \left|
 \int_{s_1}^s e^{-\lambda^*(s-\tau)}(F(\tau),\eta_{1\ell})_{{\cal C},\pa\R_+^n}d\tau
 \right|
\leq
 c\int_{s_1}^s e^{-\lambda^*(s-\tau)}e^{-(\lambda^*+\delta)\tau}d\tau
\leq
 ce^{-\delta s_1}e^{-\lambda^*s}.
\]
Therefore
since $|\eta_{1\ell}|<c(1+|y|^{2\lambda^*-m+\gamma})$,
we obtain the conclusion.
\end{proof}
%%%%%%%%%%%%%%%%%%%%%%%%%%%%%%%%%%%%%%%%%%%%%%%%%%%%%%%%%%%

To derive the estimates of $e^{{\cal A}(s-s_1)}Pb_1$ and
$P(\int_{s_1}^se^{{\cal B}(s-\tau)}F(\tau)d\tau)$ in \eqref{bLong-eq},
we divide $O_{\text{Long}}$ into three parts.
\[
\begin{array}{lll}
 O_{\text{Long}}^{(\text{I})}
\hspace{-2mm}&=&\hspace{-2mm} \dis
 \left\{ (y,s);\ Ke^{-\omega s}<|y|<e^{1/4},\ s_1+1<s<s_2 \right\},
\\[3mm]
 O_{\text{Long}}^{(\text{II})}
\hspace{-2mm}&=&\hspace{-2mm} \dis
 \left\{ (y,s);\ e^{1/4}<|y|<\min\{e^{(s-s_1)/2},e^{\sigma s}\},\ s_1+1<s<s_2 \right\},
\\[3mm]
 O_{\text{Long}}^{(\text{III})}
\hspace{-2mm}&=&\hspace{-2mm} \dis
 \left\{ (y,s);\ |y|>e^{(s-s_1)/2},\ s_1+1<s<s_2 \right\}.
\end{array}
\]

%%%%%%%%%%%%%%%%%%%%%%%%%%%%%%%%%%%%%%%%%%%%%%%%%%%%%%%%%%%
\subsection{Long time I}
%%%%%%%%%%%%%%%%%%%%%%%%%%%%%%%%%%%%%%%%%%%%%%%%%%%%%%%%%%%

Here
we provide the estimate in
\[
 O_{\text{Long}}^{(\text{I})}=\{(y,s);\ Ke^{-\omega  s}<|y|<e^{1/4},\ s_1+1<s<s_2\}.
\]

%%%%%%%%%%%%%%%%%%%%%%%%%%%%%%%%%%%%%%%%%%%%%%%%%%%%%%%%%%
\begin{lem}\label{deltac2-lem}
There exist $\delta>0$ and $c>0$ such that
\[
 \|e^{{\cal A}(s-s_1)}P b_1\|_{\cal C}
\leq
 ce^{-(\lambda^*+\delta)s}.
\]
\end{lem}
%%%%%%%%%%%%%%%%%%%%%%%%%%%%%%%%%%%%%%%%%%%%%%%%%%%%%%%%%%

%%%%%%%%%%%%%%%%%%%%%%%%%%%%%%%%%%%%%%%%%%%%%%%%%%%%%%%%%%
\begin{proof}
Since $(P b_1, \eta_{ij})_{\cal C}=0$ for $(i,j)\in\bar{\Pi}$,
there exists $\delta>0$ such that
\[
 \|e^{{\cal A}(s-s_1)}P b_1\|_{\cal C}
\leq
 e^{-(\lambda^*+\delta)(s-s_1)}\|P b_1\|_{\cal C}.
\]
Here
we recall that
$b(y,s_1)=\Phi(y,s_1)/\sigma(y)$ and
$\Phi(y,s_1)=\sum_{(i,j)\in\Pi}d_{ij}\phi_{ij}-e^{-\lambda^*s_1}\phi_\ell^*$.
Then
by definition of $P$,
we observe that
\[
\begin{array}{lll}
\dis
 P b_1
\hspace{-2mm}&=&\hspace{-2mm} \dis
 P
 \left( \sum_{(i,j)\in\Pi}d_{ij}\phi_{ij}-e^{-\lambda^*s_1}\phi_\ell^* \right)
\\[8mm]
\hspace{-2mm}&=&\hspace{-2mm} \dis
 P
 \left( \sum_{(i,j)\in\Pi}d_{ij}\phi_{ij}-e^{-\lambda^*s_1}\phi_{1\ell} \right)
+
 e^{-\lambda^*s_1}
 
 P
 \left( \phi_{1\ell}-\phi_\ell^* \right)
\\[6mm]
\hspace{-2mm}&=&\hspace{-2mm} \dis
 e^{-\lambda^*s_1}
 P
 \left( \phi_{1\ell}-\phi_\ell^* \right).
\end{array}
\]
From Lemma \eqref{(*-ell)-lem},
there exists $\delta_1>0$ such that $\|\phi_\ell^*-\phi_{1\ell}\|_{\cal C}<ce^{-\delta_1s_1}$.
Therefore
we obtain
\[
 \|P b_1\|_{\cal C}
\leq
 ce^{-(\lambda^*+\delta_1)s_1}.
\]
Combining the above estimate,
we obtain the conclusion.
\end{proof}
%%%%%%%%%%%%%%%%%%%%%%%%%%%%%%%%%%%%%%%%%%%%%%%%%%%%%%%%%%

To estimate $Q(y,s)$,
we prepare the trace inequality for axial symmetric functions.

%%%%%%%%%%%%%%%%%%%%%%%%%%%%%%%%%%%%%%%%%%%%%%%%%%%%%%%%%%%
\begin{lem}\label{TracePolar-lem}
There exists $c>0$ independent of $r_1,r_2\in[0,\infty]$ $(r_1<r_2)$ and $g\in C(\R)$ such that
\[
 \int_{r_1}^{r_2}
 g(|\xi'|)|Q(\xi')|{\cal C}(\xi')d\xi'
\leq
 c\int_{r_1}^{r_2}
 g(|\xi|)
 \left( |\nabla Q(\xi)|+\frac{|Q(\xi)|}{|\xi|} \right)
 {\cal C}(\xi)d\xi
\]
for any smooth axial symmetric function $Q(\xi)=Q(|\xi'|,\xi_n)$.
\end{lem}
%%%%%%%%%%%%%%%%%%%%%%%%%%%%%%%%%%%%%%%%%%%%%%%%%%%%%%%%%%%

%%%%%%%%%%%%%%%%%%%%%%%%%%%%%%%%%%%%%%%%%%%%%%%%%%%%%%%%%%%
\begin{proof}
Let $q(\xi)$ be a smooth axial symmetric function
and
$\chi(\theta)$ be a cut off function such that
$\chi(\theta)=1$ if $\pi/4<\theta<\pi/2$, $\chi(\theta)=0$ if $0<\theta<\pi/8$.
Since $(\sin\theta)>\sin(\pi/8)$ for $\pi/8<\theta<\pi/2$,
we see that
\[
\begin{array}{l}
\dis
 \int_{r_1}^{r_2}|q(r,\pi/2)|r^{n-2}dr
\leq
 \int_{r_1}^{r_2}r^{n-2}dr
 \int_0^{\pi/2}\pa_\theta\Bigl( \chi(\theta)|q(r,\theta)| \Bigr)d\theta
\\[4mm] \dis \hspace{10mm}
\leq
 \left( 1+\|\chi_\theta\|_\infty \right)
 \int_{r_1}^{r_2}r^{n-2}dr
 \int_{\pi/8}^{\pi/2}\Bigl( |q(r,\theta)|+|q_\theta(r,\theta)| \Bigr)d\theta
\\[4mm] \dis \hspace{10mm}
\leq
 c\left( 1+\|\chi_\theta\|_\infty \right)
 \int_{r_1}^{r_2}r^{n-2}dr
 \int_{\pi/8}^{\pi/2}\Bigl( |q(r,\theta)|+|q_\theta(r,\theta)| \Bigr)
 (\sin\theta)^{n-2}d\theta.
\end{array}
\]
Therefore
any smooth axial symmetric function $q(\xi)$ satisfies
\[
 \int_{r_1}^{r_2}|q(\xi')|d\xi'
\leq
 c\int_{r_1}^{r_2}\left( |\pa_\theta q(\xi)|+|q(\xi)| \right)|\xi|^{-1}
 d\xi,
\]
where $c>0$ is a constant independent of $r_1,r_2$ and $q(\xi)$.
Here
we recall that
$c_1|\xi|^{-2\gamma}e^{-|\xi|/4}<{\cal C}(\xi)<c_2|\xi|^{-2\gamma}e^{-|\xi|/4}$.
Therefore
applying the above inequality with $q(\xi)=g(|\xi|)Q(\xi)|\xi|^{-2\gamma}e^{-|\xi|^2/4}$
and
using $|\pa_\theta Q(\xi)|\leq|\xi||\nabla Q(\xi)|$,
we obtain the conclusion.
\end{proof}
%%%%%%%%%%%%%%%%%%%%%%%%%%%%%%%%%%%%%%%%%%%%%%%%%%%%%%%%%%%

%%%%%%%%%%%%%%%%%%%%%%%%%%%%%%%%%%%%%%%%%%%%%%%%%%%%%%%%%%%
\begin{lem}\label{deltac4-lem}
There exist $\delta$ and $c>0$ such that
\[
 \left\| P\left( \int_{s_1}^se^{{\cal B}(s-\tau)}F(\tau)d\tau \right) \right\|_{\cal C}
\leq
 ce^{-(\lambda^*+\delta)s}.
\]
\end{lem}
%%%%%%%%%%%%%%%%%%%%%%%%%%%%%%%%%%%%%%%%%%%%%%%%%%%%%%%%%%%

%%%%%%%%%%%%%%%%%%%%%%%%%%%%%%%%%%%%%%%%%%%%%%%%%%%%%%%%%%%
\begin{proof}
Let $Q(y,s)=P(\int_{s_1}^se^{{\cal B}(s-\tau)}F(\tau)d\tau)$.
By definition  of $Q(y,s)$,
we find that $Q(y,s)$ solves
\[
\begin{cases}
 \dis
 Q_s = \frac{1}{\cal C}\nabla({\cal C}\nabla Q) + \left( \frac{\gamma-m}{2} \right)Q,
 & (y,s)\in\R_+^n\times(s_1,\infty),
 \\ \dis
 \pa_\nu Q = F,
 & (y,s)\in\pa\R_+^n\times(s_1,\infty),
 \\ \dis
 Q(y,s_1)=0,
 & y\in\R_+^n.
\end{cases}
\]
Multiplying this equation by $Q(y,s)$ and integrating over $\R_+^n$,
then we get
\begin{equation}\label{difQ-eq}
 \frac{1}{2}\frac{d}{ds}\|Q\|_{\cal C}^2
=
 -\|\nabla Q\|_{\cal C}^2 + \int_{\pa\R_+^n}F(\xi',s)Q(\xi',s){\cal C}(\xi')d\xi'.
\end{equation}
Then
since $H<K$,
Lemma \ref{f(Phi)-lem} implies
\[
\begin{array}{l}
\dis
 \int_{\pa\R_+^n}|F(\xi',s)Q(\xi',s)|{\cal C}d\xi'
\leq
 \int_0^{He^{-\omega s}}
 |\xi'|^{-(m+1)+\gamma}|Q|{\cal C}d\xi'
+
 e^{-2\lambda^*s}
 \int_{He^{-\omega s}}^1
|\xi'|^{-1+m-\gamma}|Q|{\cal C}d\xi'
\\[6mm] \dis \hspace{15mm}
+
 e^{-2\lambda^*s}
 \int_1^{e^{\sigma s}}
 |\xi'|^{-(m+1)+4\lambda^*+\gamma}|Q|{\cal C}d\xi'
+
 \int_{e^{\sigma s}}^\infty
 |\xi'|^{-(m+1)+\gamma}|Q|{\cal C}d\xi'
\\[6mm] \dis \hspace{10mm}
=
 J_1 + J_2 + J_3 +J_4.
\end{array}
\]
By Lemma \ref{TracePolar-lem},
we see that
\[
\begin{array}{lll}
\dis
 J_1
\hspace{-2mm}&\leq&\hspace{-2mm} \dis
 c\int_0^{He^{-\omega s}}|\xi|^{-(m+1)+\gamma}
 \left( |\nabla Q|+\frac{|Q|}{|\xi|} \right)
 {\cal C}d\xi
\\[6mm]
\hspace{-2mm}&\leq&\hspace{-2mm} \dis
 c\left( \int_0^{He^{-\omega s}}r^{-2(m+1)+n-1}dr \right)^{1/2}
 \left( \|\nabla Q\|_{\cal C}+\left\| \frac{Q}{|\xi|} \right\|_{L_{\cal C}^2(B_1)} \right)
\\[6mm]
\hspace{-2mm}&=&\hspace{-2mm} \dis
 c(He^{-\omega s})^{(n-2-2m)/2}
 \left( \|\nabla Q\|_{\cal C}+\left\| \frac{Q}{|\xi|} \right\|_{L_{\cal C}^2(B_1)} \right).
\end{array}
\]
Since $\lambda^*=\omega(\gamma-m)$,
we find that $\omega(n-2-2m)=2\lambda^*+\omega(n-2-2\gamma)$.
Therefore
we obtain
\begin{equation}\label{J_1-eq}
 J_1
\leq
 cH^{(n-2-2m)/2}e^{-(n-2-2\gamma)\omega s/2}e^{-\lambda^*s}
 \left( \|\nabla Q\|_{\cal C}+\left\| \frac{Q}{|\xi|} \right\|_{L_{\cal C}^2(B_1)} \right).
\end{equation}
Next we estimate $J_2$.
% Put $2d=(n-2)-2\gamma>0$.
% we can choose $d_1\in(0,1)$ such that $(n-2)-2\gamma-2(1-d_1)(\gamma-m)>0$.
Then
by the same way,
we see that
\[
\begin{array}{lll}
\dis
 J_2
\hspace{-2mm}&\leq&\hspace{-2mm} \dis
 ce^{-2\lambda^*s}
 \int_{He^{-\omega s}}^1
 |\xi|^{-1+m-\gamma}
 \left( |\nabla Q|+\frac{|Q|}{|\xi|} \right)
 {\cal C}d\xi
\\[6mm]
\hspace{-2mm}&\leq&\hspace{-2mm} \dis
 ce^{-2\lambda^*s}
 \left(
 \int_{He^{-\omega s}}^1
 |\xi|^{-2+2m-2\gamma}{\cal B}d\xi
 \right)^{1/2}
 \left(
 \|\nabla Q\|_{\cal C}+\left\| \frac{Q}{|\xi|} \right\|_{L_{\cal C}^2(B_1)}
 \right).
\end{array}
\]
Here
we put $2d=(n-2)-2\gamma>0$.
Then
since ${\cal B}(\xi)\sim|\xi|^{-2\gamma}$,
we observe that
\[
 \int_{He^{-\omega s}}^1
 |\xi|^{-2+2m-2\gamma}{\cal B}d\xi
\leq
 \begin{cases}
 \dis
 \left( He^{-\omega s} \right)^{-2(\gamma-m)+d}
 \int_{He^{-\omega s}}^1
 |\xi|^{-2-2\gamma-d}d\xi
 & \text{if } 2(\gamma-m)>d,
 \\[2mm] \dis
 \int_{He^{-\omega s}}^1
 |\xi|^{-2-2\gamma-d}d\xi
 & \text{if } 2(\gamma-m)\leq d.
 \end{cases}
\]
Therefore
since $-2-2\gamma-2d=d-n$,
we obtain
\begin{equation}\label{J_2-eq}
 J_2
\leq
 c\left(
 H^{-(\gamma-m)+d/2}e^{-d\omega s/2}e^{-\lambda^*s} + e^{-2\lambda^*s}
 \right)
 \left( \|\nabla Q\|_{\cal C}+\left\| \frac{Q}{|\xi|} \right\|_{L_{\cal C}^2(B_1)} \right).
\end{equation}
Furthermore
by the same calculation as above,
we see that
\begin{equation}\label{J_3-eq}
\begin{array}{lll}
\dis
 J_3
\hspace{-2mm}&\leq&\hspace{-2mm} \dis
 ce^{-2\lambda^*s}
 \int_1^{e^{\sigma s}}
 |\xi|^{-(m+1)+4\lambda^*+\gamma}
 \left( |\nabla Q|+\frac{|Q|}{|\xi|} \right)
 {\cal C}d\xi
\\[4mm]
\hspace{-2mm}&\leq&\hspace{-2mm} \dis
 ce^{-2\lambda^*s}
 \int_1^{e^{\sigma s}}
 |\xi|^{-(m+1)+4\lambda^*+\gamma}
 \left( |\nabla Q|+|Q| \right)
 {\cal C}d\xi
\\[4mm]
\hspace{-2mm}&\leq&\hspace{-2mm} \dis
 ce^{-2\lambda^*s}
 \left( \int_1^{e^{\sigma s}}r^{-2(m+1)+8\lambda^*+n-2}\rho(r)dr \right)^{1/2}
 \left( \|\nabla Q\|_{\cal C}+\|Q\|_{\cal C} \right).
\end{array}
\end{equation}
Finally
repeating the above argument,
we obtain
\begin{equation}\label{J_4-eq}
\begin{array}{lll}
\dis
 J_4
\hspace{-2mm}&\leq&\hspace{-2mm} \dis
 c\int_{e^{\sigma s}}^\infty
 |\xi|^{-(m+1)+\gamma}
 \left( |\nabla Q|+\frac{|Q|}{|\xi|} \right)
 {\cal C}d\xi
\\[4mm]
\hspace{-2mm}&\leq&\hspace{-2mm} \dis
 c\left(
 \int_{e^{\sigma s}}^\infty r^{-2(m+1)}\rho(r)r^{n-1}dr
 \right)^{1/2}
 \left( \|\nabla Q\|_{\cal C}+\|Q\|_{\cal C} \right)
\\[4mm]
\hspace{-2mm}&\leq&\hspace{-2mm} \dis
 ce^{-2\lambda^*s}
 \left(
 \int_{e^{\sigma s}}^\infty r^{-2(m+1)+4\lambda^*/\sigma}\rho(r)r^{n-1}dr
 \right)^{1/2}
 \left( \|\nabla Q\|_{\cal C}+\|Q\|_{\cal C} \right).
\end{array}
\end{equation}
As a consequence,
since $H=e^{a's_1}$ with $0<a'\ll1$,
by \eqref{J_1-eq}-\eqref{J_4-eq},
there exists $\delta>0$ such that
\[
 \int_{\pa\R_+^n}|F(\xi',s)Q(\xi',s)|{\cal C}d\xi'
% \leq
%  J_1+J_2+J_3+J_4
\leq
 ce^{-(1+\delta)\lambda^*s}
 \left(
 \|\nabla Q\|_{\cal C}+\|Q\|_{\cal C}+\left\| \frac{Q}{|\xi|} \right\|_{L_{\cal C}^2(B_1)}
 \right).
\]
Let $\chi(|\xi|)$ be a cut off function such that $\chi(r)=1$ if $0<r<1$ and $\chi(r)=0$ if $r>2$.
Here
applying the Caffarelli-Kohn-Nirenberg inequality \eqref{CaffarelliKN-eq}
with $m=\alpha=-2\gamma$, $r=p-2$, $a=1$,
then we get
\[
 \left\| \frac{Q}{|\xi|} \right\|_{L_{\cal B}^2(B_1)}
\leq
 \left\| \frac{\chi Q}{|\xi|} \right\|_{\cal B}
\leq
 c\|\nabla(\chi Q)\|_{\cal B}
\leq
 c\left( \|Q\|_{L_{\cal B}^2(B_2)}+\|\nabla Q\|_{L_{\cal B}^2(B_2)} \right).
\]
Since ${\cal C}(\xi)=\rho(\xi){\cal B}(\xi)$,
we find
$\|\cdot\|_{L_{\cal C}^2(B_1)}\leq\|\cdot\|_{L_{\cal B}^2(B_1)}$
and
$\|\cdot\|_{L_{\cal B}^2(B_2)}\leq e\|\cdot\|_{L_{\cal C}^2(B_2)}$.
Hence
we obtain
\[
  \int_{\pa\R_+^n}|F(\xi',s)Q(\xi',s)|{\cal C}d\xi'
\leq
 ce^{-(1+\delta)\lambda^*s}
 \left( \|\nabla Q\|_{\cal C}+\|Q\|_{\cal C} \right).
\]
Therefore
plugging this estimate into \eqref{difQ-eq},
we get
\[
 \frac{1}{2}\frac{d}{ds}
 \|Q(s)\|_{\cal C}^2
\leq
-
 \left( 1-e^{-\delta\lambda^*s} \right)
 \|\nabla Q(s)\|_{\cal C}^2
+
 e^{-\delta\lambda^*s}\|Q(s)\|_{\cal C}^2
+
 ce^{-(2+\delta)\lambda^*s}.
\]
Since
$(Q(s),\eta_{ij})_{\cal C}=0$ for $(i,j)\in\bar{\Pi}$ and $Q(s_1)=0$,
there exists $\delta'>0$ such that
\[
 \|Q(s)\|_{\cal C} \leq ce^{-(1+\delta')\lambda^*s},
\]
which completes the proof.
\end{proof}
%%%%%%%%%%%%%%%%%%%%%%%%%%%%%%%%%%%%%%%%%%%%%%%%%%%%%%%%%%%

%%%%%%%%%%%%%%%%%%%%%%%%%%%%%%%%%%%%%%%%%%%%%%%%%%%%%%%%%%%
\begin{lem}\label{deltac4'-lem}
There exist $\delta>0$ and $c>0$ such that
\[
 \left| e^{{\cal A}(s-s_1)}Pb_1 \right|
+
 \left| P\left( \int_{s_1}^se^{{\cal B}(s-\tau)}F(\tau)d\tau \right) \right|
<
 ce^{-\delta s_1}e^{-\lambda^*s}
\hspace{5mm}\mathrm{for}\ (y,s)\in O_{\mathrm{Long}}^{(\mathrm{I})}.
\]
\end{lem}
%%%%%%%%%%%%%%%%%%%%%%%%%%%%%%%%%%%%%%%%%%%%%%%%%%%%%%%%%%%

%%%%%%%%%%%%%%%%%%%%%%%%%%%%%%%%%%%%%%%%%%%%%%%%%%%%%%%%%%%
\begin{proof}
Since $|y|<e^{1/4}$ and $s>s_1+1$,
we get from Lemma \ref{L^inftyL^2-lem}
\[
\begin{array}{lll}
\dis
 \left| e^{{\cal A}(s-s_1)}Pb_1 \right|
\hspace{-2mm}&=&\hspace{-2mm} \dis
 \left| e^{{\cal A}/2}e^{{\cal A}(s-s_1-1/2)}Pb_1 \right|
=
 e^{(\gamma-m)/4}\left| e^{{\cal A}_0/2}e^{{\cal A}(s-s_1-1/2)}Pb_1 \right|
\\[4mm]
\hspace{-2mm}&\leq&\hspace{-2mm} \dis
 c\left\| e^{{\cal A}(s-s_1-1/2)}Pb_1 \right\|_{\cal C}.
\end{array}
\]
Therefore
Lemma \ref{deltac2-lem} implies
\[
 \left| e^{{\cal A}(s-s_1)}Pb_1 \right|
\leq
 ce^{-(\lambda^*+\delta)s}.
\]
Next
we estimate $P(\int_{s_1}^se^{{\cal B}(s-\tau)}F(\tau)d\tau)$.
We divide the integral into two parts.
\[
 P\left(
 \int_{s_1}^se^{{\cal B}(s-\tau)}F(\tau)d\tau
 \right)
=
 P\left(
 \int_{s_1}^{s-1}
 e^{{\cal B}(s-\tau)}F(\tau)d\tau
 \right)
+
 P\left(
 \int_{s-1}^s
 e^{{\cal B}(s-\tau)}F(\tau)d\tau
 \right).
\]
From Lemma \ref{A'F-lem},
we recall that
$e^{{\cal B}(s-\tau)}F(\tau)=e^{{\cal A}/2}e^{{\cal B}(s-\tau-1/2)}F(\tau)$
for $s_1<\tau<s-1$.
Hence
by the same way as above,
we obtain
\[
\begin{array}{lll}
\dis
 \left|
 P\left(
 \int_{s_1}^{s-1}
 e^{{\cal B}(s-\tau)}F(\tau)d\tau
 \right)
 \right|
\hspace{-2mm}&=&\hspace{-2mm} \dis
 \left|
 e^{{\cal A}/2}
 P\left(
 \int_{s_1}^{s-1}
 e^{{\cal B}(s-\tau-1/2)}F(\tau)d\tau
 \right)
 \right|
\\[4mm]
\hspace{-2mm}&\leq&\hspace{-2mm} \dis
 c\left\|
 P\left(
 \int_{s_1}^{s-1}
 e^{{\cal B}(s-\tau-1/2)}F(\tau)d\tau
 \right)
 \right\|_{\cal C}.
\end{array}
\]
Therefore
by Lemma \ref{deltac4-lem},
it follows that
\[
 \left|
 P\left(
 \int_{s_1}^{s-1}
 e^{{\cal B}(s-\tau)}F(\tau)d\tau
 \right)
 \right|
\leq
 ce^{-(\lambda^*+\delta)s}.
\]
Now
we estimate
$P(\int_{s-1}^se^{{\cal B}(s-\tau)}F(\tau)d\tau)$.
Then
by Lemma \ref{ShortS_2-lem},
we see that
\[
 \left|
 \int_{s-1}^s
 e^{{\cal B}(s-\tau)}F(\tau)d\tau
 \right|
\leq
 ce^{-\delta s_1}e^{-\lambda^*s}
\hspace{5mm}\text{for } Ke^{-\omega s}<|y|<e^{\sigma s}.
\]
Furthermore
Lemma \ref{A'F-lem} and Lemma \ref{Feta_ij-lem} imply
\[
\begin{array}{lll}
\dis
 \left|
 \left(
 \int_{s-1}^s
 e^{{\cal B}(s-\tau)}F(\tau)d\tau,
 \eta_{ij}
 \right)_{\cal C}
 \right|
\hspace{-2mm}&=&\hspace{-2mm} \dis
 \left|
 \int_{s-1}^s
 e^{-\lambda_{ij}(s-\tau)}
 \left(F(\tau),\eta_{ij}\right)_{{\cal C},\pa\R_+^n}
 d\tau
 \right|
\\[6mm]
\hspace{-2mm}&\leq&\hspace{-2mm} \dis
 c\int_{s-1}^s
 e^{-\lambda_{ij}(s-\tau)}e^{-(\lambda^*+\delta)s}
 d\tau
\leq
 ce^{-(\lambda^*+\delta)s}.
\end{array}
\]
Here
we recall that $Pb=b-\sum_{(i,j)\in\bar{\Pi}}(b,\eta_{ij})_{\cal C}\eta_{ij}$.
Therefore
since $|y|<e^{1/4}$,
we obtain
\[
 \left|
 P\left(
 \int_{s-1}^s
 e^{{\cal B}(s-\tau)}F(\tau)d\tau
 \right)
 \right|
\leq
 ce^{-\delta s_1}e^{-\lambda^*s}.
\]
Thus
the proof is completed.
\end{proof}
%%%%%%%%%%%%%%%%%%%%%%%%%%%%%%%%%%%%%%%%%%%%%%%%%%%%%%%%%%%

Therefore
combining Lemma \ref{Long1-lem}{\h-\h}Lemma \ref{Long2-lem} and Lemma \ref{deltac4'-lem},
we obtain the following result.

%%%%%%%%%%%%%%%%%%%%%%%%%%%%%%%%%%%%%%%%%%%%%%%%%%%%%%%%%%%
\begin{pro}\label{Longtime1-pro}
There exist $\delta>0$ and $c>0$ such that
if $\varphi(y,s)$ satisfies \eqref{AssumeAPi-eq},
then it holds that
\[
 \left| b(y,s)+e^{-\lambda^*s}\eta_{1\ell}(y) \right|
\leq
 ce^{-\delta s_1}e^{-\lambda^*s}
 \left( 1+|y|^{2\lambda^*-m+\gamma} \right)
\hspace{5mm}\mathrm{for}\ (y,s)\in O_{\mathrm{Long}}^{(\mathrm{I})}.
\]
\end{pro}
%%%%%%%%%%%%%%%%%%%%%%%%%%%%%%%%%%%%%%%%%%%%%%%%%%%%%%%%%%%

%%%%%%%%%%%%%%%%%%%%%%%%%%%%%%%%%%%%%%%%%%%%%%%%%%%%%%%%%%%
\subsection{Long time II}
%%%%%%%%%%%%%%%%%%%%%%%%%%%%%%%%%%%%%%%%%%%%%%%%%%%%%%%%%%%

In this subsection,
we derive the estimate in
\[
 O_{\text{Long}}^{(\text{II})}=\{(y,s);\ e^{1/4}<|y|<\min\{e^{(s-s_1)/2},e^{\sigma s}\},\ s_1+1<s<s_2\}.
\]
Throughout this section,
we always assume $(y,s)\in O_{\text{Long}}^{(\text{II})}$ and \eqref{AssumeAPi-eq}.
Since $e^{1/4}<|y|<e^{(s-s_1)/2}$,
we can fix $s'\in(s_1,s)$ such that
\begin{equation}\label{s'Def-eq}
 |y| = e^{(s-s')/2}.
\end{equation}
Furthermore
since $e^{(s-s')/2}=|y|>e^{1/4}$,
it follows that
\begin{equation}\label{below-eq}
 1-e^{-(s-s')} > \left( \sqrt{2}-1 \right)/\sqrt{2}.
\end{equation}
For simplicity of notations,
$s'$ always stands for \eqref{s'Def-eq} in this subsection.

%%%%%%%%%%%%%%%%%%%%%%%%%%%%%%%%%%%%%%%%%%%%%%%%%%%%%%%%%%%
\begin{lem}\label{deltac2'-lem}
There exist $\delta>0$ and $c>0$ such that
\[
 |e^{{\cal A}(s-s_1)}Pb_1|
\leq
 ce^{-\delta s_1}e^{-\lambda^*s}
\hspace{5mm}\mathrm{for}\
 (y,s)\in O_{\mathrm{Long}}^{(\mathrm{II})}.
\]
\end{lem}
%%%%%%%%%%%%%%%%%%%%%%%%%%%%%%%%%%%%%%%%%%%%%%%%%%%%%%%%%%%

%%%%%%%%%%%%%%%%%%%%%%%%%%%%%%%%%%%%%%%%%%%%%%%%%%%%%%%%%%%
\begin{proof}
Since $|y|=e^{(s-s')/2}$,
applying Lemma \ref{L^inftyL^2-lem},
we see that
\[
\begin{array}{lll}
 \left| e^{{\cal A}(s-s')}\left( e^{{\cal A}(s'-s_1)}P b_1 \right) \right|
\hspace{-2mm}&=&\hspace{-2mm} \dis
 e^{(\gamma-m)(s-s')/2}
 \left| e^{{\cal A}_0(s-s')}\left( e^{{\cal A}(s'-s_1)}P b_1 \right) \right|
\\[3mm]
\hspace{-2mm}&\leq&\hspace{-2mm} \dis
\frac{ce^{(\gamma-m)(s-s')/2}}{(1-e^{-(s-s')})^{n}}
 \left\| e^{{\cal A}(s'-s_1)}P b_1 \right\|_{\cal C}.
\end{array}
\]
Therefore
we obtain from \eqref{s'Def-eq} and \eqref{below-eq}
\[
\begin{array}{lll}
\dis
 \left| e^{{\cal A}(s-s')}\left( e^{{\cal A}(s'-s_1)}P b_1 \right) \right|
\hspace{-2mm}&<&\hspace{-2mm} \dis
 ce^{(\gamma-m)(s-s')/2}e^{-\lambda^*(s'-s_1)}
 \|P b_1\|_{\cal C}
\\[4mm]
\hspace{-2mm}&=&\hspace{-2mm} \dis
 ce^{(\gamma-m)(s-s')/2}e^{\lambda^*(s-s')}
 e^{-\lambda^*s}e^{\lambda^*s_1}\|P b_1\|_{\cal C}
\\[4mm]
\hspace{-2mm}&=&\hspace{-2mm} \dis
 ce^{-\lambda^*s}e^{\lambda^*s_1}\|P b_1\|_{\cal C}
 |y|^{2\lambda^*-m+\gamma}.
\end{array}
\]
Here
we recall that $\|P b_1\|_{\cal C}<ce^{-(\lambda^*+\delta)s_1}$
(see proof of Lemma \ref{deltac2-lem}).
Therefore
we obtain the conclusion.
\end{proof}
%%%%%%%%%%%%%%%%%%%%%%%%%%%%%%%%%%%%%%%%%%%%%%%%%%%%%%%%%%%

%%%%%%%%%%%%%%%%%%%%%%%%%%%%%%%%%%%%%%%%%%%%%%%%%%%%%%%%%%%
\begin{lem}\label{deltac3'-lem}
There exist $\delta>0$ and $c>0$ such that
\[
 \left|
 P\left(
 \int_{s_1}^s
 e^{{\cal B}(s-\tau)}F(\tau)d\tau
 \right)
 \right|
<
 ce^{-\delta s_1}e^{-\lambda^*s}
\hspace{5mm}\mathrm{for}\ (y,s)\in O_{\mathrm{Long}}^{(\mathrm{II})}.
\]
\end{lem}
%%%%%%%%%%%%%%%%%%%%%%%%%%%%%%%%%%%%%%%%%%%%%%%%%%%%%%%%%%%

%%%%%%%%%%%%%%%%%%%%%%%%%%%%%%%%%%%%%%%%%%%%%%%%%%%%%%%%%%%
\begin{proof}
By using $Pb=b-\sum_{(i,j)\in\bar{\Pi}}(b,\eta_{ij})_{\cal C}\eta_{ij}$,
we divide the integral into four parts. 
\begin{equation}\label{PFthree-eq}
\begin{array}{lll}
\dis
 P\left(
 \int_{s_1}^s
 e^{{\cal B}(s-\tau)}F(\tau)d\tau
 \right)
\hspace{-2mm}&=&\hspace{-2mm} \dis
 P\left(
 \int_{s_1}^{s'}e^{{\cal B}(s-\tau)}F(\tau)d\tau
 \right)
+
 \left( \int_{s'}^{s-1}+\int_{s-1}^s \right)
 e^{{\cal B}(s-\tau)}F(\tau)d\tau
\\[4mm]
\hspace{-2mm}&&\hspace{0mm} \dis
 -
 \sum_{(i,j)\in\bar{\Pi}}
 \left(
 \int_{s'}^s
 (e^{{\cal B}(s-\tau)}F(\tau),\eta_{ij})_{\cal C}d\tau
 \right)
 \eta_{ij}.
\end{array}
\end{equation}
Since
$e^{{\cal B}(s-\tau)}=e^{{\cal A}(s-s')}e^{{\cal B}(s'-\tau)}$ for $\tau<s'<s$
(see Lemma \ref{A'F-lem}),
the first integral of \eqref{PFthree-eq} is written by
\[
 P\left(
 \int_{s_1}^{s'}
 e^{{\cal B}(s-\tau)}F(\tau)d\tau
 \right)
=
 e^{{\cal A}(s-s')}
 P\left(
 \int_{s_1}^{s'}
 e^{{\cal B}(s'-\tau)}F(\tau)
 d\tau
 \right).
\]
Then
since $e^{{\cal A}(s-s')}=e^{(\gamma-m)(s-s')/2}e^{{\cal A}_0(s-s')}$,
by using $|y|=e^{(s-s')/2}$,
we get from Lemma \ref{L^inftyL^2-lem},
\[
\begin{array}{lll}
\dis
 \left|
 e^{{\cal A}(s-s')}
 P\left(
 \int_{s_1}^{s'}e^{{\cal B}(s'-\tau)}F(\tau)
 d\tau
 \right)
 \right|
\hspace{-2mm}&<&\hspace{-2mm} \dis
 \frac{ce^{(\gamma-m)(s-s')/2}}{(1-e^{-(s-s')})^n}
 \left\|
 P\left(
 \int_{s_1}^{s'}e^{{\cal B}(s'-\tau)}F(\tau)
 d\tau
 \right)
 \right\|_{\cal C}
\\[6mm]
\hspace{-2mm}&=&\hspace{-2mm} \dis
 \frac{c|y|^{\gamma-m}}{(1-e^{-(s-s')})^n}
 \left\|
 P\left(
 \int_{s_1}^{s'}e^{{\cal B}(s'-\tau)}F(\tau)
 d\tau
 \right)
 \right\|_{\cal C}.
\end{array}
\]
Therefore
by using Lemma \ref{deltac4-lem}, \eqref{s'Def-eq} and \eqref{below-eq},
we obtain
\[
\begin{array}{lll}
\dis
 \left|
 e^{{\cal A}(s-s')}
 P\left(
 \int_{s_1}^{s'}e^{{\cal B}(s'-\tau)}F(\tau)
 d\tau
 \right)
 \right|
\hspace{-2mm}&\leq&\hspace{-2mm} \dis
 ce^{-(\lambda^*+\delta)s'}|y|^{\gamma-m}
=
 e^{-\delta s'}e^{-\lambda^*s}ce^{\lambda^*(s-s')}|y|^{\gamma-m}
\\[4mm]
\hspace{-2mm}&=&\hspace{-2mm} \dis
 ce^{-\delta s'}e^{-\lambda^*s}|y|^{2\lambda^*-m+\gamma}.
\end{array}
\]
Next
we estimate the second integral of \eqref{PFthree-eq}.
Since $|y|=e^{(s-s')/2}$,
we obtain from Lemma \ref{LongIII2-lem}
\[
 \left|
 \int_{s'}^{s-1}
 e^{{\cal B}(s-\tau)}F(\tau)d\tau
 \right|
\leq
 ce^{-\delta s_1}e^{-\lambda^*s}
 \left( 1+|y|^{2\lambda^*-m+\gamma} \right).
\]
To estimate the third term in \eqref{PFthree-eq},
we apply Lemma \ref{ShortS_2-lem} and obtain
\[
 \left| \int_{s-1}^se^{{\cal B}(s-\tau)}F(\tau) \right|
\leq
 ce^{-\lambda^*s}e^{-\delta s_1}
 \left( 1+|y|^{2\lambda^*-m+\gamma} \right)
\hspace{5mm}\text{for } Ke^{-\omega s}<|y|<e^{\sigma s}.
\]
Finally
we provide the estimate of the last term in \eqref{PFthree-eq}.
Then
since $\lambda_{ij}\leq\lambda^*$ for $(i,j)\in\bar{\Pi}$,
from Lemma \ref{A'F-lem} and Lemma \ref{Feta_ij-lem},
we get
\[
\begin{array}{l}
\dis
 \left|
 \left(
 \int_{s'}^s(e^{{\cal B}(s-\tau)}F(\tau),\eta_{ij})_{\cal C}
 d\tau
 \right)
 \right|
=
 \left|
 \left(
 \int_{s'}^se^{-\lambda_{ij}(s-\tau)}(F(\tau),\eta_{ij})_{{\cal C},\pa\R_+^n}
 d\tau
 \right)
 \right|
\\[6mm] \dis \hspace{10mm}
\leq
 c\left(
 \int_{s'}^s
 e^{-\lambda_{ij}(s-\tau)}e^{-(\lambda^*+\delta)\tau}
 d\tau
 \right)
\leq
 ce^{-\lambda_{ij}(s-s')}e^{-(\lambda^*+\delta)s'}
\hspace{5mm}\text{for } (i,j)\in\bar{\Pi}.
% \\[4mm]
% \hspace{-2mm}&\leq&\hspace{-2mm} \dis
%  ce^{-\lambda_{ij}(s-s')}e^{-(\lambda^*+\delta)s'}
%  |y|^{2\lambda_{ij}-m+\gamma}
% \\[2mm]
% \hspace{-2mm}&=&\hspace{-2mm} \dis
%  ce^{-\lambda_{ij}(s-s')}e^{\lambda^*(s-s')}e^{-\lambda^*s}e^{-\delta s'}
%  |y|^{2\lambda_{ij}-m+\gamma}.
\end{array}
\]
From Lemma \ref{2A-lem},
we recall that
$|\eta_{ij}|<c|y|^{2\lambda_{ij}-m+\gamma}$ for $|y|>e^{1/4}$.
Therefore
by using $|y|=e^{(s-s')/2}$,
we obtain
\[
 \left|
 \sum_{(i,j)\in\bar{\Pi}}
 \left(
 \int_{s'}^s
 \left( e^{{\cal B}(s-\tau)}F(\tau),\eta_{ij} \right)_{\cal C}
 d\tau
 \right)
 \eta_{ij}
 \right|
\leq
  ce^{-\delta s'}e^{-\lambda^*s}
 |y|^{2\lambda^*-m+\gamma}.
\]
Thus
the proof is completed.
\end{proof}
%%%%%%%%%%%%%%%%%%%%%%%%%%%%%%%%%%%%%%%%%%%%%%%%%%%%%%%%%%%

Combining Lemma \ref{deltac2'-lem}{\h-\h}Lemma \ref{deltac3'-lem}
and Lemma \ref{Long1-lem}{\h-\h}Lemma \ref{Long2-lem},
we obtain the following result.

%%%%%%%%%%%%%%%%%%%%%%%%%%%%%%%%%%%%%%%%%%%%%%%%%%%%%%%%%%%
\begin{pro}\label{Longtime2-pro}
There exist $\delta>0$ and $c>0$ such that
if $\varphi(y,s)$ satisfies \eqref{AssumeAPi-eq},
then it holds that
\[
 \left| b(y,s)+e^{-\lambda^*s}\eta_{1\ell}(y) \right|
\leq
 ce^{-\delta s_1}e^{-\lambda^*s}
 \left( 1+|y|^{2\lambda^*-m+\gamma} \right)
\hspace{5mm}\mathrm{for}\ (y,s)\in O_{\mathrm{Long}}^{(\mathrm{II})}.
\]
\end{pro}
%%%%%%%%%%%%%%%%%%%%%%%%%%%%%%%%%%%%%%%%%%%%%%%%%%%%%%%%%%%

%%%%%%%%%%%%%%%%%%%%%%%%%%%%%%%%%%%%%%%%%%%%%%%%%%%%%%%%%%%
\subsection{Long time III}
%%%%%%%%%%%%%%%%%%%%%%%%%%%%%%%%%%%%%%%%%%%%%%%%%%%%%%%%%%%

Finally
we provide the estimate in
\[
 O_{\text{Long}}^{(\text{III})}=\{\min\{e^{\sigma s},e^{(s-s_1)/2}\}<|y|<e^{\sigma s},\ s_1+1<s<s_2\}.
\]
Without loss of generality,
we can assume $e^{(s-s_1)/2}<e^{\sigma s}$.
Furthermore
throughout this subsection,
we always assume
$|d|<\epsilon_1e^{-\lambda^*s_1}$, $\varphi(y,s)\in A_{s_1,s_2}$
and $(y,s)\in O_{\text{Long}}^{(\text{III})}$.
Here
we recall that $b(y,s)$ is given by
\[
 b(s)
=
 e^{{\cal A}(s-s_1)}P_\ell b_1
-
 e^{-\lambda^*s}\eta_{1\ell}
+
 \int_{s_1}^se^{{\cal B}(s-\tau)}F(\tau)d\tau,
\]
where $P_\ell b_1=b_1+e^{-\lambda^*s_1}\eta_{1\ell}$.

%%%%%%%%%%%%%%%%%%%%%%%%%%%%%%%%%%%%%%%%%%%%%%%%%%%%%%%%%%%
\begin{lem}\label{LongIII-lem}
Let be $\nu(s)$ given in Lemma {\rm\ref{phi-initial-lem}}.
Then there exist $\delta>0$ and $c>0$ such that
\[
 \left| e^{{\cal A}(s-s_1)}P_\ell b_1 \right|
\leq
 c\left( e^{-\delta s_1}+\nu(s_1)+e^{\lambda^*s_1}d_{\max} \right)
 e^{-\lambda^*s}\left( 1+|y|^{2\lambda^+-m+\gamma} \right)
\]
for $(y,s)\in O_{\mathrm{Long}}^{(\mathrm{III})}$.
\end{lem}
%%%%%%%%%%%%%%%%%%%%%%%%%%%%%%%%%%%%%%%%%%%%%%%%%%%%%%%%%%%

%%%%%%%%%%%%%%%%%%%%%%%%%%%%%%%%%%%%%%%%%%%%%%%%%%%%%%%%%%%
\begin{proof}
We divide the integral into four parts.
\begin{equation}\label{C_1-eq}
 e^{{\cal A}(s-s_1)}P_\ell b_1
=
 \left(
 \int_0^{He^{-\omega s_1}}
 +
 \int_{He^{-\omega s_1}}^1
 +
 \int_1^{e^{\varrho s_1}}
 +
 \int_{e^{\varrho s_1}}^\infty
 \right)
 \Gamma(y,\xi,s-s_1)P_\ell b_1{\cal B}(\xi)d\xi.
\end{equation}
Since $s>s_1+1$,
Proposition \ref{Rep-pro} implies
\begin{equation}\label{tildeGamma-eq}
 \Gamma(y,\xi,s-s_1)
\leq
 ce^{(\gamma-m)(s-s_1)/2}\left( |\xi|+1 \right)^{2\gamma}
 \exp\left( -c\left| e^{-(s-s_1)/2}y-\xi \right|^2 \right).
\end{equation}
Therefore
we get from Lemma \ref{phi-initial-lem}
\[
\begin{array}{l}
\dis
 \int_0^{He^{-\omega s_1}}
 \Gamma(y,\xi,s-s_1)P_\ell b_1(\xi){\cal B}(\xi)d\xi
\\[2mm] \dis \hspace{10mm}
\leq
 ce^{(\gamma-m)(s-s_1)/2}
 \int_0^{He^{-\omega s_1}}
 \left( |\xi|^{-m+\gamma}+e^{-\lambda^*s} \right)
 |\xi|^{-2\gamma}
 d\xi
\\[4mm] \dis \hspace{10mm}
\leq
ce^{(\gamma-m)(s-s_1)/2}
 \left(
 \left( He^{-\omega s_1} \right)^{n-(\gamma+m)}+e^{-\lambda^*s_1}\left( He^{-\omega s_1} \right)^{n-2\gamma}
 \right)
\\[4mm] \dis \hspace{10mm}
\leq
 ce^{(\gamma-m)(s-s_1)/2}
 \left( H^{n-(\gamma+m)}+H^{n-2\gamma} \right)e^{-(n-2\gamma)\omega s_1}e^{-\lambda^*s_1}.
\end{array}
\]
Since $|y|>e^{(s-s_1)/2}$,
we note that
\begin{equation}\label{|y|exp-eq}
 e^{(\gamma-m)(s-s_1)/2}\cdot e^{\lambda^*(s-s_1)} \leq |y|^{\gamma-m+2\lambda^*}.
\end{equation}
Therefore
we obtain from $\gamma>m$
\[
 \int_0^{He^{-\omega s_1}}
 \Gamma(y,\xi,s-s_1)P_\ell b_1(\xi){\cal B}(\xi)d\xi
\leq
 cH^{n-(\gamma+m)}e^{-(n-2\gamma)\omega s_1}e^{-\lambda^*s}
 |y|^{2\lambda^*-m+\gamma}.
\]
Next we estimate the second integral in \eqref{C_1-eq}.
By Lemma \ref{phi-initial-lem},
we recall that
\begin{equation}\label{hatb_1-eq}
 \left| P_\ell b_1 \right|
\leq
 c\left( 
 \nu(s_1)e^{-\lambda^* s_1}+d_{\max}
 \right)
 \left( 1+|y|^{2\lambda^*-m+\gamma} \right)
\hspace{5mm}\text{for }
He^{-\omega s_1}<|y|<e^{\varrho s_1}.
\end{equation}
Plugging \eqref{tildeGamma-eq} and \eqref{hatb_1-eq} into the second integral in \eqref{C_1-eq},
we get
\[
\begin{array}{l}
\dis
 \int_{He^{-\omega s_1}}^1
 \Gamma(y,\xi,s-s_1)P_\ell b_1(\xi){\cal B}(\xi)d\xi
\\ \dis \hspace{25mm}
\leq
 c\left( \nu(s_1)e^{-\lambda^* s_1}+d_{\max} \right)
 e^{(\gamma-m)(s-s_1)/2}
 \int_{He^{-\omega s_1}}^1|\xi|^{-2\gamma}
 d\xi.
\end{array}
\]
Therefore
since $\gamma<(n-2)/2$,
by using \eqref{|y|exp-eq},
we obtain
\[
 \int_{He^{-\omega s_1}}^1
 \Gamma(y,\xi,s-s_1)P_\ell b_1(\xi){\cal B}(\xi)d\xi
\leq
 c\left( \nu(s_1)+d_{\max}e^{\lambda^* s_1} \right)
 e^{-\lambda^*s}
 |y|^{\gamma-m+2\lambda^*}.
\]
From \eqref{tildeGamma-eq} and \eqref{hatb_1-eq},
the third integral in \eqref{C_1-eq} is estimated by
\[
\begin{array}{l}
\dis
 \int_1^{e^{\varrho s_1}}
 \Gamma(y,\xi,s-s_1)P_\ell b_1(\xi){\cal B}(\xi)d\xi
\leq
 ce^{(\gamma-m)(s-s_1)/2}\left( 
 \nu(s_1)e^{-\lambda^* s_1}+d_{\max}
 \right)
\\[3mm] \dis \hspace{30mm}
\times
 \int_1^{e^{\varrho s_1}}|\xi|^{2\lambda^*-m+\gamma}
 \exp\left( -c\left| e^{-(s-s_1)/2}y-\xi \right|^2 \right)
 d\xi.
\end{array}
\]
We put
\[
\begin{array}{c}
\dis
 D_1 = \{1<|\xi|<e^{\varrho s_1};|e^{-(s-s_1)/2}y-\xi|<|\xi|/2\},
\\[2mm] \dis
 D_2 = \{1<|\xi|<e^{\varrho s_1};|e^{-(s-s_1)/2}y-\xi|>|\xi|/2\}.
\end{array}
\]
For the case $\xi\in D_1$,
we note that $|\xi|<2e^{-(s-s_1)/2}|y|$.
Hence we see that
\[
\begin{array}{l}
\dis
 \int_{D_1}
 |\xi|^{2\lambda^*-m+\gamma}
 \exp\left( -c\left| e^{-(s-s_1)/2}y-\xi \right|^2 \right)
 d\xi
\\[4mm] \dis \hspace{20mm}
\leq
 c\left( e^{-(s-s_1)/2}|y| \right)^{2\lambda^*-m+\gamma}
 \int_{D_1}
 \exp\left( -c\left| e^{-(s-s_1)/2}y-\xi \right|^2 \right)
 d\xi
\\[6mm] \dis \hspace{20mm}
\leq
 ce^{-(2\lambda^*-m+\gamma)(s-s_1)/2}|y|^{2\lambda^*-m+\gamma}.
\end{array}
\]
Furthermore
by definition of $D_2$,
we get
\[
 \int_{D_2}|\xi|^{2\lambda^*-m+\gamma}
 \exp\left( -c\left| e^{-(s-s_1)/2}y-\xi \right|^2 \right)
 d\xi
\leq
 \int_{\R_+^n}|\xi|^{2\lambda^*-m+\gamma}
 e^{-c|\xi|^2}
 d\xi
\leq
 c.
\]
Therefore
we obtain from \eqref{|y|exp-eq}
\[
\begin{array}{l}
\dis
 \int_1^{e^{\varrho s_1}}
 \Gamma(y,\xi,s-s_1)P_\ell b_1(\xi){\cal B}(\xi)d\xi
\leq
 ce^{(\gamma-m)(s-s_1)/2}
 \left( 
 \nu(s_1)e^{-\lambda^* s_1}+d_{\max}
 \right)
\\[2mm] \dis \hspace{30mm}
\times
 \left( \int_{D_1}+\int_{D_2} \right)
 |\xi|^{2\lambda^*-m+\gamma}
 \exp\left( -c\left| e^{-(s-_1)/2}y-\xi \right|^2 \right)
 d\xi
\\[4mm] \dis \hspace{5mm}
\leq
 c\left( 
 \nu(s_1)e^{-\lambda^* s_1}+d_{\max}
 \right)
 \left( e^{-\lambda^*(s-s_1)}|y|^{2\lambda^*-m+\gamma}+e^{(\gamma-m)(s-s_1)/2} \right)
\\[4mm] \dis \hspace{5mm}
\leq
 c\left( 
 \nu(s_1)+e^{\lambda^* s_1}d_{\max}
 \right)
 e^{-\lambda^*s}
 |y|^{2\lambda^*-m+\gamma}.
\end{array}
\]
Finally
we provide the estimate of the fourth term in \eqref{C_1-eq}.
Here
we note that
\[
\begin{array}{lll}
\dis
 e^{-(s-s_1)/2}|y|
\hspace{-2mm}&=&\hspace{-2mm} \dis
 e^{-(s-s_1)/2}e^{\sigma s}
=
 e^{(2\sigma-1)(s-s_1)/2}e^{\sigma s_1}
\\[2mm]
\hspace{-2mm}&=&\hspace{-2mm} \dis
 e^{(2\sigma-1)(s-s_1)/2}e^{-(\varrho -\sigma)s_1}e^{\varrho s_1}
\hspace{7.5mm}
\text{for } |y|<e^{\sigma s}.
\end{array}
\]
As a consequence,
since $\sigma<1/2$,
it holds that for $|y|<e^{\sigma s}$ and $|\xi|>e^{\varrho  s_1}$
\[
 |e^{-(s-s_1)/2}y-\xi| > c|\xi| > c|\xi|+ce^{\varrho  s_1}.
\]
Therefore 
we obtain from \eqref{tildeGamma-eq} and Lemma \ref{phi-initial-lem}
\[
\begin{array}{lll}
\dis
 \int_{e^{\varrho s_1}}^\infty
 \Gamma(y,\xi,s-s_1)P_\ell b_1(\xi){\cal B}(\xi)d\xi
% \\[2mm] \dis \hspace{10mm}
\hspace{-2mm}&\leq&\hspace{-2mm} \dis
 ce^{(\gamma-m)(s-s_1)/2}
 \exp\left( -ce^{2\varrho s_1} \right)
 \int_{e^{\varrho s_1}}^\infty
 |\xi|^{-m+\gamma}e^{-c|\xi|^2}
 d\xi
\\[4mm]
\hspace{-2mm}&\leq&\hspace{-2mm} \dis
 ce^{(\gamma-m)(s-s_1)/2}
 \exp\left( -ce^{2\varrho s_1} \right).
\end{array}
\]
Thus
it follows from \eqref{|y|exp-eq} that
\[
 \int_{e^{\varrho s_1}}^\infty
 \Gamma(y,\xi,s-s_1)P_\ell b_1(\xi){\cal B}(\xi)d\xi
\leq
 ce^{\lambda^*s_1}\exp\left( -ce^{2\varrho s_1} \right)
 e^{-\lambda^*s}|y|^{2\lambda^*-m+\gamma}.
\]
Since $H=e^{a's_1}$ with $0<a'\ll1$,
combining the above estimates,
we find that there exists $\delta>0$ such that
\[
 \left| e^{{\cal A}(s-s_1)}P_\ell b_1 \right|
\leq
 c\left(
 e^{-\delta s_1}+\nu(s_1)+e^{\lambda^*s_1}d_{\max}
 \right)
 e^{-\lambda^*s}|y|^{2\lambda^*-m+\gamma}
\]
for $(y,s)\in O_{\text{Long}}^{(\text{III})}$,
which completes the proof.
\end{proof}
%%%%%%%%%%%%%%%%%%%%%%%%%%%%%%%%%%%%%%%%%%%%%%%%%%%%%%%%%%%

Next we provide the estimate of $\int_{s_1}^se^{{\cal B}(s-\tau)}F(\tau)d\tau$.
We divide this integral into two parts.
\begin{equation}\label{FLong-eq}
 \int_{s_1}^{s-1}
 e^{{\cal B}(s-\tau)}F(\tau)d\tau
=
 \left(
 \int_{s_1}^{s-1}
 +
 \int_{s-1}^s
 \right)
 e^{{\cal B}(s-\tau)}F(\tau)d\tau.
\end{equation}
From Lemma \ref{ShortS_2-lem},
the second integral of \eqref{FLong-eq} is estimated by 
\begin{equation}\label{LongIIIF-eq}
 \left|
 \int_{s-1}^s
 e^{{\cal B}(s-\tau)}F(\tau)d\tau
 \right|
\leq
 ce^{-\delta s_1}e^{-\lambda^*s}
 \left( 1+|y|^{2\lambda^*-m+\gamma} \right).
\end{equation}
Now
we estimate the first integral of \eqref{FLong-eq}.
Here
we consider a more general from.

%%%%%%%%%%%%%%%%%%%%%%%%%%%%%%%%%%%%%%%%%%%%%%%%%%%%%%%%%%%
\begin{lem}\label{LongIII2-lem}
There exist $\delta>0$ and $c>0$ such that
if $|y|\geq e^{(s-\mu_1)/2}$ and $s_1\leq\mu_1<s-1$,
then it holds that
\[
 \left|
 \int_{\mu_1}^{s-1}
 e^{{\cal B}(s-\tau)}F(\tau)d\tau
 \right|
\leq
 ce^{-\delta s_1}e^{-\lambda^*s}
 \left( 1+|y|^{2\lambda^*-m+\gamma} \right)
\hspace{5mm}\mathrm{for}\
 (y,s)\in O_{\mathrm{Long}}^{(\mathrm{III})}.
\]
\end{lem}
%%%%%%%%%%%%%%%%%%%%%%%%%%%%%%%%%%%%%%%%%%%%%%%%%%%%%%%%%%%

%%%%%%%%%%%%%%%%%%%%%%%%%%%%%%%%%%%%%%%%%%%%%%%%%%%%%%%%%%%
\begin{proof}
First
we divide the integral into four parts.
\[
\begin{array}{l}
\dis
 \int_{\mu_1}^{s-1}
 e^{{\cal B}(s-\tau)}F(\tau)d\tau
=
 \int_{\mu_1}^{s-1}d\tau
 \int_{\pa\R_+^n}
 \Gamma(y,\xi',s-\tau)F(\xi',\tau){\cal B}(\xi')d\xi'
\\[4mm] \dis \hspace{5mm}
=
 \int_{\mu_1}^{s-1}d\tau
 \left(
 \int_0^{He^{-\omega \tau}}
 +
 \int_{He^{-\omega \tau}}^1
 +
 \int_1^{e^{\sigma\tau}}
 +
 \int_{e^{\sigma\tau}}^\infty
 \right)
 \Gamma(y,\xi',s-\tau)F(\xi',\tau){\cal B}(\xi')d\xi'
\\[6mm] \dis \hspace{5mm}
=:
 I_1+I_2+I_3+I_4.
\end{array}
\]
Since $s-\tau>1$ for $\tau\in(\mu_1,s-1)$,
Proposition \ref{Rep-pro} implies 
\[
 \Gamma(y,\xi,s-\tau)
\leq
 ce^{(\gamma-m)(s-\tau)/2}(|\xi|+1)^{2\gamma}
 \exp\left( -c\left| e^{-(s-\tau)/2}y-\xi \right|^2 \right)
\]
for $\mu_1<\tau<s-1$.
Since $H<K$,
we note from Lemma \ref{f(Phi)-lem} that
\begin{equation}\label{FH-eq}
 |F(\xi',\tau)|
=
 |f(\Phi(\xi',\tau))/\sigma(y)|
\leq
c\begin{cases}
 \dis
 |y|^{-(m+1)+\gamma} & \text{for } |y|<He^{-\omega s},
 \\[1mm] \dis
 e^{-2\lambda^*s}|y|^{-1+m-\gamma} & \text{for } He^{-\omega s}<|y|<1.
\end{cases}
\end{equation}
Hence
this implies
\[
\begin{array}{lll}
\dis
 I_1
\hspace{-2mm}&\leq&\hspace{-2mm} \dis
 c\int_{\mu_1}^{s-1}e^{(\gamma-m)(s-\tau)/2}
 d\tau
 \int_0^{He^{-\omega \tau}}
 |\xi'|^{-m-1-\gamma}d\xi'
\\[5mm]
\hspace{-2mm}&=&\hspace{-2mm} \dis
 c\int_{\mu_1}^se^{(\gamma-m)(s-\tau)/2}
 (He^{-\omega \tau})^{-m-\gamma+n-2}
 d\tau
\\[5mm]
\hspace{-2mm}&\leq&\hspace{-2mm} \dis
 cH^{n-2-(m+\gamma)}
 e^{-(n-2-2\gamma)\omega\mu_1}
 e^{(\gamma-m)(s-\mu_1)/2}e^{-\lambda^*\mu_1}.
\end{array}
\]
Here
since $|y|>e^{(s-\mu_1)/2}$,
we note that
\begin{equation}\label{(s-mu_1)-eq}
 e^{(\gamma-m)(s-\mu_1)/2}\cdot e^{\lambda^*(s-\mu_1)}\leq |y|^{\gamma-m+2\lambda^*}.
\end{equation}
Therefore
it follows that
\[
 I_1
\leq
 cH^{n-2-(m+\gamma)}e^{-(n-2-2\gamma)\omega\mu_1}
 e^{-\lambda^*s}|y|^{2\lambda^*-m+\gamma}.
\]
Next
we estimate $I_2$.
Then
from \eqref{FH-eq} and \eqref{(s-mu_1)-eq},
we see that
\[
\begin{array}{lll}
\dis
 I_2
\hspace{-2mm}&\leq&\hspace{-2mm} \dis
 c\int_{\mu_1}^{s-1}
 e^{(\gamma-m)(s-\tau)/2}e^{-2\lambda^*\tau}
 d\tau
 \int_{He^{-\omega\tau}}^1
 |\xi'|^{m-1-3\gamma}
 d\xi'
\\[4mm]
\hspace{-2mm}&\leq&\hspace{-2mm} \dis
 c\int_{\mu_1}^s
 e^{(\gamma-m)(s-\tau)/2}e^{-2\lambda^*\tau}
 (He^{-\omega \tau})^{-(\gamma-m)}
 d\tau
 \int_{Le^{-\omega\tau}}^1
 |\xi'|^{-1-2\gamma}d\xi'
\\[5mm]
\hspace{-2mm}&\leq&\hspace{-2mm} \dis
 cH^{-(\gamma-m)}
 e^{(\gamma-m)(s-\mu_1)/2}e^{-\lambda^*\mu_1}
\\[5mm]
\hspace{-2mm}&\leq&\hspace{-2mm} \dis
 cH^{-(\gamma-m)}e^{-\lambda^*s}
 |y|^{2\lambda^*-m+\gamma}.
\end{array}
\]
Furthermore
we get from Lemma \ref{f(Phi)-lem}
\[
 I_3
\leq
 c\int_{\mu_1}^{s-1}
 e^{(\gamma-m)(s-\tau)/2}e^{-2\lambda^*\tau}
 d\tau
 \int_1^{e^{\sigma\tau}}
 |\xi'|^{-m-1+4\lambda^*+\gamma}
 \exp\left( -c\left| e^{-(s-\tau)/2}y-\xi' \right|^2 \right)
 d\xi'.
\]
Here
we put
\[
\begin{array}{c}
\dis
 D_1 = \left\{ 1<|\xi'|<e^{\sigma\tau};|e^{-(s-\tau)/2}y-\xi'|<|\xi'|/2 \right\},
\\[2mm] \dis
 D_2 = \left\{ 1<|\xi'|<e^{\sigma\tau};|e^{-(s-\tau)/2}y-\xi'|>|\xi'|/2 \right\}.
\end{array}
\]
Then
since $e^{-(s-\tau)/2}|y|>|\xi'|/2$ for $\xi'\in D_1$,
we observe that
\[
\begin{array}{l}
\dis
 \int_{D_1}
 |\xi'|^{-m-1+4\lambda^*+\gamma}
 \exp\left( -c\left| e^{-(s-\tau)/2}y-\xi' \right|^2 \right)
 d\xi'
\\[4mm] \dis \hspace{10mm}
\leq
 c\left( e^{-(s-\tau)/2}|y| \right)^{4\lambda^*-m+\gamma}
 \int_{D_1}|\xi'|^{-1}
 \exp\left( -c\left| e^{-(s-\tau)/2}y-\xi' \right|^2 \right)
 d\xi'
\\[4mm] \dis \hspace{10mm}
\leq
 ce^{-(\gamma-m)(s-\tau)/2}e^{-2\lambda^*(s-\tau)}|y|^{4\lambda^*-m+\gamma}
 \int_{D_1}
 \exp\left( -c\left| e^{-(s-\tau)/2}y-\xi' \right|^2 \right)
 d\xi'
\\[4mm] \dis \hspace{10mm}
\leq
 ce^{-(\gamma-m)(s-\tau)/2}e^{-2\lambda^*(s-\tau)}|y|^{4\lambda^*-m+\gamma}.
\end{array}
\]
Therefore
since  $|y|<e^{\sigma s}$,
it follows that
\[
\begin{array}{l}
\dis
 \int_{\mu_1}^{s-1}
 e^{(\gamma-m)(s-\tau)/2}e^{-2\lambda^*\tau}
 d\tau
 \int_{D_1}
 |\xi'|^{-m-1+4\lambda^*+\gamma}
 \exp\left( -c\left| e^{-(s-\tau)/2}y-\xi' \right|^2 \right)
 d\xi'
 \\[6mm] \dis \hspace{10mm}
\leq
 cse^{-2\lambda^*s}
 |y|^{4\lambda^*-m+\gamma}
\leq
 cse^{-(1-2\sigma)\lambda^*s}e^{-\lambda^*s}
 |y|^{2\lambda^*-m+\gamma}.
\end{array}
\]
Furthermore
by definition of $D_2$,
we see that
\[
\begin{array}{l}
\dis
 \int_{D_2}
 |\xi'|^{-m-1+4\lambda^*+\gamma}
 \exp\left( -c\left| e^{-(s-\tau)/2}y-\xi' \right|^2 \right)
 d\xi'
\\[2mm] \dis \hspace{15mm}
\leq
 \int_{D_2}
 |\xi'|^{-m-1+4\lambda^*+\gamma}e^{-c|\xi'|^2}
 d\xi'
\leq
 c.
\end{array}
\]
As a consequence,
we obtain from \eqref{(s-mu_1)-eq}
\[
\begin{array}{l}
\dis
 \int_{\mu_1}^{s-1}
 e^{(\gamma-m)(s-\tau)/2}e^{-2\lambda^*\tau}
 d\tau
 \int_{D_2}
 |\xi'|^{-m-1+4\lambda^*+\gamma}
 \exp\left( -c\left| e^{-(s-\tau)/2}y-\xi' \right|^2 \right)
 d\xi'
 \\[6mm] \dis \hspace{10mm}
\leq
 ce^{(\gamma-m)(s-\mu_1)/2}e^{-2\lambda^*\mu_1}
\leq
 ce^{-\lambda^*\mu_1}e^{-\lambda^*s}
 |y|^{2\lambda^*-m+\gamma}.
\end{array}
\]
Thus
we conclude
\[
 I_3
\leq
 c\left( se^{-(1-2\sigma)\lambda^*s}+e^{-\lambda^*\mu_1} \right)
 e^{-\lambda^*s}|y|^{2\lambda^*-m+\gamma}.
\]
Finally
from Lemma \ref{f(Phi)-lem},
we see that
\[
\begin{array}{lll}
\dis
 I_4
\hspace{-2mm}&\leq&\hspace{-2mm} \dis
 \int_{\mu_1}^{s-1}e^{(\gamma-m)(s-\tau)/2}d\tau
 \int_{e^{\sigma \tau}}^\infty
 |\xi'|^{-m-1+\gamma}
 \exp\left( -c\left| e^{-(s-\tau)/2}y-\xi' \right|^2 \right)
 d\xi'
\\[6mm]
\hspace{-2mm}&=&\hspace{-2mm} \dis
 \int_{\mu_1}^{s-1}e^{(\gamma-m)(s-\tau)/2}d\tau
 \left( \int_{D_1}+\int_{D_2} \right)
 |\xi'|^{-m-1+\gamma}
 \exp\left( -c\left| e^{-(s-\tau)/2}y-\xi' \right|^2 \right)
 d\xi',
\end{array}
\]
where $D_1$ and $D_2$ are given by
\[
\begin{array}{c}
\dis
 D_1=\{|\xi'|>e^{\sigma\tau};|e^{-(s-\tau)/2}y-\xi'|<|\xi'|/2\},
\\[2mm] \dis
 D_2=\{|\xi'|>e^{\sigma\tau};|e^{-(s-\tau)/2}y-\xi'|>|\xi'|/2\}.
\end{array}
\]
Since $e^{-(s-\tau)/2}|y|>|\xi'|/2$ for $\xi'\in D_1$,
we get
\[
\begin{array}{l}
\dis
 \int_{D_1}
 |\xi'|^{-m-1+\gamma}
 \exp\left( -c\left| e^{-(s-\tau)/2}y-\xi' \right|^2 \right)
 d\xi'
\\[4mm] \dis \hspace{5mm}
\leq
 \left( e^{-(s-\tau)/2}|y| \right)^{2\lambda^*-m+\gamma}
 \int_{D_1}|\xi'|^{-2\lambda^*-1}
 \exp\left( -c\left| e^{-(s-\tau)/2}y-\xi' \right|^2 \right)
 d\xi'
\\[4mm] \dis \hspace{5mm}
\leq
 ce^{-(\gamma-m)(s-\tau)/2}e^{-\lambda^*(s-\tau)}e^{-(2\lambda^*+1)\sigma\tau}
 |y|^{2\lambda^*-m+\gamma}
 \int_{D_1}
 \exp\left( -c\left| e^{-(s-\tau)/2}y-\xi' \right|^2 \right)
 d\xi'.
\end{array}
\]
Therefore
since $(2\lambda^*+1)\sigma>\lambda^*$,
it follows that
\[
\begin{array}{l}
\dis
 \int_{\mu_1}^{s-1}e^{(\gamma-m)(s-\tau)/2}d\tau
 \int_{D_1}
 |\xi'|^{-m-1+\gamma}
 \exp\left( -c\left| e^{-(s-\tau)/2}y-\xi' \right|^2 \right)
 d\xi'
\\[4mm] \dis \hspace{25mm}
\leq
 ce^{-\lambda^*s}
 \left(
 \int_{\mu_1}^{s-1}
 e^{(\lambda^*-(2\lambda^*+1)\sigma)\tau}\tau
 d\tau
 \right)
 |y|^{2\lambda^*-m+\gamma}
\\[4mm] \dis \hspace{25mm}
\leq
 ce^{(\lambda^*-(2\lambda^*+1)\sigma)\mu_1}e^{-\lambda^*s}|y|^{2\lambda^*-m+\gamma}.
\end{array}
\]
Furthermore
by definition of $D_2$,
we have
\[
\begin{array}{lll}
\dis
 \int_{D_2}
 |\xi'|^{-m-1+\gamma}
 \exp\left( -c\left| e^{-(s-\tau)/2}y-\xi' \right|^2 \right)
 d\xi'
\hspace{-2mm}&\leq&\hspace{-2mm} \dis
 \int_{D_2}|\xi'|^{-m-1+\gamma}e^{-c|\xi'|^2}
 d\xi'
\\[6mm]
\hspace{-2mm}&\leq&\hspace{-2mm} \dis
 \int_{e^{\sigma \tau}}^\infty
 r^{-m+\gamma+n-3}e^{-cr^2}
 dr.
\end{array}
\]
Here
we note that $r^{-m+\gamma+n-3}e^{-cr^2}<re^{-c'r^2}$ for $r>1$.
Therefore
it holds that
\[
 \int_{D_2}
 |\xi'|^{-m-1+\gamma}
 \exp\left( -c\left| e^{-(s-\tau)/2}y-\xi' \right|^2 \right)
 d\xi'
\leq
 c\exp \left( -c'e^{2\sigma \tau} \right).
\]
As a consequence,
we obtain
\[
\begin{array}{l}
\dis
 \int_{\mu_1}^{s-1}e^{(\gamma-m)(s-\tau)/2}d\tau
 \int_{D_2}
 |\xi'|^{-m-1+\gamma}
 \exp\left( -c\left| e^{-(s-\tau)/2}y-\xi' \right|^2 \right)
 d\xi'
\\[4mm] \dis \hspace{10mm}
\leq
 c\int_{\mu_1}^se^{(\gamma-m)(s-\tau)/2}\exp\left( -c'e^{2\sigma\tau} \right)d\tau
% \\[4mm] \dis \hspace{10mm}
\leq
 ce^{(\gamma-m)(s-\mu_1)/2}\exp\left( -c'e^{2\sigma \mu_1} \right).
\end{array}
\]
Therefore
\eqref{(s-mu_1)-eq} implies
\[
\begin{array}{l}
\dis
 \int_{\mu_1}^{s-1}e^{(\gamma-m)(s-\tau)/2}d\tau
 \int_{D_2}
 |\xi'|^{-m-1+\gamma}
 \exp\left( -c\left| e^{-(s-\tau)/2}y-\xi' \right|^2 \right)
 d\xi'
\\[4mm] \dis \hspace{20mm}
\leq
 ce^{\lambda^*\mu_1}\exp\left( -c'e^{2\sigma\mu_1} \right)e^{-\lambda^*s}
 |y|^{2\lambda^*-m+\gamma}.
\end{array}
\]
Thus
we conclude
\[
 I_4
\leq
 c\left(
 e^{(\lambda^*-(2\lambda^*+1)\sigma)\mu_1}+e^{\lambda^*\mu_1}\exp\left( -c'e^{2\sigma \mu_1} \right)
 \right)
 e^{-\lambda^*s}|y|^{2\lambda^*-m+\gamma}.
\]
Since $(2\lambda^*+1)\sigma>\lambda^*$,
the proof is completed.
\end{proof}
%%%%%%%%%%%%%%%%%%%%%%%%%%%%%%%%%%%%%%%%%%%%%%%%%%%%%%%%%%%

Combining Lemma \ref{LongIII-lem}\hspace{0.5mm}-\hspace{0.5mm}Lemma \ref{LongIII2-lem} and \eqref{LongIIIF-eq},
we obtain the following result.

%%%%%%%%%%%%%%%%%%%%%%%%%%%%%%%%%%%%%%%%%%%%%%%%%%%%%%%%%%%
\begin{pro}\label{Longtime3-pro}
There exist $\delta>0$ and $c>0$ such that
if $|d|<\epsilon_1e^{-\lambda^*s_1}$ and $\varphi(y,s)\in A_{s_1,s_2}$,
then it holds that
\[
 \left| b(y,s)+e^{-\lambda^*s}\eta_{1\ell}(y) \right|
<
 ce^{-\delta s_1}e^{-\lambda^*s}
 \left( 1+|y|^{2\lambda^*-m+\gamma} \right)
\hspace{5mm}\mathrm{for}\ (y,s)\in O_{\mathrm{Long}}^{(\mathrm{III})}.
\]
\end{pro}
%%%%%%%%%%%%%%%%%%%%%%%%%%%%%%%%%%%%%%%%%%%%%%%%%%%%%%%%%%%

%%%%%%%%%%%%%%%%%%%%%%%%%%%%%%%%%%%%%%%%%%%%%%%%%%%%%%%%%%%
%%%%%%%%%%%%%%%%%%%%%%%%%%%%%%%%%%%%%%%%%%%%%%%%%%%%%%%%%%%
\section{Exterior domain estimates}\label{Exterior-sec}
%%%%%%%%%%%%%%%%%%%%%%%%%%%%%%%%%%%%%%%%%%%%%%%%%%%%%%%%%%%
%%%%%%%%%%%%%%%%%%%%%%%%%%%%%%%%%%%%%%%%%%%%%%%%%%%%%%%%%%%

In this section,
we provide the estimate in
\[
 O_{\text{Ext}} = \left\{ (y,s);\ r>e^{\sigma s},\ s_1<s<s_2 \right\}.
\]
Here
we define $m_0>1$ appearing in \eqref{Assume1-eq} as follows.
Let $c_H$ be the best constant given in Lemma \ref{HHardy-lem}.
Then
by definition of a JL-supercritical exponent,
we recall that $U_\infty^{q-1}|_{\pa\R_+^n}<qU_\infty^{q-1}|_{\pa\R_+^n}<c_H$.
Then
we define $m_0=m_0(q,n)$ by
\begin{equation}\label{m_0-eq}
 m_0^{q-1} = c_H/U_\infty^{q-1}|_{\pa\R_+^n}>1.
\end{equation}
Furthermore
let $e_H(\theta)$ be the first eigenfunction of
\[
\begin{cases}
 -\Delta_Se = \lambda e & \text{in } (0,\pi/2),
 \\
 \pa_\nu e = c_He & \text{on } \theta=\pi/2.
\end{cases}
\]
Then
it is known that the first eigenvalue is $-(n-2)^2/4$
(see Lemma \ref{Heigenvalue-lem}).
By using this fact,
we see that $L(x)=e_H(\theta)r^{-(n-2)/2}$ gives a solution of
\[
\begin{cases}
 \dis
 -\Delta L = 0 & \text{in } \R_+^n,
 \\ \dis
 \pa_\nu L = c_Hr^{-1}L & \text{on } \pa\R_+^n.
\end{cases}
\]
Since $m<(n-2)/2$,
we can choose $r_1>0$ such that
\begin{equation}\label{r_1Def-eq}
 L(x) < \left( \frac{1+m_0}{2} \right)U_\infty(x)
\hspace{5mm}\text{for } |x|>r_1.
\end{equation}
Furthermore
from $m_0>1$,
we can fix $d_0\in(0,1)$ such that
\begin{equation}\label{d_0Def-eq}
 \left( \frac{1+m_0}{2} \right)U_\infty(x) + d_0 < \left( \frac{1+3m_0}{4} \right)U_\infty(x)
\hspace{5mm}\text{for } |x|<r_1.
\end{equation}
Now
we divide $O_{\text{Ext}}$ into two parts.
\[
\begin{array}{c}
\dis
 O_{\text{Ext}}^{(\text{I})}
=
 \left\{ (y,s);\ e^{\sigma s}<r<r_1e^{s/2},\ s_1<s<s_2 \right\},
\\[3mm]
 O_{\text{Ext}}^{(\text{II})}
=
 \left\{ (y,s);\ r>r_1e^{s/2},\ s_1<s<s_2 \right\}.
\end{array}
\]
It is clear that
$O_{\text{Ext}}=O_{\text{Ext}}^{(\text{I})}\cup O_{\text{Ext}}^{(\text{II})}$.
First
we provide the estimate in $O_{\text{Ext}}^{(\text{I})}$.

%%%%%%%%%%%%%%%%%%%%%%%%%%%%%%%%%%%%%%%%%%%%%%%%%%%%%%%%%%%
\begin{pro}\label{Exterior1-pro}
There exists $s_0>0$ such that if $s_1>s_0$, $|d|<\epsilon_1e^{-\lambda^*s_1}$ and
$\varphi(y,s)\in A_{s_1,s_2}$,
then it holds that
\[
 \varphi(y,s) \leq \left( \frac{1+3m_0}{4} \right)U_\infty(y)
\hspace{5mm}\mathrm{for}\ (y,s)\in O_{\mathrm{Ext}}^{(\mathrm{I})}.
\]
\end{pro}
%%%%%%%%%%%%%%%%%%%%%%%%%%%%%%%%%%%%%%%%%%%%%%%%%%%%%%%%%%%

%%%%%%%%%%%%%%%%%%%%%%%%%%%%%%%%%%%%%%%%%%%%%%%%%%%%%%%%%%%
\begin{proof}
Throughout this proof,
we assume $\varphi(y,s)\in A_{s_1,s_2}$.
Here we put $T=e^{-s_1}$, $T'=e^{-s_1}-e^{-s_2}$ and
\[
\hspace{7.5mm}
 u(x,t) = (T-t)^{-m/2}\varphi((T-t)^{-1/2}x,-\log(T-t)).
\]
Here
we easily see that $0<T'<T$ and $T'\to T$ as $s_2\to\infty$.
Then since $\varphi(y,s)\in A_{s_1,s_2}$,
$u(x,t)$ is defined on $\R_+^n\times(0,T')$ and
satisfies \eqref{u-eq} with $u_0(x)=T^{-m/2}\varphi(T^{-1/2}x,s_1)$.
For simplicity,
we put
\[
 {\cal O}_{T'} = \left\{ (x,t);\ |x|>(T-t)^{(1-2\sigma)/2},\ 0<t<T' \right\}.
\]
Since $0<\varphi(y,s)<m_0U_\infty(y)$ for $|y|>e^{\sigma s}$,
it holds that
\[
\begin{array}{lll}
\dis
 0 < u(x,t)
\hspace{-2mm}&\leq&\hspace{-2mm} \dis
 m_0(T-t)^{-m/2}U_\infty\left( (T-t)^{-1/2}x \right)
\leq
 c|x|^{-m}
\\[3mm]
\hspace{-2mm}&\leq&\hspace{-2mm} \dis
 c(T-t)^{-(1-2\sigma)m/2}
\hspace{5mm}\text{for } (x,t)\in{\cal O}_{T'}.
\end{array}
\]
Therefore
there exists $\kappa>0$ such that
the boundary condition on $\pa\R_+^n$ is given by
\begin{equation}\label{u(x,t)boundary-eq}
 \pa_\nu u = u^q \leq  \kappa(T-t)^{-mq(1-2\sigma)/2}
\hspace{5mm}\text{for } (x,t)\in{\cal O}_{T'}.
\end{equation}
Furthermore
since $|\varphi(y,s)-U_\infty(y)|\leq ce^{-\lambda^*s}|y|^{2\lambda^*-m}$ for $|y|=e^{\sigma s}$,
we observe that
\[
\begin{array}{lll}
\dis
 u(x,t)
\hspace{-2mm}&\leq&\hspace{-2mm} \dis
 (T-t)^{-m/2}
 \left(
 U_\infty\left( (T-t)^{-1/2}x \right)+c(T-t)^{\lambda^*}\left( (T-t)^{-1/2}|x| \right)^{2\lambda^*-m}
 \right)
\\[4mm]
\hspace{-2mm}&\leq&\hspace{-2mm} \dis
 U_\infty(x)+c|x|^{2\lambda^*-m}
\hspace{5mm}\text{for } |x|=(T-t)^{(1-2\sigma)/2},\ 0<t<T'.
\end{array}
\]
Therefore
since $m_0>1$,
there exists $T_1>0$ such that if $T<T_1$,
then it holds that
$2c|x|^{2\lambda^*-m}<(m_0-1)U_\infty(x)$ on $|x|=(T-t)^{(1-2\sigma)/2}$ for $0<t<T$.
As a consequence,
we get if $T<T_1$
\begin{equation}\label{u(x,t)boundary|-eq}
 u(x,t) < \left( \frac{1+m_0}{2} \right)U_\infty(x)
\hspace{5mm}\text{for } |x|=(T-t)^{(1-2\sigma)/2},\ 0<t<T'.
\end{equation}
Since $\sigma<\varrho$,
by definition of $\phi_\ell^*$,
we easily see that
\[
 \varphi(y,s_1)
\leq
 \begin{cases}
 \dis
 U_\infty(y) + \sum_{(i,j)\in\Pi}|d_{ij}||\phi_{ij}(y)| + e^{-\lambda^*s_1}|\phi_{1\ell}(y)|
 & \text{for } e^{\sigma s_1}<|y|<e^{\varrho s_1}+1,
 \\ \dis
 0 & \text{for } |y|>e^{\varrho s_1}+1.
 \end{cases}
\]
Since
$|d_{ij}|<\epsilon_1e^{-\lambda^*s_1}$ and
$|\phi_{ij}(y)|\leq c_{ij}|y|^{2\lambda_{ij}-m}$ for $|y|>1$,
it holds that
\[
\begin{array}{lll}
\dis
 \varphi(y,s_1)
\hspace{-2mm}&\leq&\hspace{-2mm} \dis
 U_\infty(y) + ce^{-\lambda^*s_1}|y|^{2\lambda^*-m}
\leq
 U_\infty(y) + ce^{-(1-2\varrho)\lambda^*s_1}|y|^{-m}
\\[2mm]
\hspace{-2mm}&\leq&\hspace{-2mm} \dis
 \left( 1+ce^{-(1-2\varrho)\lambda^*s_1} \right)U_\infty(y)
\hspace{7.5mm} \text{for } e^{\sigma s_1}<|y|<e^{\varrho s_1}+1.
\end{array}
\]
Therefore
by using $u_0(x)=T^{-m/2}\varphi(T^{-1/2}x,-\log T)$,
we see that
\begin{equation}\label{u_0T-eq}
 u_0(x)
\leq
\begin{cases}
 \left( 1+cT^{(1-2\varrho)\lambda^*} \right)U_\infty(x)
 & \text{for } T^{1/2-\sigma}<|x|<T^{1/2-\varrho}+T^{1/2},
 \\[1mm] \dis
 0 & \text{for } |x|>T^{1/2-\varrho}+T^{1/2}.
\end{cases}
\end{equation}
As a consequence,
there exists $T_2>0$ such that
if $T<T_2$,
then it holds that
\begin{equation}\label{u_0(x)|-eq}
 u_0(x) \leq \left( \frac{1+m_0}{2} \right)U_\infty(x)
\hspace{5mm}\text{for } |x|>T^{(1-2\sigma)/2}.
\end{equation}
Now
we construct a super-solution.
For simplicity,
we put $p=mq(1-2\sigma)$.
Let $g(x,t)$ be a solution of
\begin{equation}\label{g(x,t)-eq}
\begin{cases}
 g_t = \Delta g, & (x,t)\in\R_+^n\times(0,T),
\\[1mm]\dis
 \pa_\nu g = \left( \frac{2\kappa}{d_0} \right)(T-t)^{-p/2}g,
& (x,t)\in\pa\R_+^n\times(0,T),
 \\
 g(x,0) \equiv d_0/2, & x\in\R_+^n.
\end{cases}
\end{equation}
Here
we put
\begin{equation}\label{baru(x,t)Def-eq}
 \bar{u}(x,t) = \left( \frac{1+m_0}{2} \right)U_\infty(x) + g(x,t).
\end{equation}
Then
we easily check that $\bar{u}(x,t)$ satisfies
\[
\begin{cases}
\dis 
 \bar{u}_t = \Delta \bar{u}, & (x,t)\in\R_+^n\times(0,T),
\\ \dis
 \pa_\nu \bar{u} \geq \left( \frac{2\kappa}{d_0} \right)(T-t)^{-p/2}g,
& (x,t)\in\pa\R_+^n\times(0,T),
\\ \dis
 \bar{u}(x,0) \equiv \bar{u}_0(x):= \left( \frac{1+m_0}{2} \right)U_\infty(x)+d_0/2,
& x\in\R_+^n.
\end{cases}
\]
Since $g(x,t)\geq d_0/2$ (see Lemma \ref{g(x,t)-lem}),
we find
\begin{equation}\label{baru(x,t)boundary-eq}
 \pa_\nu \bar{u} \geq \kappa(T-t)^{-p/2}
\hspace{5mm}\text{for } (x,t)\in\pa\R_+^n\times(0,T).
\end{equation}
We put
$U(x,t)=\bar{u}(x,t)-u(x,t)$.
Then
from \eqref{u(x,t)boundary-eq}-\eqref{u_0(x)|-eq} and \eqref{baru(x,t)Def-eq}-\eqref{baru(x,t)boundary-eq},
we see that if $T<\min\{T_1,T_2\}$
\[
\begin{array}{cl}
\dis
 \pa_\nu U \geq 0 & \text{for } (x,t)\in{\cal O}_{T'}, x\in\pa\R_+^n,
\\[2mm] \dis
 U > 0 & \text{for } |x|=(T-t)^{(1-2\sigma)/2},\ 0<t<T',
\\[2mm] \dis
 U > 0 & \text{for } |x|>(T-t)^{(1-2\sigma)/2},\ t=0.
\end{array}
\]
Therefore
applying a comparison lemma in ${\cal O}_{T'}$,
we obtain if $T<\min\{T_1,T_2\}$
\[
 u(x,t) \leq \bar{u}(x,t)
\hspace{5mm}\text{for } (x,t)\in {\cal O}_{T'}.
\]
Since $\sigma>1/2q$,
we easily see that $p<1$.
Therefore
Lemma \ref{g(x,t)-lem} implies
$\sup_{(x,t)\in\R_+^n\times(0,T)}|g(x,t)|<d_0$ if $T<T_0$.
Therefore
if $T<\{T_0,T_1,T_2\}$,
it holds that
\[
 u(x,t) \leq \left( \frac{1+m_0}{2} \right)U_\infty(x) + d_0
\hspace{5mm}\text{for } (x,t)\in{\cal O}_{T'}.
\]
As a consequence,
by \eqref{d_0Def-eq},
if $T<\{T_0,T_1,T_2\}$,
it follows that
\[
 u(x,t) \leq \left( \frac{1+3m_0}{4} \right)U_\infty(x)
\hspace{5mm}\text{for } (T-t)^{(1-2\sigma)/2}<|x|<r_1,\ 0<t<T'.
\]
Since $\varphi(y,s)=(T-t)^{m/2}u((T-t)^{1/2}y,t)$ with $T-t=e^{-s}$,
we obtain the conclusion.
\end{proof}
%%%%%%%%%%%%%%%%%%%%%%%%%%%%%%%%%%%%%%%%%%%%%%%%%%%%%%%%%%%

%%%%%%%%%%%%%%%%%%%%%%%%%%%%%%%%%%%%%%%%%%%%%%%%%%%%%%%%%%%
\begin{lem}\label{g(x,t)-lem}
Let $p<1$ and $g(x,t)$ be a solution of \eqref{g(x,t)-eq}.
Then
there exists $T_0>0$ such that
if $T<T_0$,
then it holds that
\[
 \sup_{(x,t)\in\R_+^n\times(0,T)}|g(x,t)|<d_0.
\]
Furthermore
it holds that $g(x,t)>d_0/2$ for $(x,t)\in\R_+^n\times(0,T)$.
\end{lem}
%%%%%%%%%%%%%%%%%%%%%%%%%%%%%%%%%%%%%%%%%%%%%%%%%%%%%%%%%%%

%%%%%%%%%%%%%%%%%%%%%%%%%%%%%%%%%%%%%%%%%%%%%%%%%%%%%%%%%%%
\begin{proof}
Let $G_N(x,\xi,t)$ be the heat kernel on $\R_+^n$ with Neumann boundary condition.
Then $g(x,t)$ is expressed by
\[
 g(t) =
 \int_{\R_+^n}G(x,\xi,t)\left( \frac{d_0}{2} \right)d\xi
+
 \left( \frac{2\kappa}{d_0} \right)
 \int_0^td\tau\int_{\pa\R_+^n}G_N(x,\xi',t-\tau)
 \frac{g(\xi',\tau)}{(T-\tau)^{p/2}}d\xi'.
\]
Let $g_\infty(t)=\sup_{(x,\tau)\in\R_+^n\times(0,t)}|g(x,\tau)|$.
Then
since $p<1$,
we get
\[
\begin{array}{lll}
\dis
 g_\infty(t)
\hspace{-2mm}&\leq&\hspace{-2mm} \dis
 \frac{d_0}{2} + g_\infty(t)\left( \frac{2\kappa_0}{d_0} \right)
 \int_0^t(T-\tau)^{-p/2}d\tau\int_{\pa\R_+^n}G_N(x,\xi',t-\tau)d\xi'
\\[4mm]
\hspace{-2mm}&\leq&\hspace{-2mm} \dis
 \frac{d_0}{2} + cg_\infty(t)\int_0^t(T-\tau)^{-(p+1)/2}d\tau
\\[4mm]
\hspace{-2mm}&\leq&\hspace{-2mm} \dis
 \frac{d_0}{2} + cT^{(1-p)/2}g_\infty(t).
\end{array}
\]
Therefore
since $p<1$,
the first statement is proved.
Furthermore
since $g_1(x,t)\equiv d_0/2$ is a sub-solution,
we obtain the second statement.
\end{proof}
%%%%%%%%%%%%%%%%%%%%%%%%%%%%%%%%%%%%%%%%%%%%%%%%%%%%%%%%%%%

Next
we provide the estimate in $O_{\text{Ext}}^{(\text{II})}$.

%%%%%%%%%%%%%%%%%%%%%%%%%%%%%%%%%%%%%%%%%%%%%%%%%%%%%%%%%%%
\begin{pro}\label{Exterior2-pro}
There exists $s_0>0$ such that if $s_1>s_0$, $|d|<\epsilon_1e^{-\lambda^*s_1}$ and
$\varphi(y,s)\in A_{s_1,s_2}$,
then it holds that
\[
 \varphi(y,s) \leq \left( \frac{1+m_0}{2} \right)U_\infty(y)
\hspace{5mm}\mathrm{for}\ (x,t)\in O_{\mathrm{Ext}}^{(\mathrm{II})}.
\]
\end{pro}
%%%%%%%%%%%%%%%%%%%%%%%%%%%%%%%%%%%%%%%%%%%%%%%%%%%%%%%%%%%

%%%%%%%%%%%%%%%%%%%%%%%%%%%%%%%%%%%%%%%%%%%%%%%%%%%%%%%%%%%
\begin{proof}
Throughout this proof,
we assume $\varphi(y,s)\in A_{s_1,s_2}$.
Let $u(x,t)$, $T$ and $T'$ be as in the proof of Proposition \ref{Exterior1-pro}.
From \eqref{u_0T-eq},
there exists $T_1>0$ such that
if $T<T_1$,
then it hold that
\[
 u_0(x) \leq L(x)
\hspace{5mm}\text{for } x\in\R_+^n
\]
Furthermore
since $\varphi(y,s)\in A_{s_1s_2}$,
there exists $s_0>0$ such that if $s_1>s_0$,
it holds that $\varphi(y,s)\leq m_0U_\infty(y)$
for $(y,s)\in\R_+^n\times(s_1,s_2)$.
This implies
\[
 u(x,t) \leq m_0U_\infty(x)
\hspace{5mm}\text{for } (x,t)\in\R_+^n\times(0,T').
\]
Therefore
the boundary condition on $\pa\R_+^n$ is given by
\[
 \pa_\nu u = u^q \leq (m_0U_\infty)^{q-1}u \leq c_Hr^{-1}u
\hspace{5mm}\text{for } t\in(0,T').
\]
Next
we compare the values of $u(x,t)$ and $L(x)$ on $|x|=(T-t)^{(1+\omega)/2}$ for $0<t<T'$.
Since
$\varphi(y,s)\in A_{s_1s_2}$, 
we recall that $u(x,t)<U_\infty(x)$ for $|x|=K(T-t)^{(1+\omega)/2}$ and $0<t<T'$.
Furthermore
we easily see that $L(x)>U_\infty(x)$ for $|x|=K(T-t)^{(1+\omega)/2}$ and $0<t<T$
if $T$ is small enough.
Therefore
there exists $T_2>0$ such that if $T<T_2$
\[
 u(x,t) < L(x)
 \hspace{5mm}\text{for } |x|=K(T-t)^{(1+\omega)/2},\ 0<t<T'.
\]
Thus
applying a comparison lemma in ${\cal O}_{T'}=\{(x,t);\ |x|>K(T-t)^{(1+\omega)/2},\ 0<t<T'\}$,
we obtain if $T<\min\{T_1,T_2,e^{-s_0}\}$
\[
 u(x,t) < L(x)
 \hspace{5mm}\text{for } (x,t)\in{\cal O}_{T'}.
\]
Since $r_1>K(T-t)^{(1+\omega)/2}$,
we obtain from \eqref{r_1Def-eq}
\[
 u(x,t)\leq \left( \frac{1+m_0}{2} \right)U_\infty(x)
\hspace{5mm}\text{for } |x|>r_1,
\]
which completes the proof.
\end{proof}
%%%%%%%%%%%%%%%%%%%%%%%%%%%%%%%%%%%%%%%%%%%%%%%%%%%%%%%%%%%

%%%%%%%%%%%%%%%%%%%%%%%%%%%%%%%%%%%%%%%%%%%%%%%%%%%%%%%%%%%
\appendix
\section{Appendix}
%%%%%%%%%%%%%%%%%%%%%%%%%%%%%%%%%%%%%%%%%%%%%%%%%%%%%%%%%%%

In Appendix,
we provide complete eigenfunctions and eigenvalues of
\begin{eqnarray}\label{A1-eq}
\begin{cases}
\dis
 -\left( \Delta - \frac{y}{2}\cdot\nabla - \frac{m}{2} \right)\phi = \lambda \phi 
& \text{in } \R_+^n,
\\[2mm] \dis
 \pa_{\nu}\phi = {\cal K}r^{-1}\phi
& \text{on } \pa\R_+^n.
\end{cases}
\end{eqnarray}
Here
we restrict ourselves to $y_n$-axial symmetric functions.
We again introduce the following eigenvalue problem on the unit sphere introduced in \eqref{Eigen-eq}.
\begin{equation}\label{A2-eq}
\begin{cases}
\dis
 -\Delta_S e = \kappa e
& \text{in } (0,\pi/2),
\\ \dis
 \pa_{\theta}e = {\cal K}e
& \text{on } \{\pi/2\}.
\end{cases}
\end{equation}
Let $\kappa_i$, $e_i(\theta)$ be the $i$-th eigenvalue,
the $i$-th eigenfunction with $\|e_i\|_{L^2(S_+^{n-1})}=1$.
Then
we find that
$L_S^2(S_+^{n-1})$ is spanned by $\{e_i(\theta)\}_{i\in\N}$.
Therefore
any $y_n$-axial symmetric continuous function $f(y)$ is expressed by
\[
 f(y) = \sum_{i=1}^\infty a_i(r)e_i(\theta).
\]
Plugging this expression into \eqref{A1-eq},
the eigenvalue problem is reduced to
\begin{equation}\label{A3-eq}
 -\left(
 a'' + \frac{n-1}{r}a' - \frac{\kappa_i}{r^2}a' - \frac{r}{2}a' - \frac{m}{2}a
 \right)
= \lambda a,
\hh r>0.
\end{equation}
Let $a_{ij}(r)$ and $\lambda_{ij}$ be the $j$-th eigenfunction with
$\int_0^\infty a_{ij}(r)^2e^{-r^2/4}r^{n-1}dr=1$ and the $j$-th eigenvalue of \eqref{A3-eq}.
Therefore
all eigenfunctions of \eqref{A3-eq} are expressed by
\[
 \phi_{ij}(y)=e_i(\theta)a_{ij}(r)
\]
and its eigenvalue is given by $\lambda_{ij}$.
Then
the eigenvalue $\lambda_{ij}$ is explicitly expressed in terms of $\kappa_i$.
Here
from Lemma \ref{Heigenvalue-lem},
we recall that $\kappa_1<0$ and $\kappa_i>0$ for $i\geq2$.

%%%%%%%%%%%%%%%%%%%%%%%%%%%%%%%%%%%%%%%%%%%%%%%%%%%%%%%%%%%
\begin{lem}\label{2A-lem}
Let $a_{ij}(r)$ and $\lambda_{ij}$ be the $j$-th eigenfunction and eigenvalue of \eqref{A3-eq}.
Then it holds that
\[
 a_{ij}(r) =
 \begin{cases}
 \dis
 A_{1j}r^{-\gamma}M\left( -(j-1),-\gamma+\frac{n}{2},\frac{r^2}{4} \right)
 & \text{if } i=1,
 \\[4mm] \dis
 A_{ij}r^{\gamma_i}M\left( -(j-1),\gamma_i+\frac{n}{2},\frac{r^2}{4} \right)
 & \text{if } i\geq2,
 \end{cases}
\hh
 \lambda_{ij} =
 \begin{cases}
 \dis
 -\frac{\gamma}{2}+\frac{m}{2}+j-1 & \text{if } i=1,
 \\[4mm] \dis
 \frac{\gamma_i}{2}+\frac{m}{2}+j-1 & \text{if } i\geq2,
 \end{cases}
\]
where $A_{ij}$ is a normalization constant,
$\gamma\in(0,(n-2)/2)$ is a root of
\[
 \gamma^2-(n-2)\gamma = \kappa_1
 \hh (\kappa_1<0)
\]
and $\gamma_i>0$ $(i\geq2)$ is a root of
\[
 \gamma_i^2+(n-2)\gamma_i = \kappa_i
 \hh (\kappa_i>0).
\]
Furthermore
the following asymptotic formula holds
\[
\begin{array}{ll}
\dis
 a_{ij}(r) =
 \begin{cases}
 \dis
 (c_{1j}+o(1))r^{-\gamma} & \text{if } i=1,
 \\ \dis
 (c_{ij}+o(1))r^{\gamma_i}& \text{if } i\geq2,
 \end{cases}
 \hh (r\sim0)
\\[8mm] \dis
 a_{ij}(r) =
 \begin{cases}
 \dis
 (\bar{c}_{1j}+o(1))r^{2\lambda_{1j}-m} & \text{if } i=1,
 \\ \dis
 (\bar{c}_{ij}+o(1))r^{2\lambda_{ij}-m}& \text{if } i\geq2.
 \end{cases}
 \hh (r\sim\infty)
\end{array}
\]
\end{lem}
%%%%%%%%%%%%%%%%%%%%%%%%%%%%%%%%%%%%%%%%%%%%%%%%%%%%%%%%%%%

%%%%%%%%%%%%%%%%%%%%%%%%%%%%%%%%%%%%%%%%%%%%%%%%%%%%%%%%%%%
\begin{proof}
By the same way as in the proof of Proposition 2.2 in \cite{Mizoguchi04}
(see also p.\h{p}.\h8{\h-\h}9 in \cite{Seki}),
we obtain the conclusion.
\end{proof}
%%%%%%%%%%%%%%%%%%%%%%%%%%%%%%%%%%%%%%%%%%%%%%%%%%%%%%%%%%%

%%%%%%%%%%%%%%%%%%%%%%%%%%%%%%%%%%%%%%%%%%%%%%%%%%%%%%%%%%%
\begin{lem}\label{3A-lem}
It holds that $\lambda_{ij}\geq-(\gamma-m)/2$ for any $(i,j)\in\N^2$.
\end{lem}
%%%%%%%%%%%%%%%%%%%%%%%%%%%%%%%%%%%%%%%%%%%%%%%%%%%%%%%%%%%

%%%%%%%%%%%%%%%%%%%%%%%%%%%%%%%%%%%%%%%%%%%%%%%%%%%%%%%%%%%
\begin{proof}
In Lemma \ref{2A-lem},
we take $i=j=1$.
Then we easily see that $\lambda_{11}=-(\gamma-m)/2$.
Since $\lambda_{11}$ is the first eigenvalue,
the proof is completed.
\end{proof}
%%%%%%%%%%%%%%%%%%%%%%%%%%%%%%%%%%%%%%%%%%%%%%%%%%%%%%%%%%%

% \section*{Acknowledgement}

%%%%%%%%%%%%%%%%%%%%%%%%%%%%%%%%%%%%%%%%%%%%%%%%%%%%%%%%%%%

%%%%%%%%%%%%%%%%%%%%%%%%%%%%%%%%%%%%%%%%%%%%%%%%%%%%%%%%%%%

\end{document}